\documentclass[11pt]{amsart}

\usepackage{latexsym}
\usepackage{amssymb}
\usepackage{amsmath}
\usepackage{color}

\numberwithin{equation}{section}

\newtheorem{theorem}{Theorem}[section]
\newtheorem{lemma}[theorem]{Lemma}
\newtheorem{proposition}[theorem]{Proposition}
\newtheorem{corollary}[theorem]{Corollary}
\newtheorem{definition}[theorem]{Definition}

\newtheorem{remark}[theorem]{Remark}

\newcommand\supp{\mathop{\rm supp}}

\newcommand\id{\mathop{\rm id}}

\newcommand\dd{\mathop{\rm d}}
\newcommand\mm{\mathop{\rm m}}

\newcommand{\cl}[1]{\mathcal{#1}}
\newcommand{\bb}[1]{\mathbb{#1}}

\newcommand\Tr{\mathop{\rm Tr}}
\newcommand\omin{\mathop{\rm OMIN}}
\newcommand\omax{\mathop{\rm OMAX}}

\newcommand\nph{\mathop{\varphi}}

\begin{document}

\title{Quantum no-signalling correlations 
and non-local games}




\author[I. G. Todorov]{Ivan G. Todorov}
\address{
School of Mathematical Sciences, University of Delaware, 501 Ewing Hall,
Newark, DE 19716, USA, and 
Mathematical Sciences Research Centre,
Queen's University Belfast, Belfast BT7 1NN, United Kingdom}
\email{todorov@udel.edu}
\email{i.todorov@qub.ac.uk}

\author[L. Turowska]{Lyudmila Turowska}
\address{Department of Mathematical Sciences, Chalmers University
of Technology and the University of Gothenburg, Gothenburg SE-412 96, Sweden}
\email{turowska@chalmers.se}



\begin{abstract}
We introduce and 
examine three subclasses of the family of quantum no-signalling (QNS) correlations introduced 
by Duan and Winter: quantum commuting, quantum and local. 
We formalise the notion of a universal TRO of a block operator isometry, 
define an operator system, universal for stochastic operator matrices, 
and realise it as a quotient of a matrix algebra. 
We describe the classes of QNS correlations in terms of states on the tensor products
of two copies of the universal operator system, 
and specialise the correlation classes and their representations to classical-to-quantum correlations.
We study various quantum versions of synchronous no-signalling correlations and 
show that they possess invariance properties for suitable sets of states. 
We introduce quantum non-local games as a generalisation of non-local games. 
We define the operation of quantum game composition and
show that the perfect strategies belonging to a certain class are closed under 
channel composition. 
We specialise to the case of graph colourings, where we 
exhibit quantum versions of the orthogonal rank of a graph as 
the optimal output dimension for which perfect classical-to-quantum strategies of 
the graph colouring game exist, as well as to non-commutative graph homomorphisms, 
where we identify quantum versions of non-commutative graph homomorphisms 
introduced by Stahlke. 
\end{abstract}

\date{4 September 2020}

\maketitle

\tableofcontents


\section{Introduction}\label{s_intro}

Non-local games \cite{chtw} have in the past decade acquired significant prominence,
demonstrating both the power and limitations of quantum entanglement.
These are cooperative games, played by two players, Alice and Bob, against a verifier, in each round of which 
the verifier feeds in as an input a pair $(x,y)$, selected from the cartesian product $X\times Y$ of two finite sets, 
and the players produce as an output a pair $(a,b)$ from a cartesian product $A\times B$.
The combinations $(x,y,a,b)$ that yield a win are determined by a predicate function 
$\lambda : X\times Y \times A\times B \to \{0,1\}$. 
A probabilistic strategy is a family $p = \{(p(a,b|x,y))_{(a,b)\in A\times B} : (x,y)\in X\times Y\}$
of probability distributions, one for each input pair $(x,y)$, where the value $p(a,b|x,y)$ denotes the probability that 
the players spit out the output $(a,b)$ given they have received the input $(x,y)$. 
Such families $p$ are in addition required to satisfy a no-signalling condition, which 
ensures no communication between the players takes place during the course of the game, and are hence called 
\emph{no-signalling (NS) correlations}.

In pseudo-telepathy games \cite{bbt}, no deterministic perfect (that is, winning) strategies exist, while 
shared entanglement can produce perfect \emph{quantum} strategies. 
Such strategies consist of two parts: a unit vector $\xi$ in the tensor product $H_A\otimes H_B$ of 
two finite dimensional Hilbert spaces (representing the joint physical system of the players), and 
local measurement operators $(E_{x,a})_{x,a}$ (for Alice) and $(F_{y,b})_{y,b}$ (for Bob), leading to the 
probabilities $p(a,b|x,y) = \langle (E_{x,a}\otimes F_{y,b})\xi,\xi\rangle$. 
Employing the commuting model of Quantum Mechanics leads, on the other hand, 
to the broader set of \emph{quantum commuting} strategies, whose underlying no-signalling correlations arise 
from mutually commuting measurement operators (that is, $E_{x,a}F_{y,b} = F_{y,b}E_{x,a}$)
acting on a single Hilbert space.
This viewpoint leads to the following chain of classes of no-signalling correlations:
\begin{equation}\label{eq_chain0}
\cl C_{\rm loc}\subseteq \cl C_{\rm q}\subseteq \cl C_{\rm qa}\subseteq \cl C_{\rm qc}\subseteq \cl C_{\rm ns}.
\end{equation}
The class $\cl C_{\rm qa}$ of \emph{approximately quantum} correlations is the closure of the quantum class $\cl C_{\rm q}$ 
-- known, due to the work of Slofstra \cite{slofstra} (see also \cite{dpp}) to be strictly larger than $\cl C_{\rm q}$ --
and $\cl C_{\rm ns}$ is the class of all no-signalling correlations, playing a fundamental role in 
generalised probabilistic theories \cite{alp, barrett}.
The long-standing question of whether $\cl C_{\rm qa}$ coincides with the class $\cl C_{\rm qc}$ 
of all quantum commuting correlations, 
known as Tsirelson's problem, was recently settled in the negative in \cite{jnvwy}. Due to 
the works \cite{jnpp} and \cite{oz}, this also resolved the fundamental Connes Embedding Problem \cite{Pi2}.

In this paper, we propose a quantisation of the chain of inclusions (\ref{eq_chain0}). 
Our motivation is two-fold. 
Firstly, the resolution of the Connes Embedding Problem in \cite{jnvwy}
follows complexity theory routes, and it remains of great interest if
an operator algebraic approach is within reach. The classes of correlations we introduce are 
wider and hence may offer more flexibility in looking for counterexamples. 

Our second source of motivation is the development of non-local games with quantum inputs and quantum outputs. 
A number of versions of quantum games have already been examined.
In \cite{cjppg}, the authors studied the computability and the parallel repetition behaviour of the 
entangled value of a rank one quantum game, where the players receive quantum inputs from the verifier, 
but a measurement is taken against a rank one projection to determine the likelihood of winning. 
In \cite{gw}, the focus is on multiple round quantum strategies that are
available to players with quantum memory, while the quantum-classical and extended non-local games 
considered in \cite{rw} both have classical outputs
(see also \cite{buscemi}). 
Here, we propose a framework for quantum-to-quantum non-local games, which generalises directly
(classical) non-local games. 
This allows us to define a quantum version of the graph homomorphism game
(see \cite{dpp, mrob, mrob2, pt_QJM}), 
and leads to notions of quantum homomorphisms between
(the widely studied at present \cite{bcehpsw, btw, dsw, dw, lp, stahlke}) 
non-commutative graphs.

Our starting point is the definition of quantum no-signalling correlations
given by Duan and Winter in \cite{dw}. 
Note that no-signalling (NS) correlations correspond precisely to (bipartite) classical information channels 
from $X\times Y$ to $A\times B$ with well-defined marginals.
In \cite{dw}, \emph{quantum no-signalling  (QNS) correlations} are thus defined as quantum channels 
$M_{X\times Y}\to M_{A\times B}$ (here $M_Z$ denotes the space of all $Z\times Z$ complex matrices)
whose marginal channels are well-defined.  
In Section \ref{s_cqnsc}, we define the quantum versions of the classes in (\ref{eq_chain0}), arriving at an  
analogous chain
\begin{equation}\label{eq_chainq0}
\cl Q_{\rm loc}\subseteq \cl Q_{\rm q}\subseteq \cl Q_{\rm qa}\subseteq \cl Q_{\rm qc}\subseteq \cl Q_{\rm ns}.
\end{equation}
The base for our definitions is a quantisation of positive operator valued measures, which we develop in 
Section \ref{s_gpovm}. The \emph{stochastic operator matrices} defined therein replace the families 
$(E_{x,a})_{x\in X,a\in A}$ of measurement operators that play a crucial role in the definitions of the 
classical classes (\ref{eq_chain0}).
In Section \ref{s_uopsys}, we define a universal operator system $\cl T_{X,A}$, whose 
concrete representations on Hilbert spaces are precisely determined by stochastic operator matrices. 
Our route passes through the definition of a universal ternary ring of operators $\cl V_{X,A}$ of 
a given $A\times X$-block operator isometry, which 
is a generalisation of the Brown algebra of a unitary matrix \cite{b} 
(see also \cite{ghj}).  
We describe $\cl T_{X,A}$ as a quotient of a full matrix algebra (Corollary \ref{c_quotxa});
this is a quantum version of a previous known result in the classical case \cite{fkpt}. 
We show that any such quotient 
possesses the local lifting property \cite{kptt_adv}. This unifies a number of results in the literature, 
implying in particular \cite[Theorem 4.9]{harris}. 

In Section \ref{s_dtp}, we provide operator theoretic descriptions of the classes 
$\cl Q_{\rm loc}$, $\cl Q_{\rm qa}$, $\cl Q_{\rm qc}$ and $\cl Q_{\rm ns}$, 
establishing a perfect correspondence between the elements of these classes and states on operator system tensor products.
We see that, similarly to the case of classical NS correlations \cite{lmprsstw}, 
each QNS correlation of the class $\cl Q_{\rm qc}$
arises from a state on the commuting tensor product $\cl T_{X,A}\otimes_{\rm c}\cl T_{Y,B}$, and that 
similar descriptions hold for the rest of the aforementioned classes. 
Along with the hierarchy (\ref{eq_chainq0}), we introduce an intermediate chain
\begin{equation}\label{eq_chaincq0}
\cl{CQ}_{\rm loc}\subseteq \cl{CQ}_{\rm q}\subseteq \cl{CQ}_{\rm qa}\subseteq \cl{CQ}_{\rm qc}\subseteq \cl{CQ}_{\rm ns},
\end{equation}
lying between 
(\ref{eq_chain0}) and (\ref{eq_chainq0}), whose terms are classes of \emph{classical-to-quantum no-signalling (CQNS) 
correlations}.
We define their universal operator system, and provide analogous characterisations in terms of states on 
its  tensor products; this is achieved in Section \ref{s_ctqqnsc}. In Section \ref{s_crs}, 
we point out the canonical surjections $\cl Q_{\rm x}\to \cl{CQ}_{\rm x} \to \cl C_{\rm x}$ (where ${\rm x}$ denotes any 
specific correlation class from the set $\{{\rm loc}, {\rm q}, {\rm qa}, {\rm qc}, {\rm ns}\}$). 
Combined with the separation results at each term, known for (\ref{eq_chain0}), 
this implies that the inclusions in (\ref{eq_chainq0}) and (\ref{eq_chaincq0}) are proper. 

The class $\cl Q_{\rm loc}$ at the ground level of the chain (\ref{eq_chainq0}) is in fact well-known: 
its elements are precisely the local operations and shared randomness (LOSR) channels (see e.g. \cite[p. 358]{jw}). 
Thus, the channels from $\cl Q_{\rm q}$ can be thought of as entanglement assisted LOSR transformations, 
and a similar interpretation can be adopted for the higher terms of (\ref{eq_chainq0}). 

The notion of a synchronous NS correlation \cite{psstw} is of crucial importance 
when correlations are employed as strategies of non-local games. Here, we assume that $X = Y$ and $A = B$. 
These correlations were characterised in \cite{psstw} as arising from traces on a universal C*-algebra $\cl A_{X,A}$
-- the free product of $|X|$ copies of the $|A|$-dimensional abelian C*-algebra. 
In Section \ref{s_fc}, we propose two quantum versions of synchronicity. 
\emph{Fair} correlations are defined in operational terms, but display a lower level of relevance than 
\emph{tracial} correlations, which are defined operator algebraically, via traces on the 
universal C*-algebra of a stochastic operator matrix.
Tracial QNS correlations are closely related to factorisable channels \cite{delaroche} 
which have been used to produce counterexamples to the asymptotic Birkhoff conjecture \cite{haag-musat0}. 
More precisely, if one restricts attention to QNS correlations that arise from the Brown algebra
as opposed to the ternary ring of operators $\cl V_{X,A}$, then the tracial QNS correlations
are precisely the couplings of a pair of factorisable channels with equal terms. 

Restricted to CQNS and NS correlations, traciality produces classes of correlations that 
strictly contain synchronous NS correlations. The difference between synchronous and tracial NS correlations
can be heuristically compared to that between projection and positive operator valued measures. 
The operational significance of tracial QNS, tracial CQNS and tracial NS correlations arises from the preservation 
of appropriate classes of states, which quantise the symmetry possessed by the 
classical pure states supported on the diagonal of a matrix algebra. 
The ground class, of \emph{locally reciprocal states}, turns out to be a twisted version of de Finetti states 
\cite{ckr}. Thus, the higher classes of \emph{quantum reciprocal} and \emph{C*-reciprocal states} 
can be thought of as an entanglement assisted and a commuting model version, respectively, of de Finetti states. 

In Section \ref{s_tqnlg}, we point out how QNS and CQNS correlations can be 
used as strategies for quantum-to-quantum and classical-to-quantum non-local games. 
This is not an exhaustive treatment, and is rather intended to summarise several directions and provide 
a general context that we hope to investigate subsequently. 
In Subsection \ref{ss_gcg}, we show that, when compared to NS correlations, 
CQNS correlations provide a significant advantage in the graph colouring game \cite{cmnsw}.
Employing the CQNS classes, we define new versions of quantum chromatic numbers of a classical graph $G$.
The class $\cl{CQ}_{\rm loc}$ yields the well-known orthogonal rank $\xi(G)$ of $G$ \cite{ss}; thus, 
the chromatic numbers $\xi_{\rm q}(G)$ and $\xi_{\rm qc}(G)$, 
arising from $\cl{CQ}_{\rm q}$ and $\cl{CQ}_{\rm qc}$, respectively, can be thought of as 
entanglement assisted and commuting model versions of this classical graph parameter. 
We show that $\xi_{\rm qc}(G)$ does not degenerate, in that it is always lower bounded by 
$\sqrt{d/\theta(G)}$, where $d$ is the number of vertices of $G$ and $\theta(G)$ is its Lov\'{a}sz number. 

In Subsection \ref{ss_qhncg}, we define a non-commutative version of the graph homomorphism game 
\cite{mrob}. We show that its perfect strategies from the class $\cl Q_{\rm loc}$ correspond precisely to 
non-commutative graph homomorphisms in the sense of Stahlke \cite{stahlke}. 
Thus, the perfect strategies from the larger classes in (\ref{eq_chainq0}) 
can be thought of as quantum non-commutative graph 
homomorphisms. 
We note that special cases have been previously considered in \cite{bcehpsw} and \cite{mrv}. 
The treatment in the latter papers was restricted to non-commutative graph isomorphisms, and
the suggested approach was operator-algebraic. 
We remedy this by suggesting, up to our knowledge, the first 
operational approach to non-commutative graph homomorphisms, thus aligning 
the non-commutative case with the case of quantum homomorphisms between classical graphs \cite{mrob}. 

Finally, in Subsection \ref{ss_gnlg}, we introduce a quantum version of non-local games that contains as a 
special case the games considered in the previous subsections. 
To this end, we view the rule predicate as a map between the projection lattices of algebras of diagonal
matrices. 
We define game composition, show that the perfect strategies from a fixed class 
${\rm x}\in \{{\rm loc}, {\rm q}, {\rm qa}, {\rm qc}, {\rm ns}\}$ are closed under 
channel composition, and prove that channel composition preserves traciality. 
Some of these results extend results previously proved in \cite{po} in the case of classical
no-signalling strategies.



\section{Preliminaries}\label{s_prel}

All inner products appearing in the paper will be assumed linear in the first variable.
Let $H$ be a Hilbert space.
We denote by $\cl B(H)$ the space of all bounded linear operators on $H$
and often write $\cl L(H)$ if $H$ is finite dimensional.
If $\xi,\eta \in H$, we write $\xi\eta^*$ for the rank one operator given by
$(\xi\eta^*)(\zeta) = \langle\zeta,\eta\rangle \xi$.
In addition to inner products, $\langle\cdot,\cdot\rangle$ will denote 
bilinear dualities between a vector space and its dual.
We write $\cl B(H)^+$ for the cone of positive operators in $\cl B(H)$,
denote by 
$\cl T(H)$ its ideal of trace class operators, and by $\Tr$ -- the trace functional on $\cl T(H)$.

An \emph{operator system} is a self-adjoint subspace $\cl S$ of $\cl B(H)$ for some Hilbert space $H$, 
containing the identity operator $I_H$. The linear space $M_n(\cl S)$ of all $n$ by $n$ 
matrices with entries in $\cl S$ can be canonically identified with a subspace of $\cl B(H^n)$, 
where $H^n$ is the direct sum of $n$-copies of $H$; 
we set $M_n(\cl S)^+ = M_n(\cl S)\cap \cl B(H^n)^+$ and write $\cl S_h$ for the real vector space of all hermitian 
elements of $\cl S$.
If $K$ is a Hilbert space, $\cl T\subseteq \cl B(K)$ is an operator system and 
$\phi : \cl S\to \cl T$ is a linear map, we let $\phi^{(n)} : M_n(\cl S)\to M_n(\cl T)$ be the 
(linear) map given by $\phi^{(n)}((x_{i,j})_{i,j}) = (\phi(x_{i,j}))_{i,j}$. 
The map $\phi$ is called \emph{positive} (resp. \emph{unital}) 
if $\phi(\cl S^+)\subseteq \cl T^+$ (resp. $\phi(I_H) = I_K$), and 
\emph{completely positive} if $\phi^{(n)}$ is positive for every $n\in \bb{N}$. 
We call $\phi$ a complete order embedding if it is injective and $\phi^{-1}|_{\phi(\cl S)} : \phi(\cl S) \to \cl S$ 
is completely positive; 
we write $\cl S\subseteq_{\rm c.o.i.} \cl T$. 
We note that $\bb{C}$ is an operator system in a canonical way;
a \emph{state} of $\cl S$ is a unital positive (linear) map $\phi : \cl S\to \bb{C}$. 
We denote by $S(\cl S)$ the (convex) set of all states of $\cl S$. 
We note that every operator system is an operator space in a canonical fashion,
and denote by $\cl S^{\rm d}$ the dual Banach space of $\cl S$, equipped with its 
canonical matrix order structure. 
Operator systems can be described abstractly via a set of axioms \cite{Pa}; 
we refer the reader to \cite{er}, \cite{Pa} and \cite{Pi} for details and for further background on 
operator space theory.

We denote by $|X|$ the cardinality of a finite set $X$,
let $H^X = \oplus_{x\in X}H$ and denote by
$M_X$ the space of all complex matrices of size $|X|\times |X|$; 
we identify $M_X$ with $\cl L(\bb{C}^X)$ and 
write $I_X = I_{\bb{C}^{X}}$. 
For $n\in \bb{N}$, we set $[n] = \{1,\dots,n\}$ and $M_n = M_{[n]}$. 
We write $(e_x)_{x\in X}$ for the canonical orthonormal basis of $\bb{C}^{X}$,
denote by $\cl D_X$ the subalgebra of $M_X$ of all diagonal, with respect to the
basis $(e_x)_{x\in X}$, matrices, and let $\Delta_X : M_X\to \cl D_X$
be the corresponding conditional expectation.

When $\omega$ is a linear functional on $M_X$, we often write $\omega = \omega_X$.
The canonical complete order isomorphism
from $M_X$ onto $M_X^{\rm d}$ maps
an element $\omega\in M_X$ to the linear functional $f_{\omega} : M_X\to \bb{C}$ given by 
$f_{\omega}(T) = \Tr(T\omega^{\rm t})$ (here, and in the sequel, $\omega^{\rm t}$ denotes the 
transpose of $\omega$ in the canonical basis); see e.g. \cite[Theorem 6.2]{ptt}. 
We will thus consider $M_X$ as self-dual space with pairing 
\begin{equation}\label{eq_rdu}
(\rho,\omega) \to \langle \rho,\omega\rangle := \Tr(\rho\omega^{\rm t}).
\end{equation}
On the other hand, note that the Banach space predual $\cl B(H)_*$ can be canonically identified with 
$\cl T(H)$; every normal functional $\phi : \cl B(H) \to \bb{C}$ thus 
corresponds to a (unique) operator $S_{\phi} \in \cl T(H)$ such that 
$\phi(T) = \Tr(TS_{\phi})$, $T\in \cl B(H)$.
In the case where $X$ is a fixed finite set (which will sometimes come in the form of a direct product),
we will use a mixture of the two dualities just discussed: if $\omega,\rho\in M_X$, $S\in \cl T(H)$ and $T\in \cl B(H)$, 
it will be convenient to continue writing
$$\langle \rho\otimes T, \omega\otimes S\rangle = \Tr(\rho\omega^{\rm t}) \Tr(TS).$$

If $X$ and $Y$ are finite sets, we identify $M_X\otimes M_Y$ with $M_{X\times Y}$ and
write $M_{XY}$ in its place. Similarly, we set $\cl D_{XY} = \cl D_X\otimes \cl D_Y$.
Here, and in the sequel, we use the symbol $\otimes$ to denote the algebraic tensor product of 
vector spaces.
For an element $\omega_X\in M_X$ and a Hilbert space $H$, we let 
$L_{\omega_X} : M_X\otimes \cl B(H)\to \cl B(H)$ be the 
linear map given by 
$L_{\omega_X}(S\otimes T) = \langle S,\omega_X\rangle T$. 
If $H = \bb{C}^Y$ and $\omega_Y\in M_Y$, we thus have linear maps
$L_{\omega_X} : M_{XY}\to M_Y$ and $L_{\omega_Y} : M_{XY}\to M_X$; note that  
$$\langle L_{\omega_X}(R),\rho_Y\rangle = \langle R, \omega_X\otimes \rho_Y\rangle, \ \ R\in M_{XY}, \rho_Y\in M_Y,$$
and a similar formula holds for $L_{\omega_Y}$.
We let $\Tr_X : M_{XY}\to M_Y$ (resp. $\Tr_Y : M_{XY}\to M_X$) be the partial trace, that is,
$\Tr_X = L_{I_X}$ (resp. $\Tr_Y = L_{I_Y}$).

Let $X$ and $A$ be finite sets.
A \emph{classical information channel} from $X$ to $A$
is a positive trace preserving linear map $\cl N : \cl D_X\to \cl D_A$.
It is clear that if $\cl N : \cl D_X\to \cl D_A$ is a classical channel
then $p(\cdot|x) := \cl N(e_x e_x^*)$ is a probability distribution over $A$, and
that $\cl N$ is completely determined by the family $\{(p(a|x))_{a\in A} : x\in X\}$.

A \emph{quantum channel} from $M_X$ into $M_A$ is a completely positive trace preserving map
$\Phi : M_X\to M_A$; such a $\Phi$ will be called \emph{$(X,A)$-classical} if 
$\Phi = \Delta_A\circ \Phi \circ \Delta_X$. 
A classical channel $\cl N : \cl D_X\to \cl D_A$ gives rise to a $(X,A)$-classical (quantum) channel
$\Phi_{\cl N} : M_X\to M_A$ by letting $\Phi_{\cl N} = \cl N \circ \Delta_X$.
Conversely, a quantum channel $\Phi : M_X\to M_A$ induces a classical channel
$\cl N_{\Phi} : \cl D_X\to \cl D_A$ by letting $\cl N_{\Phi} = \Delta_A\circ \Phi|_{\cl D_X}$.
Note that $\cl N_{\Phi_{\cl N}} = \cl N$ for every classical channel $\cl N$.

Let $X,Y,A$ and $B$ be finite sets.
A \emph{quantum correlation over $(X,Y,A,B)$} 
(or simply a \emph{quantum correlation} if the sets are understood from the context) 
is a quantum channel $\Gamma : M_{XY}\to M_{AB}$.
Such a $\Gamma$ is called a
\emph{quantum no-signalling (QNS) correlation} \cite{dw} if
\begin{equation}\label{eq_qns1}
\Tr\mbox{}_A\Gamma(\rho_X\otimes \rho_Y) = 0 \ \mbox{ whenever } \Tr(\rho_X) = 0
\end{equation}
and
\begin{equation}\label{eq_qns2}
\Tr\mbox{}_B\Gamma(\rho_X\otimes \rho_Y) = 0 \ \mbox{ whenever } \Tr(\rho_Y) = 0.
\end{equation}
We denote by $\cl Q_{\rm ns}$ the set of all QNS correlations; it is clear that 
$\cl Q_{\rm ns}$ is a closed convex subset of the cone ${\rm CP}(M_{XY},M_{AB})$
of all completely positive maps from $M_{XY}$ into $M_{AB}$.

\begin{remark}\label{c_strpt}\rm
A quantum channel $\Gamma : M_{XY}\to M_{AB}$ is a QNS correlation if and only if 
$$\Tr\mbox{}_A\Gamma(\rho') = 0 \ \mbox{ and  } \Tr\mbox{}_B\Gamma(\rho'') = 0$$
 provided  $\rho'$, $\rho''\in M_{XY}$ are such that $\Tr\mbox{}_X\rho'=0$ and $\Tr\mbox{}_Y\rho''=0$.
 Indeed, suppose that $\Gamma$ is a QNS correlation and $\rho'\in M_{XY}$, 
 $\Tr\mbox{}_X\rho'=0$. 
Writing $\rho'=\sum_{x,x',y,y'}\rho'_{x,x',y,y'}e_{x}e_{x'}^*\otimes e_ye_{y'}^*$, we have that
 $\sum_{x\in X}\rho'_{x,x,y,y'}=0$ for all $y$, $y'\in Y$.
Thus $\Tr\left(\sum_{x\in X}\rho'_{x,x,y,y'}e_xe_x^*\right) = 0$, and hence
$$\Tr\mbox{}_A\Gamma\left(\left(\sum_{x\in X} \rho'_{x,x,y,y'}e_xe_x^*\right)\otimes e_ye_{y'}^*\right) = 0, \ \ y, y'\in Y.$$ 
Since $\Tr e_xe_{x'}^* = \delta_{x,x'}$,
we also have $\Tr_A\Gamma(e_xe_{x'}^* \otimes e_ye_{y'}^*) = 0$ if $x\ne x'$, for all $y, y'\in Y$. 
It follows that $\Tr_A\Gamma(\rho') = 0$.  
The second property is verified similarly, while the converse direction of the statement is trivial.  
\end{remark}

A \emph{classical correlation} over $(X,Y,A,B)$
is a family
$$p = \left\{(p(a,b|x,y))_{(a,b)\in A\times B} : (x,y)\in X\times Y\right\},$$
where $(p(a,b|x,y))_{(a,b)\in A\times B}$ is a probability distribution for each $(x,y)\in X\times Y$;
classical correlations $p$ thus correspond precisely to classical channels $\cl N_p : \cl D_{XY}\to \cl D_{AB}$.
A \emph{classical no-signalling correlation} (or simply a no-signalling \emph{(NS)} correlation)
is a correlation $p = ((p(a,b|x,y))_{a,b})_{x,y}$ that satisfies the conditions
\begin{equation}\label{eq_ns1}
\sum_{a'\in A} p(a',b|x,y) = \sum_{a'\in A} p(a',b|x',y), \ \ \ x,x'\in X, y\in Y, b\in B,
\end{equation}
and
\begin{equation}\label{eq_ns2}
\sum_{b'\in B} p(a,b'|x,y) = \sum_{b'\in B} p(a,b'|x,y'), \ \ \ x\in X, y,y'\in Y, a\in A.
\end{equation}
We denote by $\cl C_{\rm ns}$ the set of all NS correlations
and identify its elements with classical channels from $\cl D_{XY}$ to $\cl D_{AB}$.
Given a classical correlation $p$, we write $\Gamma_p = \Phi_{\cl N_p}$; thus,
$\Gamma_p : M_{XY}\to M_{AB}$ is 
the $(X\times Y,A\times B)$-classical channel given by
\begin{equation}\label{eq_Gap}
\Gamma_p(\rho) = 
\sum_{x\in X,y\in Y} \sum_{a \in A,b\in B}
p(a,b|x,y) \left\langle \rho(e_x\otimes e_y),e_x\otimes e_y\right\rangle
e_ae_a^*\otimes e_b e_b^*.
\end{equation}

\begin{remark}\label{r_NSQNS}
{\rm 
If $p$ is a classical correlation over $(X,Y,A,B)$ then 
$p$ is an NS correlation precisely when $\Gamma_p$ is a QNS correlation. 
Indeed, if $\Tr\rho_X = 0$ and $p$ satisfies (\ref{eq_ns1}) and (\ref{eq_ns2}) then
\begin{eqnarray*}
& & \Tr\mbox{}_A\Gamma_p(\rho_X\otimes \rho_Y)\\
& = &\sum_{x\in X,y\in Y} \sum_{a \in A,b\in B}
p(a,b|x,y) \left\langle \rho_Xe_x, e_x\right\rangle\left\langle\rho_Ye_y, e_y\right\rangle e_b e_b^*\\
& = &
\sum_{y\in Y} \sum_{b\in B}
\left(\sum_{x\in X}\sum_{a\in A}p(a,b|x,y) 
\left\langle \rho_Xe_x, e_x\right\rangle\right)\left\langle\rho_Ye_y, e_y\right\rangle e_b e_b^*
= 0;
\end{eqnarray*}
(\ref{eq_qns2}) is checked similarly. 
Conversely, assuming that $\Gamma_p$ satisfies (\ref{eq_qns1}) and (\ref{eq_qns2}),
the relations (\ref{eq_ns1}) and (\ref{eq_ns2}) are obtained by substituting in
(\ref{eq_Gap}) 
$\rho = e_{x}e_{x}^*\otimes e_{y}e_{y}^* - e_{x'}e_{x'}^*\otimes e_{y}e_{y}^*$ 
and 
$\rho = e_{x}e_{x}^*\otimes e_{y}e_{y}^* - e_{x}e_{x}^*\otimes e_{y'}e_{y'}^*$.
It follows that if $\Gamma$ is a $(X\times Y, A\times B)$-classical QNS correlation then 
$\Gamma = \Gamma_p$ for some NS correlation $p$. 
}
\end{remark}

Let $H_1,\dots,H_k$ be Hilbert spaces, at most one of which is infinite dimensional, 
$T\in \cl B(H_1\otimes \cdots \otimes H_k)$ and
$f$ be a bounded functional on $\cl B(H_{i_1}\otimes\cdots \otimes H_{i_k})$, where $k\leq n$ and
$i_1,\dots,i_k$ are distinct elements of $[n]$ (not necessarily in increasing order).
We will use the expression
$L_{f}(T)$, or $\langle T,f\rangle$ (in the case $k = n$), without mentioning explicitly that a suitable
permutation of the tensor factors has been applied before the action of $f$.
We note that, if $g$ is a bounded functional on 
$\cl B(H_{j_1}\otimes\cdots \otimes H_{j_l})$, where $l\leq n$ and
the subset $\{j_1,\dots,j_l\}$ does not intersect $\{i_1,\dots,i_k\}$, then 
\begin{equation}\label{eq_LfLg}
L_f L_g = L_g L_f.
\end{equation}
Considering an element $\omega\in M_X$ as a functional on $M_X$ via (\ref{eq_rdu}), 
we have that, if $E = (E_{x,x'})_{x,x'}\in M_X\otimes\cl B(H)$ then 
\begin{equation}\label{Exx'}
L_{e_x e_{x'}^*}(E) = E_{x,x'}, \ \ \ x,x'\in X.
\end{equation}


\section{Stochastic operator matrices}\label{s_gpovm}

Let $X,Y,A$ and $B$ be finite sets.
A \emph{stochastic operator matrix} over $(X,A)$ is a positive operator
$E\in M_X\otimes M_A\otimes \cl B(H)$ for some Hilbert space $H$ such that
\begin{equation}\label{eq_gpovm0}
\Tr\mbox{}_A E = I_X\otimes I_{H}.
\end{equation}
We say that $E$ \emph{acts on} $H$.
This terminology becomes natural after noting that
the operator stochastic matrices $E\in \cl D_X\otimes \cl D_A\otimes \cl B(\bb{C})$
coincide, after the natural identification of $\cl D_X\otimes \cl D_A$ with the space of all $|X|\times |A|$ 
matrices, with the row-stochastic scalar-valued matrices.

Let
$E\in M_X\otimes M_A\otimes \cl B(H)$ be a stochastic operator matrix and $E_{x,x',a,a'}\in \cl B(H)$, 
$x,x'\in X$, $a,a'\in A$, be the operators such that
$$E = \sum_{x,x'\in X}\sum_{a,a'\in A} e_x e_{x'}^* \otimes e_a e_{a'}^* \otimes E_{x,x',a,a'};$$
we write $E = (E_{x,x',a,a'})_{x,x',a,a'}$. Note that 
$$E_{x,x',a,a'} = L_{e_{x} e_{x'}^* \otimes e_{a} e_{a'}^*} \left(E\right), \ \ x,x'\in X, a,a'\in A.$$
Set
$$E_{a,a'} = (E_{x,x',a,a'})_{x,x'\in X}\in M_X\otimes \cl B(H);$$
thus, $E_{a,a'} = L_{e_{a} e_{a'}^*}(E)$, $a,a'\in A$, 
and hence $E_{a,a} \in \left(M_X\otimes \cl B(H)\right)^+$, $a\in A$.
By Choi's Theorem, 
stochastic operator matrices $E$ are precisely the Choi matrices
of unital completely positive maps $\Phi_E : M_A\to M_X\otimes \cl B(H)$ defined by 
\begin{equation}\label{eq_PhiE}
\Phi_E(e_a e_{a'}^*) = E_{a,a'}, \ \ \ a,a'\in A.
\end{equation}

Recall that a \emph{positive operator-valued measure (POVM)} on a Hilbert space $H$, 
indexed by $A$, is 
a family $(E_a)_{a\in A}$ of positive operators on $H$, such that $\sum_{a\in A} E_a = I_H$. 
If $E_a$ is a projection for each $a\in A$, the family $(E_a)_{a\in A}$ is called a 
\emph{projection valued measure (PVM)}.

\begin{theorem}\label{p_coor}
Let $H$ be a Hilbert space and $E \in (M_X\otimes M_A\otimes \cl B(H))^+$.
The following are equivalent:
\begin{itemize}
\item[(i)] $E$ is a stochastic operator matrix;

\item[(ii)] $(E_{a,a})_{a\in A}$ is a POVM in $M_X\otimes \cl B(H)$;

\item[(iii)] $\Tr_A L_{\omega_X}(E) = I_{H}$, for all states  $\omega_X \in M_X$;

\item[(iv)] $\Tr_A L_{\omega_X}(E) = \Tr(\omega_X) I_{H}$, for all $\omega_X \in M_X$;

\item[(v)] there exists a Hilbert space $K$ and operators $V_{a,x} : H\to K$, $x\in X$, $a\in A$,
such that $(V_{a,x})_{a,x} \in \cl B(H^{X},K^{A})$ is an isometry and
\begin{equation}\label{eq_v}
E_{x,x',a,a'} = V_{a,x}^* V_{a',x'}, \ \ \ x,x'\in X, a,a'\in A.
\end{equation}
\end{itemize}

\noindent
In particular, if $E$ is a stochastic operator matrix then $(E_{x,x,a,a})_{a\in A}$ is a POVM for every $x\in X$.
\end{theorem}

\begin{proof}
(i)$\Leftrightarrow$(ii) and (iv)$\Rightarrow$(iii) are trivial, 
while (i)$\Rightarrow$(iii) is immediate from (\ref{eq_LfLg}). 

(iii)$\Rightarrow$(iv)
By assumption, 
$\Tr_A L_{\omega}(E) = \Tr(\omega) I_{H}$ for every state $\omega \in M_X.$
Write $\omega = \sum_{i=1}^4 \lambda_i \omega_i$, where $\omega_i$ is a state in $M_X$
and $\lambda_i\in \bb{C}$, $i = 1,2,3,4$.
Then
$$\Tr\mbox{}_A L_{\omega}(E) = \sum_{i=1}^4 \lambda_i \Tr\mbox{}_A L_{\omega_i}(E) =
\sum_{i=1}^4 \lambda_i I_{H} = \Tr(\omega) I_{H}.$$

(iii)$\Rightarrow$(i)
By (\ref{eq_LfLg}), for all $\omega_X\in S(M_X)$ and all normal states $\tau$ on $\cl B(H)$, we have 
\begin{eqnarray*}
\left\langle I_X\otimes I_H,\omega_X\otimes \tau\right\rangle
& = & 
1 =
\left\langle \Tr\mbox{}_A L_{\omega_X} (E),\tau\right\rangle 
= \left\langle L_{\omega_X} \Tr\mbox{}_A(E),\tau\right\rangle\\ 
& = & \left\langle \Tr\mbox{}_A (E),\omega_X\otimes \tau\right\rangle.
\end{eqnarray*}
By polarisation and linearity, 
$$\left\langle \Tr\mbox{}_A(E),\sigma\right\rangle = 
\left\langle I_X\otimes I_H,\sigma\right\rangle$$ 
for all $\sigma\in (M_X\otimes\cl B(H))_*$, and hence $\Tr_A  (E) = I_X\otimes I_H$. 

(i)$\Rightarrow$(v)
Let $\Phi = \Phi_E$ be the unital completely positive map given by (\ref{eq_PhiE}). 
By Stinespring's Dilation Theorem, there exist a Hilbert space $\tilde{K}$,
an isometry $V : \bb{C}^X\otimes H\to \tilde{K}$ and
a unital *-homomorphism $\pi : M_A\to \cl B(\tilde{K})$ such that
$\Phi(T) = V^* \pi(T) V$, $T\in M_A$.
Up to unitary equivalence, $\tilde{K} = \bb{C}^A\otimes K$ for some Hilbert space $K$
and $\pi(T) = T\otimes I_K$, $T\in M_A$. 
Write $V_{a,x} : H\to K$, $a\in A$, $x\in X$, for the entries of $V$, 
when $V$ is considered as a block operator matrix.
For $\xi,\eta\in H$, $x,x'\in X$ and $a,a'\in A$, we have
\begin{eqnarray*}
& & 
\left\langle E_{x,x',a,a'}\xi,\eta\right\rangle
=
\left\langle L_{e_{x} e_{x'}^*}(E_{a,a'}) \xi,\eta\right\rangle 
= 
\Tr\left( L_{e_{x} e_{x'}^*}(\Phi(e_a e_{a'}^*)) (\xi\eta^*)\right)\\
& = & 
\Tr\left(\Phi(e_a e_{a'}^*)((e_{x'} e_{x}^*) \otimes (\xi \eta^*))\right)\\ 
& = & 
\Tr\left(V^*((e_a e_{a'}^*)\otimes I_K)V (e_{x'} \otimes \xi)(e_{x}\otimes \eta)^*\right)\\
& = & 
\left\langle V^*((e_a e_{a'}^*)\otimes I_K)V (e_{x'} \otimes\xi), e_{x}\otimes \eta\right\rangle\\
& = &
\left\langle ((e_a e_{a'}^*)\otimes I_K)V (e_{x'} \otimes\xi), V(e_x\otimes \eta)\right\rangle\\
& = & 
\left\langle ((e_a e_{a'}^*)\otimes I_K)((e_{a'} e_{a'}^*)\otimes I_K)V (e_{x'} \otimes\xi), ((e_a e_a^*)\otimes I_K)V(e_x\otimes \eta)\right\rangle\\
& = & 
\left\langle V_{a',x'}\xi,V_{a,x}\eta\right\rangle 
= \left\langle V_{a,x}^* V_{a',x'}\xi,\eta\right\rangle.
\end{eqnarray*}

(v)$\Rightarrow$(ii)
Let $\xi = \sum_{x\in X} \sum_{a\in A} e_x\otimes e_a \otimes \xi_{x,a}$, where 
$\xi_{x,a}\in H$, $x\in X$, $a\in A$. 
Using (\ref{eq_v}), we have 
$$\left\langle E\xi,\xi \right\rangle
=
\sum_{x,x'\in X} \sum_{a,a'\in A} \left\langle V_{a',x'} \xi_{x',a'}, V_{a,x} \xi_{x,a}\right\rangle
= 
\left\|\sum_{x\in X} \sum_{a\in A} V_{a,x}\xi_{x,a}\right\|^2,
$$
and thus $E$ is positive. 
Since $V$ is an isometry, we have 
$$\sum_{a\in A} E_{x,x',a,a} = \sum_{a\in A} V_{a,x}^* V_{a,x'} = \delta_{x,x'} I_H.$$
\end{proof}

Let $(E_{x,a})_{a\in A}$ be a POVM on a Hilbert space $H$ for every $x\in X$.
A stochastic operator matrix of the form
\begin{equation}\label{eq_cgpovm}
E = \sum_{x\in X}\sum_{a\in A} e_x e_x^* \otimes e_a e_a^* \otimes E_{x,a}
\end{equation}
will be called \emph{classical}.
A general stochastic operator matrix can thus be thought of as a coordinate-free version of
a finite family of POVM's.

\medskip

\noindent {\bf Remarks. (i) }
In view of Theorem \ref{p_coor}, stochastic operator matrices are precisely the positive completions $E$ of
partially defined diagonal block matrices $D = (E_{a,a})_{a\in A}$ with entries in $M_X\otimes \cl B(H)$ and $\Tr_A(D) = I$.

\smallskip

{\bf (ii) }
The following generalisation of Naimark's Dilation Theorem was proved in \cite{paulsen_notes}: 
if $(E_{x,a})_{a\in A}\subseteq \cl B(H)$, $x\in X$, are POVM's 
then there exist a Hilbert space $\tilde{H}$, a PVM
$(\tilde{E}_a)_{a\in A}\subseteq \cl B(\tilde{H})$
and isometries $V_x : H\to \tilde{H}$, $x\in X$, with orthogonal ranges such that
\begin{equation}\label{eq_gennai}
E_{x,a} = V_x^* \tilde{E}_a V_x, \ \ \ a\in A, x\in X.
\end{equation}
This can be seen as a corollary of Theorem \ref{p_coor}:
given POVM's $(E_{x,a})_{a\in A}\subseteq \cl B(H)$, $x\in X$, 
let $E$ be the stochastic operator matrix defined by (\ref{eq_cgpovm}) and let 
$V = (V_{a,x})_{a,x}$ be the isometry from Theorem \ref{p_coor}. 
Set $\tilde{E}_a = e_a e_a^*\otimes I_H$, $a\in A$,
and let $V_x$ be the column isometry $(V_{a,x})_{a\in A} : H\to K^A$, $x\in X$. 
Then $(\tilde{E}_a)_{a\in A}$ is a PVM fulfilling (\ref{eq_gennai}).

\medskip

Let $E\in M_X\otimes M_A\otimes \cl B(H)$ be a stochastic operator matrix
and $\Phi = \Phi_E$ be given by (\ref{eq_PhiE}).
Recall that the predual $\Phi_* : M_X\otimes \cl T(H) \to M_A$ of $\Phi$ is the 
completely positive map satisfying 
$\langle \Phi_*(\rho),\omega\rangle = \langle \rho, \Phi(\omega)\rangle$, $\rho\in M_X\otimes \cl T(H)$, $\omega\in M_A$. 
For a state $\sigma\in \cl T(H)$, set 
$$\Gamma_{E,\sigma}(\rho_X) = \Phi_*(\rho_X\otimes\sigma), \ \ \ \rho_X\in M_X;$$
then $\Gamma_{E,\sigma} : M_X\to M_A$ is a quantum channel. 
We have
\begin{equation}\label{eq_gammad}
\Gamma_{E,\sigma}(\rho_X) = L_{\rho_X\otimes \sigma}(E),  \ \ \ \rho_X\in M_X;
\end{equation}
indeed, if $a, a'\in A$ then 
\begin{eqnarray*}
\left \langle \Gamma_{E,\sigma}(e_{x} e_{x'}^*), e_{a} e_{a'}^* \right\rangle 
& = & 
\left \langle \Phi_*(e_{x} e_{x'}^* \otimes \sigma), e_{a} e_{a'}^* \right\rangle 
= 
\left \langle e_{x} e_{x'}^* \otimes\sigma, \Phi(e_{a} e_{a'}^*)\right\rangle\\
& = & 
\left \langle e_{x} e_{x'}^* \otimes\sigma, E_{a,a'}\right\rangle 
= 
\left \langle \sigma, E_{x,x',a,a'}\right\rangle\\
& = & 
\left\langle L_{e_{x} e_{x'}^* \otimes \sigma}(E), e_{a} e_{a'}^*\right\rangle;
\end{eqnarray*}
(\ref{eq_gammad}) now follows by linearity. 
By Choi's Theorem, every quantum channel $\Phi : M_X\to M_A$ has the form
$\Gamma_{E,1}$ for some stochastic operator matrix $E\in M_X\otimes M_A$.

\begin{remark}\label{r_vmm}
{\rm Let $H$ be a Hilbert space and $E\in M_X\otimes M_A\otimes \cl B(H)$ be a stochastic operator matrix. 
The following are equivalent: 

(i) \ $E$ is classical; 

(ii) for each state $\sigma\in \cl T(H)$, the quantum channel $\Gamma_{E,\sigma} : M_X\to M_A$ is 
$(X,A)$-classical. }
\end{remark}

\begin{proof}
The channel $\Gamma_{E,\sigma}$ is $(X,A)$-classical if and only if 
$\Gamma_{E,\sigma}(e_x e_{x'}^*) = 0$ whenever $x\neq x'$
and 
$$\left \langle \Gamma_{E,\sigma}(e_x e_{x}^*), e_a e_{a'}^* \right\rangle = 0 \ \mbox{ whenever } a\neq a'.$$
The latter equality holds for every $\sigma$ if and only if $E_{x,x',a,a'} = 0$ whenever $x\neq x'$ and 
$E_{x,x,a,a'} = 0$ whenever $a\neq a'$, that is, if and only if $E$ is classical.
\end{proof}


\section{Three subclasses of QNS correlations}\label{s_cqnsc}

In this section, we introduce several classes of QNS correlations, which generalise corresponding 
classes of NS correlations studied in the literature (see e.g. \cite{lmprsstw}).


\subsection{Quantum commuting QNS correlations}\label{cs_cqnsc}

Let $H$ be a Hilbert space,
and $E\in M_X\otimes M_A\otimes \cl B(H)$ and $F\in M_Y\otimes M_B\otimes \cl B(H)$
be stochastic operator matrices. The pair $(E,F)$ will be called \emph{commuting} if
$L_{\omega_X\otimes \omega_A}(E)$ and $L_{\omega_Y\otimes \omega_B}(F)$
commute for all $\omega_X\in M_X$, $\omega_Y\in M_Y$, $\omega_A\in M_A$ and $\omega_B\in M_B$.
Writing $E = (E_{x,x',a,a'})_{x,x',a,a'}$ and $F = (F_{y,y',b,b'})_{y,y',b,b'}$, we have that
$(E,F)$ is commuting if and only if
$$E_{x,x',a,a'}F_{y,y',b,b'} = F_{y,y',b,b'}E_{x,x',a,a'}, \ x,x'\in X, y,y'\in Y, a,a'\in A, b,b'\in B.$$


\begin{proposition}\label{p_dotp}
Let $H$ be a Hilbert space and
$E\in M_X\otimes M_A\otimes \cl B(H)$, $F\in M_Y\otimes M_B\otimes \cl B(H)$ 
form a commuting pair of stochastic operator matrices.
There exists a unique operartor
$E\cdot F\in M_{XY}\otimes M_{AB}\otimes \cl B(H)$
such that
\begin{equation}\label{eq_long}
\langle E\cdot F,\rho_X \otimes \rho_Y \otimes \rho_A \otimes \rho_B\otimes \sigma\rangle
= \left\langle L_{\rho_X\otimes \rho_A}(E)L_{\rho_Y\otimes \rho_B}(F),\sigma \right\rangle,
\end{equation}
for all $\rho_X \in M_X$, $\rho_Y\in M_Y$, $\rho_A \in M_A$, $\rho_B\in M_B$ and $\sigma\in \cl T(H)$.
Moreover, 
\begin{itemize}
\item[(i)] $E\cdot F$ is a stochastic operator matrix;
\item[(ii)] $\|E\cdot F\| \leq \|E\|\|F\|$;
\item[(iii)] If $\sigma \in \cl T(H)$ is a state then $\Gamma_{E \cdot F,\sigma}$ is a QNS correlation.
\end{itemize}
\end{proposition}

\begin{proof} 
Let $$E\cdot F := \left(E_{x,x',a,a'}F_{y,y',b,b'}\right) \in M_{XY}\otimes M_{AB}\otimes \cl B(H).$$ 
Denote by $\cl A$ (resp. $\cl B$) the C*-algebra, generated by $E_{x,x',a,a'}$, 
$x,x'\in X$, $a,a'\in A$ (resp. $F_{y,y',b,b'}$, $y,y'\in Y$, $b,b'\in B$); 
by assumption, $\cl B\subseteq \cl A'$. 
Let $\pi_{\cl A} : M_{XA}(\cl A) \to M_{XYAB}(\cl B(H))$ 
(resp. $\pi_{\cl B} : M_{YB}(\cl B) \to M_{XYAB}(\cl B(H))$)
be the *-representation given by $\pi_{\cl A}(S) = S\otimes I_{YB}$ (resp. $\pi_{\cl B}(T) = T\otimes I_{XA}$).
Then the ranges of $\pi_{\cl A}$ and $\pi_{\cl B}$ commute and hence the pair $(\pi_{\cl A},\pi_{\cl B})$
gives rise to a *-representation 
$\pi : M_{XA}(\cl A)\otimes_{\max}M_{YB}(\cl B) \to M_{XYAB}(\cl B(H))$ 
with $\pi(S\otimes T) = \pi_{\cl A}(S)\pi_{\cl B}(T)$, $S\in M_{XA}(\cl A)$, $T\in M_{YB}(\cl B)$.
Thus, $E \cdot F = \pi(E\otimes F) \in M_{XYAB}(\cl B(H))^+$.
Inequality (ii) now follows from the contractivity of *-representations.
In addition,
\begin{eqnarray*}
\Tr\mbox{}_{AB}(E\cdot F)
& = &
\sum_{a\in A} \sum_{b\in B} \left(E_{x,x',a,a}F_{y,y',b,b}\right)_{x,x',y,y'}\\
& = &
\left(\delta_{x,x'}  \delta_{y,y'} I\right)_{x,x',y,y'}
= 
I_{XY}\otimes I_H,
\end{eqnarray*}
that is, $E\cdot F$ is a stochastic operator matrix. 
For $x,x'\in X$, $y,y'\in Y$, $a,a'\in A$, $b,b'\in B$ and $\sigma\in \cl T(H)$, we have
\begin{eqnarray}\label{eq_EcdotF}
& & \left\langle E\cdot F,  e_{x} e_{x'}^* \otimes e_{y} e_{y'}^* \otimes e_{a} e_{a'}^* \otimes e_{b} e_{b'}^* \otimes \sigma\right\rangle\\
& = & 
\left\langle E_{x,x',a,a'} F_{y,y',b,b'}, \sigma\right\rangle
= 
\left\langle L_{e_{x} e_{x'}^* \otimes e_{a} e_{a'}^*}(E)L_{e_{y} e_{y'}^*\otimes e_{b} e_{b'}^*}(F),\sigma \right\rangle,  \nonumber
\end{eqnarray}
and (\ref{eq_long}) follows by linearity.

To show (iii), let $\sigma \in \cl T(H)$ be a state.
Suppose that $\rho_X\in M_X$ is traceless and $\rho_Y\in M_Y$. 
For every $\tau_B\in M_B$, by (\ref{eq_long}) and Theorem \ref{p_coor}, we have 
\begin{eqnarray*}
\left\langle \Tr\mbox{}_A\Gamma_{E\cdot F,\sigma}(\rho_X \otimes \rho_Y), \tau_B\right\rangle
& = &
\left\langle \Gamma_{E\cdot F, \sigma}(\rho_X \otimes \rho_Y), I_A\otimes \tau_B\right\rangle\\
& = &
\left\langle E\cdot F, \rho_X\otimes \rho_Y \otimes I_A\otimes \tau_B \otimes \sigma\right\rangle\\
& = &
\left\langle \Tr\mbox{}_A L_{\rho_X}(E) L_{\rho_Y\otimes \tau_B}(F), \sigma\right\rangle = 0.
\end{eqnarray*}
Thus, (\ref{eq_qns1}) is satisfied; by symmetry, so is (\ref{eq_qns2}).
\end{proof}

If $\xi$ is a unit vector in $H$,
we set for brevity $\Gamma_{E,F,\xi} = \Gamma_{E\cdot F, \xi\xi^*}$.

\begin{definition}\label{d_qcqns}
A QNS correlation of the form $\Gamma_{E,F,\xi}$, where $(E,F)$ is a commuting pair of
stochastic operator matrices acting on a Hilbert space $H$, and $\xi\in H$ is a unit vector,
will be called \emph{quantum commuting}.
\end{definition}

We denote by $\cl Q_{\rm qc}$ the set of all quantum commuting QNS correlations.

\begin{proposition}\label{p_genstok}
In Definition \ref{d_qcqns} one can assume, without gain of generality, that $\sigma$ is an
arbitrary state.
\end{proposition}

\begin{proof}
Suppose that $H$ is a Hilbert space and 
$E\in M_X\otimes M_A\otimes \cl B(H)$, $F\in M_Y\otimes M_B\otimes \cl B(H)$
form a commuting pair of stochastic operator matrices. 
Let $\sigma$ be a state in $\cl T(H)$ and write 
$\sigma = \sum_{i=1}^{\infty} \lambda_i \xi_i \xi_i^*$, where $(\xi_i)_{i=1}^{\infty}$ is 
sequence of unit vectors and $\lambda_i \geq 0$, $i\in \bb{N}$, are such that 
$\sum_{i=1}^{\infty} \lambda_i = 1$. 
Set $\tilde{H} = H\otimes \ell^2$ and $\xi = \sum_{i=1}^{\infty} \sqrt{\lambda_i} \xi_i \otimes e_i$;
then $\xi$ is a unit vector in $\tilde{H}$ and 
$\langle\xi\xi^*,T\otimes I_{\ell^2}\rangle = \langle \sigma,T\rangle$, $T\in \cl B(H)$.

Let $\tilde{E} = E\otimes I_{\ell^2}$ and $\tilde{F} = F\otimes I_{\ell^2}$; 
thus, 
$\tilde{E}$ and $\tilde{F}$ are stochastic operator matrices acting on $\tilde{H}$ that form a commuting pair.  
Moreover, if 
$\rho_X\in M_X$, $\rho_Y \in M_Y$, $\sigma_A \in M_A$ and $\sigma_B\in M_B$ then 
\begin{eqnarray*}
& & \left\langle \Gamma_{\tilde{E},\tilde{F},\xi}(\rho_X\otimes \rho_Y), \sigma_A\otimes \sigma_B\right\rangle
=
\left\langle \tilde{E} \cdot \tilde{F}, \rho_X\otimes \rho_Y \otimes \sigma_A\otimes \sigma_B \otimes \xi\xi^*\right\rangle\\
& = &
\left\langle L_{\rho_X\otimes \sigma_A}(\tilde{E})L_{\rho_Y\otimes \sigma_B}(\tilde{F}),\xi\xi^*\right\rangle
= 
\left\langle \left(L_{\rho_X\otimes \sigma_A}(E)L_{\rho_X\otimes \sigma_A}(F)\right)\otimes I_{\ell^2} ,\xi\xi^*\right\rangle\\
& = &
\left\langle L_{\rho_X\otimes \sigma_A}(E)L_{\rho_X\otimes \sigma_A}(F),\sigma\right\rangle
= 
\left\langle \Gamma_{E \cdot F,\sigma}(\rho_X\otimes \rho_Y), \sigma_A\otimes \sigma_B\right\rangle.
\end{eqnarray*}
\end{proof}

\begin{remark}\label{r_qcqnsns}
{\rm Recall that a classical NS correlation $p$ over $(X,Y,A,B)$ is called \emph{quantum commuting} 
\cite{psstw, pt_QJM}
if there exist a Hilbert space $H$, 
POVM's $(E_{x,a})_{a\in A}$, $x\in X$, and $(F_{y,b})_{b\in B}$, $y\in Y$, on $H$
with $E_{x,a}F_{y,b} = F_{y,b}E_{x,a}$ for all $x,y,a,b$, 
and a unit vector $\xi\in H$, such that
$$p(a,b|x,y) = \langle E_{x,a}F_{y,b}\xi,\xi\rangle, \ \ \ x \in X, y\in Y, a\in A, b\in B.$$
Suppose that the stochastic operator matrices $E\in M_X\otimes M_A\otimes \cl B(H)$ and 
$F\in M_Y\otimes M_B\otimes \cl B(H)$ are classical, and correspond to the families 
$(E_{x,a})_{a\in A}$, $x\in X$, and $(F_{y,b})_{b\in B}$, $y\in Y$, respectively, as in (\ref{eq_cgpovm}). 
It is clear that pair $(E,F)$ is commuting and $E\cdot F$ is classical.  
We have that $\Gamma_{p} = \Gamma_{E,F,\xi}$.
Indeed, 
by Remark \ref{r_vmm}, the QNS correlation $\Gamma_{E,F,\xi}$ is classical; by 
Remark \ref{r_NSQNS}, 
$\Gamma_{E,F,\xi} = \Gamma_{p'}$ for some NS correlation $p'$. 
It is now straightforward that $p'=p$.
}
\end{remark}


\subsection{Quantum QNS correlations}\label{ss_qqnsc}

Let $H_A$ and $H_B$ be Hilbert spaces, and 
$E\in M_X\otimes M_A\otimes \cl B(H_A)$ and $F\in M_Y\otimes M_B\otimes \cl B(H_B)$
be stochastic operator matrices;
then 
$$E\otimes F \in M_X\otimes M_A\otimes \cl B(H_A)\otimes M_Y\otimes M_B\otimes \cl B(H_B).$$
Reshuffling the terms of the tensor product, 
we consider $E\otimes F$ as an element of $M_{XY}\otimes M_{AB}\otimes \cl B(H_A\otimes H_B)$;
to underline this distinction, the latter element will henceforth be denoted by $E\odot F$. 
Note that, if 
$$\tilde{E} = E\otimes I_{H_B}\in M_X\otimes M_A\otimes \cl B(H_A\otimes H_B)$$
and 
$$\tilde{F} = F\otimes I_{H_A}\in M_Y\otimes M_B\otimes \cl B(H_A\otimes H_B)$$
(where the last containment is up to a suitable permutation of the tensor factors),
then $(\tilde{E},\tilde{F})$ is a commuting pair of stochastic operator matrices, and $E\odot F = \tilde{E}\cdot \tilde{F}$. 
By Proposition \ref{p_dotp}, $E\odot F$ is a stochastic operator matrix on $H_A\otimes H_B$ and, 
if $\sigma\in \cl T(H_A\otimes H_B)$ is a state then, by Proposition \ref{p_dotp}, 
$\Gamma_{E\odot F,\sigma}$ is a QNS correlation.

\begin{remark}\label{r_split}
{\rm 
It is straightforward to check that, 
if $\sigma = \sigma_A\otimes \sigma_B$, where $\sigma_A\in \cl T(H_A)$ and 
$\sigma_B\in \cl T(H_B)$ are states, then 
$\Gamma_{E,\sigma_A}\otimes \Gamma_{F,\sigma_B} = \Gamma_{E\odot F,\sigma_A\otimes \sigma_B}.$}
\end{remark}

\begin{definition}\label{d_qqns}
(i) A QNS correlation $\Gamma : M_{XY}\to M_{AB}$ is called 
\emph{quantum} if there exist finite dimensional Hilbert spaces $H_A$ and $H_B$, 
stochastic operator matrices
$E\in M_X\otimes M_A\otimes \cl L(H_A)$ and $F\in M_Y\otimes M_B\otimes \cl L(H_B)$
and a pure state $\sigma\in \cl L(H_A\otimes H_B)$ such that $\Gamma = \Gamma_{E\odot F,\sigma}$. 

(ii) A QNS correlation will be called \emph{approximately quantum}
if it is the limit of a sequence of quantum QNS correlations.
\end{definition}

We denote by $\cl Q_{\rm q}$ (resp. $\cl Q_{\rm qa}$) the set of all quantum (resp. approximately quantum)
QNS correlations.
It is clear from the definitions that $\cl Q_{\rm q}\subseteq \cl Q_{\rm qc}$.
It will be shown later that $\cl Q_{\rm qc}$ is closed, and hence contains $\cl Q_{\rm qa}$.

Similarly to Proposition \ref{p_genstok}, it can be shown that quantum QNS correlations can 
equivalently be defined using arbitrary, as opposed to pure, states.

\begin{remark}\label{r_qqnsns}
{\rm Recall that a classical NS correlation $p$ over $(X,Y,A,B)$ is called \emph{quantum} 
if there exist finite dimensional Hilbert spaces $H_A$ and $H_B$, 
POVM's $(E_{x,a})_{a\in A}$, on $H_A$, $x\in X$, $(F_{y,b})_{b\in A}$ on $H_B$, $y\in Y$,
and a unit vector $\xi\in H_A\otimes H_B$, such that
\begin{equation}\label{eq_clq}
p(a,b|x,y) = \left\langle (E_{x,a}\otimes F_{y,b})\xi,\xi \right\rangle, \ \ \ x \in X, y\in Y, a\in A, b\in B.
\end{equation}
It is easy to verify that, if 
the stochastic operator matrices $E\in M_X\otimes M_A\otimes \cl B(H_A)$ and 
$F\in M_Y\otimes M_B\otimes \cl B(H_B)$ are classical, and determined by the families 
$(E_{x,a})_{a\in A}$, $x\in X$, and $(F_{y,b})_{b\in B}$, $y\in Y$, then 
$E\odot F$ is classical and determined by the family $\{(E_{x,a}\otimes F_{y,b})_{(a,b)\in A\times B} : 
(x,y)\in X\times Y\}$.
As in Remark \ref{r_qcqnsns}, it is easy to see that $\Gamma_{p} = \Gamma_{E,F,\xi}$.
}
\end{remark}

\begin{proposition}\label{p_Qqqaco}
The sets $\cl Q_{\rm q}$ and $\cl Q_{\rm qa}$ are convex. 
\end{proposition}

\begin{proof}
Let $E_i \in M_X\otimes M_A\otimes \cl L(K_{1,i})$
(resp. $F_i \in M_Y\otimes M_B\otimes \cl L(K_{2,i})$) be a stochastic operator matrix over $(X,A)$ (resp. $(Y,B)$)
and $\sigma_i = \eta_i \eta_i^*$ be a pure state on $K_{1,i}\otimes K_{2,i}$, $i = 1,\dots,n$.
Fix $\lambda_i \geq 0$, $i = 1,\dots,n$, with $\sum_{i=1}^n \lambda_i = 1$.
Let $K_{k} = \oplus_{i=1}^n K_{k,i}$, $k = 1,2$, $E = \oplus_{i=1}^n E_i$, $F = \oplus_{i=1}^n F_i$,
and $\eta = \oplus_{i=1}^n \sqrt{\lambda_i} \eta_i \in K_1\otimes K_2$, 
viewed as supported on the $(i,i)$-terms of 
$K_1\otimes K_2 \equiv \oplus_{i,j =1}^n K_{1,i}\otimes K_{2,j}$.
Set $\sigma = \eta\eta^*$.
Using distributivity, we consider 
$E$ (resp. $F$) as a stochastic operator matrix in 
$M_X\otimes M_A\otimes \cl L(K_{1})$ (resp. $M_Y\otimes M_B\otimes \cl L(K_{2})$).
A direct verification shows that 
$$\sum_{i=1}^n \lambda_i \Gamma_{E_i\odot F_i, \sigma_i} = \Gamma_{E\odot F,\sigma};$$
thus, $\cl Q_{\rm q}$ is convex, and the convexity of $\cl Q_{\rm qa}$ follows from the fact that 
$\cl Q_{\rm qa} = \overline{\cl Q_{\rm q}}$. 
\end{proof}


\subsection{Local QNS correlations}\label{ss_lqnsc}

It is clear that, if $\Phi : M_X\to M_A$ and $\Psi : M_Y \to M_B$ are quantum channels, then
the quantum channel $\Gamma := \Phi\otimes \Psi$ is a QNS correlation.

\begin{definition}\label{d_lqns}
A QNS correlation $\Gamma : M_{XY}\to M_{AB}$ is called \emph{local} if
it is a convex combination of quantum channels of the form
$\Phi\otimes\Psi$, where
$\Phi : M_X\to M_A$ and $\Psi : M_Y \to M_B$ are quantum channels.
\end{definition}

We denote by $\cl Q_{\rm loc}$ the set of all local QNS correlations. 
The elements of $\cl Q_{\rm loc}$ are precisely the maps that arise via
\emph{local operations and shared randomness (LOSR)} (see e.g. \cite[p. 358]{jw}).

\begin{remark}\label{r_lqqnsi}
We have that $\cl Q_{\rm loc}$ is a closed convex subset of $\cl Q_{\rm q}$.
\end{remark}

\begin{proof}
Let $\Phi : M_X\to M_A$ and $\Psi : M_Y \to M_B$ be quantum channels and
$E\in M_X\otimes M_A$ and $F\in M_Y\otimes M_B$ be the Choi matrices of $\Phi$ and $\Psi$, respectively. 
By Remark \ref{r_split},
$$\Phi \otimes \Psi = \Gamma_{E,1}\otimes \Gamma_{F,1} = \Gamma_{E\odot F,1}$$
and hence $\Phi \otimes \Psi \in \cl Q_{\rm q}$.

Let $(\Gamma_k)_{k\in \bb{N}} \subseteq \cl Q_{\rm loc}$ be a sequence with limit $\Gamma\in \cl Q_{\rm ns}$. 
Note that $\Gamma_k$ all are elements of a real vector space of dimension $2|X|^2 |Y|^2 |A|^2 |B|^2$. 
Let $L = 2|X|^2 |Y|^2 |A|^2 |B|^2 + 1$.
By Carath\'{e}odory's Theorem, $\Gamma_k = \sum_{l=1}^L \lambda_l^{(k)} \Phi_l^{(k)} \otimes \Psi_l^{(k)}$ as a 
convex combination. 
By compactness, we may assume, by passing to subsequences as necessary, that 
$\Phi_l^{(k)} \to_{k \to \infty} \Phi_l$, $\Psi_l^{(k)} \to_{k\to \infty} \Psi_l$ and 
$\lambda_l^{(k)} \to_{k\to \infty} \lambda_l$. 
Thus, $\Gamma = \sum_{l=1}^L \lambda_l \Phi_l \otimes \Psi_l$ as a 
convex combination, that is, $\Gamma\in \cl Q_{\rm loc}$, showing that $\cl Q_{\rm loc}$ is closed. 
\end{proof}

\begin{remark}\label{r_qlocnsns}
{\rm Recall that a classical NS correlation $p$ over $(X,Y,A,B)$ is called \emph{local} 
if there exist families of probability distributions 
$\{(p_k^1(a|x))_{a\in A} : x\in X\}$ and $\{(p_k^2(b|y))_{b\in B} : y\in Y\}$
and positive scalars $\lambda_k$, $k = 1,\dots,m$, 
such that $\sum_{k=1}^m \lambda_k = 1$ and
$$p(a,b|x,y) = \sum_{k=1}^m \lambda_kp_k^1(a|x)p_k^2(b|y), \ \ \ x\in X, y\in Y, a\in A, b\in B.$$
It is clear that, if $\Phi_k$ (resp. $\Psi_k$) is the $(X,A)$-classical (resp. $(Y,B)$-classical) channel 
corresponding to $p_1^k$ (resp. $p_2^k$) then 
$\Gamma_p = \sum_{k=1}^m \lambda_k\Phi_k\otimes \Psi_k$ and hence $\Gamma_p \in \cl Q_{\rm loc}$.
}
\end{remark}

If needed, we specify the dependence of $\cl Q_{\rm x}$ on the sets $X$, $Y$, $A$ and $B$
by using the notation 
$\cl Q_{\rm x}(X,Y,A,B)$, for ${\rm x}\in \{{\rm loc}, {\rm q}, {\rm qa}, {\rm qc}, {\rm ns}\}$.


\section{The operator system of a stochastic operator matrix}\label{s_uopsys}

Recall \cite{hestenes, zettl} that a \emph{ternary ring} is a complex vector space $\cl V$, equipped with a 
ternary operation $[\cdot,\cdot,\cdot] : \cl V\times \cl V\times \cl V\to \cl V$, linear on the outer variables and 
conjugate linear in the middle variable, such that 
$$[s,t,[u,v,w]] = [s,[v,u,t],w] = [[s,t,u],v,w], \ \ \ s,t,u,v,w\in \cl V.$$
A \emph{ternary representation} of $\cl V$ is a linear map $\theta: \cl V \to \cl B(H,K)$, for some 
Hilbert spaces $H$ and $K$, such that 
$$\theta\left([u,v,w]\right) = \theta(u)\theta(v)^*\theta(w), \ \ \ u,v,w \in \cl V.$$
We call $\theta$ \emph{non-degenerate} if
$\text{span}\{\theta(u)^*\eta : u\in \cl V, \eta\in K\}$ is dense in $H$. 
A \emph{concrete ternary ring of operators (TRO)} \cite{zettl} is a subspace $\cl U\subseteq \cl B(H,K)$ for some Hilbert spaces 
$H$ and $K$ such that $S,T,R\in \cl U$ implies $ST^*R\in \cl U$.

Let $X$ and $A$ be finite sets, and $\cl V^0_{X,A}$ be the 
ternary ring, 
generated by elements $v_{a,x}$,  $x\in X$, $a\in A$, satisfying the relations
\begin{equation}\label{eq_bc}
\sum_{a\in A} [v_{a'',x''},v_{a,x},v_{a,x'}] = \delta_{x,x'}v_{a'',x''}, \ \ \ x, x',x''\in X, a''\in A.
\end{equation}
Note that (\ref{eq_bc}) implies
\begin{equation}\label{eq_bcu}
\sum_{a\in A} [u,v_{a,x},v_{a,x'}] = \delta_{x,x'}u, \ \ \ x, x'\in X, u\in \cl V^0_{X,A}.
\end{equation}
Indeed, suppose that (\ref{eq_bcu}) holds for $u = u_i$, $i = 1,2,3$. 
Then 
$$\sum_{a\in A} [[u_1,u_2,u_3],v_{a,x},v_{a,x'}] 
= \sum_{a\in A} [u_1,u_2,[u_3,v_{a,x},v_{a,x'}]] = \delta_{x,x'}[u_1,u_2,u_3];$$
(\ref{eq_bcu}) now follows by induction.

Let $\theta : \cl V_{X,A}^0\to \cl B(H,K)$ be a non-degenerate ternary representation. 
Setting $V_{a,x} = \theta(v_{a,x})$, $x\in X$, $a\in A$, (\ref{eq_bcu}) implies
\begin{equation}\label{eq_bc3}
\sum_{a\in A} V_{a,x}^*V_{a,x'} = \delta_{x,x'}I_H, \ \ \ x,x'\in X;
\end{equation}
conversely, a family $\{V_{a,x} : x\in X, a\in A\}\subseteq \cl B(H,K)$ satisfying (\ref{eq_bc3})
clearly gives rise to a non-degenerate ternary representation $\theta : \cl V^0_{X,A}\to \cl B(H,K)$.
We therefore call such a family a \emph{representation of the relations} (\ref{eq_bc}). 
We note that the set of representations of (\ref{eq_bc}) is non-empty.
Indeed, consider isometries 
$V_x$, $x\in X$, with orthogonal ranges on some Hilbert space $H$, 
i.e. $V_x^*V_{x'} = \delta_{x, x'}I_H$, $x,x'\in X$. 
Fix $a_0\in A$ and let $V_{a,x} = \delta_{a,a_0}V_x$. 
Then $\sum_{a\in A} V_{a,x}^*V_{a,x'} = V_x^*V_{x'} = \delta_{x,x'}I_H$.

We note that (\ref{eq_bc3}) implies $\|V_{a,x}\|\leq 1$, $x\in X$, $a\in A$. 
We identify the family $\{V_{a,x} : a\in A, x\in X\}$ with the isometry $V = (V_{a,x})_{a,x}$ 
and write $H_V = H$, $K_V = K$ and $\theta_V = \theta$. 
Two representations $V = (V_{a,x})_{a,x}$ and $W = (W_{a,x})_{a,x}$
are called equivalent if there exist unitary operators $U_H : H_V\to H_W$ and $U_K : K_V\to K_W$
such that $W_{a,x}U_H = U_KV_{a,x}$, $x\in X$, $a\in A$. 

Write $\hat{\theta} = \oplus_V \theta_V$, where the direct sum is taken over all 
equivalence classes of representations of the relations (\ref{eq_bc}). 
For $u\in \cl V^0_{X,A}$, let 
$\|u\|_0 := \|\hat{\theta}(u)\|$. 
As $\|v_{a,x}\|\leq 1$ and $\cl V_{X,A}^0$ is generated by $v_{a,x}$, $a\in A$, $x\in X$, 
we have that $\|u\|_0<\infty$ for every $u\in \cl V^0_{X,A}$. 
It is also clear that $\|\cdot\|_0$ is a semi-norm on $\cl V_{X,A}^0$.
Set $N = \left\{u\in \cl V^0_{X,A} : \|u\|_0 = 0\right\}$. We have that $N$ is a ternary ideal of $\cl V^0_{X,A}$,
that is, $[u_1,u_2,u_3]\in N$ if $u_i\in N$ for some $i\in \{1,2,3\}$. 
The ternary product of $\cl V^0_{X,A}$ thus induces a ternary product on $\cl V^0_{X,A}/N$, 
and $\hat\theta$ induces a ternary representation of $\cl V^0_{X,A}/N$ that will be denoted in the
same way. 
Letting $\|u\| := \|\hat{\theta}(u)\|$, $u\in \cl V^0_{X,A}/N$, we have that $\|\cdot\|$ is a norm on 
$\cl V^0_{X,A}/N$, and hence $\cl V^0_{X,A}/N$ is a ternary pre-C*-ring (see \cite{zettl}). 
We let $\cl V_{X,A}$ be the completion of $\cl V^0_{X,A}/N$; thus, $\cl V_{X,A}$ is a ternary C*-ring \cite{zettl}. 
Note that $\hat{\theta}$ extends to a ternary representation of $\cl V_{X,A}$ (denoted in the same way)
onto a concrete TRO, 
and the equality $\|u\| = \|\hat{\theta}(u)\|$ continues to hold for every $u\in \cl V_{X,A}$. 
We thus have that $\cl V_{X,A}$ is a TRO in a canonical fashion.
It is clear that each $\theta_V$
induces a ternary representation of $\cl V_{X,A}$ onto a TRO, which will be denoted in the same way.

Let $\cl C_{X,A}$ be the \emph{right C*-algebra} of $\cl V_{X,A}$; 
if $\cl V_{X,A}$ is represented faithfully as a concrete 
ternary ring of operators in $\cl B(H,K)$ for some Hilbert spaces $H$ and $K$
(that is, $\cl V_{X,A}\cl V_{X,A}^*\cl V_{X,A} \subseteq \cl V_{X,A}$),
the C*-algebra $\cl C_{X,A}$ may be defined by letting
$$\cl C_{X,A} = \overline{{\rm span}\{S^*T : S,T\in \cl V_{X,A}\}}.$$ 
Each representation $V = (V_{a,x})_{a,x}$ of the relations (\ref{eq_bc})
gives rise \cite{hamana} to a unital *-representation $\pi_V$ of $\cl C_{X,A}$ on $H_V$ by letting 
$$\pi_V(S^*T) = \theta_V(S)^*\theta_V(T), \ \ \ S,T\in \cl V_{X,A}.$$

\begin{lemma}\label{l_rr}
The following hold true:
\begin{itemize}
\item[(i)] Every non-degenerate ternary representation of $\cl V_{X,A}$ has the form $\theta_V$, for some 
representation $V$ of the relations (\ref{eq_bc});
\item[(ii)] $\hat{\theta}$ is a faithful ternary representation of $\cl V_{X,A}$;
\item[(iii)] Every unital *-representation $\pi$ of $\cl C_{X,A}$ has the form $\pi_V$, for some 
representation $V$ of the relations (\ref{eq_bc}).
\end{itemize}
\end{lemma}

\begin{proof} 
(i) Suppose that $\theta$ is a non-degenerate ternary representation of $\cl V_{X,A}$.
Letting $q : \cl V_{X,A}^0\to \cl V_{X,A}$ be the quotient map, write $\theta_0 = \theta\circ q$; 
thus, $\theta_0$ is a non-degenerate ternary representation of $\cl V_{X,A}^0$. 
Letting $V$ be the representation of the relations (\ref{eq_bc}) such that $\theta_0 = \theta_V$, 
we have that $\theta = \theta_V$. 

(ii) follows from the fact that $\|\hat\theta(u)\|=\|u\|$, $u\in \cl V_{X,A}$. 

(iii) Let $\pi : \cl C_{X,A} \to \cl B(H)$ be a unital *-representation. 
Then there exists a ternary representation $\theta : \cl V_{X,A}\to \cl B(H,K)$
such that $\pi(S^*T) = \theta(S)^*\theta(T)$, $S,T\in \cl V_{X,A}$
(see e.g. \cite[Theorem 3.4]{blecher} and \cite[p. 1636]{elefkak}). 
Since $\pi$ is unital, $\theta$
is non-degenerate. By (i), there exists a representation 
$V$ of the relations (\ref{eq_bc}) such that 
$\theta = \theta_V$, and hence $\pi = \pi_V$.
\end{proof}

Set $e_{x,x',a,a'} = v_{a,x}^*v_{a',x'}  \in \cl C_{X,A}$, $x,x' \in X$, $a,a'\in A$. 
We call the operator subsystem
$$\cl T_{X,A} := {\rm span}\{e_{x,x',a,a'} : x,x'\in X, a,a'\in A\}$$
of $\cl C_{X,A}$ the \emph{Brown-Cuntz operator system}.
Note that relations (\ref{eq_bc}) imply
\begin{equation}\label{eq_relt}
\sum_{a\in A} e_{x,x',a,a} = \delta_{x,x'}1, \ \ \ x, x'\in X.
\end{equation}

\begin{theorem}\label{p_ucptau2}
Let $H$ be a Hilbert space and $\phi : \cl T_{X,A}\to \cl B(H)$ be a linear map.
The following are equivalent:
\begin{itemize}
\item[(i)] $\phi$ is a unital completely positive map;
\item[(ii)] $\left(\phi(e_{x,x',a,a'})\right)_{x,x',a,a'}$ is a stochastic operator matrix;
\item[(iii)] there exists a *-representation $\pi : \cl C_{X,A}\to \cl B(H)$ such that $\phi = \pi|_{\cl T_{X,A}}$. 
\end{itemize}
\noindent 
Moreover, if $\left(E_{x,x',a,a'}\right)_{x,x',a,a'}$ is a stochastic operator matrix acting on a Hilbert space $H$
then there exists a (unique) unital completely positive map 
$\phi : \cl T_{X,A}\to \cl B(H)$ such that $\phi(e_{x,x',a,a'}) = E_{x,x',a,a'}$ for all $x,x',a,a'$.
\end{theorem}

\begin{proof}
(i)$\Rightarrow$(ii)
By Arveson's Extension Theorem and Stinespring's Theorem, 
there exist a Hilbert space $K$, a *-representation $\pi : \cl C_{X,A} \to \cl B(K)$ 
and an isometry $W\in \cl B(H,K)$, such that 
$\phi(u) = W^*\pi(u) W$, $u\in \cl T_{X,A}$. 
By Lemma \ref{l_rr}, $\pi = \pi_V$ for some representation $V = (V_{a,x})_{a,x}$ of the relations (\ref{eq_bc}).
By the proof of Theorem \ref{p_coor}, $E :=\left(\pi(e_{x,x',a,a'})\right) \in (M_X\otimes M_A\otimes \cl B(K))^+$, and hence 
$$\left(\phi(e_{x,x',a,a'})\right)
= (I_X\otimes I_A\otimes W)^* E (I_X\otimes I_A\otimes W)
\in \left(M_X\otimes M_A\otimes B(H)\right)^+.$$
By (\ref{eq_relt}) and Theorem \ref{p_coor}, $\left(\phi(e_{x,x',a,a'})\right)_{x,x',a,a'}$ is a stochastic operator matrix.

(ii)$\Rightarrow$(iii)
By Theorem \ref{p_coor}, there exist a Hilbert space $K$ and 
an isometry $V = (V_{a,x})_{a,x}\in \cl B(H^X,K^A)$ such that
$$\phi(e_{x,x',a,a'}) = V_{a,x}^*V_{a',x'}, \ \ \ x,x'\in X, a,a'\in A.$$
The *-representation $\pi_V$ of $\cl C_{X,A}$ is an extension of $\phi$.

(iii)$\Rightarrow$(i) is trivial. 

\smallskip

Suppose that $E = \left(E_{x,x',a,a'}\right)_{x,x',a,a'}$ is a stochastic operator matrix acting on $H$. 
Letting $V$ be the isometry, associated with $E$ via Theorem \ref{p_coor}, 
we have that $\phi := \pi_V|_{\cl T_{X,A}}$ satisfies the required conditions. 
\end{proof}

Let $\cl S$ be an operator system. Recall that the pair $(C_u^*(\cl S),\iota)$ is called a
universal C*-cover of $\cl S$, if $C_u^*(\cl S)$ is a unital C*-algebra, $\iota : \cl S\to C_u^*(\cl S)$ is a 
unital complete order embedding, and whenever $H$ is a Hilbert space and 
$\phi : \cl S\to \cl B(H)$ is a unital completely positive map, there exists a *-representation 
$\pi_{\phi} : C^*_u(\cl S) \to \cl B(H)$ such that $\pi_{\phi} \circ \iota = \phi$. 
It is clear that the universal C*-cover is unique up to a *-isomorphism.
The following corollary is immediate from Theorem \ref{p_ucptau2}.

\begin{corollary}\label{c_ucst}
The pair $(\cl C_{X,A}, \iota)$, where $\iota$ is the inclusion map of $\cl T_{X,A}$ into $\cl C_{X,A}$, 
is the universal C*-cover of $\cl T_{X,A}$.
\end{corollary}

We will need the following slight extension of the equivalence (i)$\Leftrightarrow$(ii) of Theorem \ref{p_ucptau2}.

\begin{proposition}\label{c_ucptau2}
Let $H$ be a Hilbert space and $\phi : \cl T_{X,A}\to \cl B(H)$ be a linear map.
The following are equivalent:
\begin{itemize}
\item[(i)] $\phi$ is a completely positive map;
\item[(ii)] $\left(\phi(e_{x,x',a,a'})\right)_{x,x',a,a'} \in \left(M_X\otimes M_A\otimes \cl B(H)\right)^+$.
\end{itemize}
\end{proposition}

\begin{proof}
(i)$\Rightarrow$(ii) 
It follows from Theorem \ref{p_coor} and Lemma \ref{l_rr} (iii), 
with $\pi_V$ a faithful *-representation of $\cl C_{X,A}$,
 that $(e_{x,x',a,a'})\in (M_X\otimes M_A\otimes \cl C_{X,A})^+$. Since $\cl T_{X,A}\subseteq \cl C_{X,A}$ as 
an operator subsystem, we have 
\begin{equation}\label{eq_TC}
(e_{x,x',a,a'})_{x,x',a,a'} \in (M_X\otimes M_A\otimes \cl T_{X,A})^+
\end{equation}
and (ii) follows.

(ii)$\Rightarrow$(i)
Write $E = \left(\phi(e_{x,x',a,a'})\right)_{x,x',a,a'}$ and 
let $T = \phi(1)$. Since $1 = \sum_{a\in A} e_{x,x,a,a}$ (where $x$ is any element of $X$),
we have that $T\geq 0$. 
Note also that if $x,x'\in X$ and $x\neq x'$ then 
\begin{equation}\label{eq_Exx'}
\sum_{a\in A} E_{x,x',a,a} = \phi\left(\sum_{a\in A} e_{x,x',a,a}\right) = 0.
\end{equation}
Assume first that $T$ is invertible. 
Let $\psi : \cl T_{X,A} \to \cl B(H)$ be the map given by 
\begin{equation}\label{eq_Thalf}
\psi(u) = T^{-1/2} \phi(u) T^{-1/2}, \ \ \ u\in \cl T_{X,A}.
\end{equation}
Setting 
$F = \left(\psi(e_{x,x',a,a'})\right)_{x,x',a,a'}$, we have that 
$$F = \left(I_{XA} \otimes T^{-1/2}\right) E \left(I_{XA} \otimes T^{-1/2}\right) \geq 0.$$
Let $\omega = (\omega_{x,x'}) \in M_X$ and $\sigma$ be a state in $\cl T(H)$. 
Using (\ref{eq_Exx'}), we have
\begin{eqnarray*}
\left\langle \Tr\mbox{}_{A} L_\omega(F),  \sigma \right\rangle
& = & 
\left\langle F, \omega\otimes I_A \otimes \sigma \right\rangle\\
& = & 
\left\langle \left(I_{XA} \otimes T^{-1/2}\right) E \left(I_{XA} \otimes T^{-1/2}\right), \omega\otimes I_A \otimes \sigma \right\rangle\\
& = & 
\left\langle E, \omega\otimes I_A \otimes T^{-1/2} \sigma T^{-1/2}\right\rangle\\
& = & 
\sum_{x,x'\in X}\sum_{a\in A} \omega_{x,x'} \left\langle E_{x,x',a,a}, T^{-1/2} \sigma T^{-1/2}\right\rangle\\
& = & 
\sum_{x\in X}\sum_{a\in A} \omega_{x,x} \left\langle E_{x,x,a,a}, T^{-1/2} \sigma T^{-1/2}\right\rangle\\
& = & 
\sum_{x\in X} \omega_{x,x} \left\langle T, T^{-1/2} \sigma T^{-1/2}\right\rangle
= 
\sum_{x\in X} \omega_{x,x}  = \Tr(\omega).
\end{eqnarray*} 
By Theorem \ref{p_coor}, $F$ is a stochastic operator matrix; by Theorem \ref{p_ucptau2}, $\psi$ is completely positive.
Since 
$\phi(\cdot) = T^{1/2} \psi(\cdot) T^{1/2}$, so is $\phi$. 

Now relax the assumption that $T$ be invertible. For every $\epsilon > 0$, let 
$\phi_{\epsilon} : \cl T_{X,A}\to \cl B(H)$ be the map, given by $\phi_{\epsilon}(u) = \phi(u) + \epsilon I$. 
By the previous paragraph, $\phi_{\epsilon}$ is completely positive. Since $\phi = \lim_{\epsilon \to 0} \phi_{\epsilon}$
in the point-norm topology, $\phi$ is completely positive. 
\end{proof}

Let 
$$\cl L_{X,A} = \left\{(\lambda_{x,x',a,a'})\in M_{XA} : \exists \ c\in \bb{C} \mbox{ s.t.} 
\sum_{a\in A} \lambda_{x,x',a,a} = \delta_{x,x'} c, \ x,x'\in X\right\};$$
we consider $\cl L_{X,A}$ as an operator subsystem of $M_{XA}$. 
For the next proposition, note that, by \cite[Corollary 4.5]{CE2}, if 
$\cl T$ is a finite dimensional operator system than its 
(matrix ordered) dual $\cl T^{\rm d}$ is an operator system, when equipped with any 
faithful state of $\cl T$ as an Archimedean order unit. It is straightforward to verify that, in this case, 
$\cl T^{\rm dd}\cong_{\rm c.o.i.} \cl T$.

\begin{proposition}\label{p_dualt}
The linear map $\Lambda : \cl T_{X,A}^{\rm d} \to \cl L_{X,A}$, given by 
\begin{equation}\label{p_Lam}
\Lambda(\phi) = \left(\phi(e_{x,x',a,a'})\right)_{x,x'\in X, a,a'\in A}
\end{equation}
is a well-defined complete order isomorphism. 
\end{proposition}

\begin{proof}
By Proposition \ref{c_ucptau2}, if $\phi\in \cl T_{X,A}\to \bb{C}$ is a positive functional then 
$\left(\phi(e_{x,x',a,a'})\right)_{x,x',a,a'}\in \cl L_{X,A}^+$. 
Thus, the map $\Lambda_+ : \left(\cl T_{X,A}^{\rm d}\right)^+ \to \cl L_{X,A}^+$, given by 
$$\Lambda_+(\phi) = \left(\phi(e_{x,x',a,a'})\right)_{x,x',a,a'}, \ \ \ \phi\in \left(\cl T_{X,A}^{\rm d}\right)^+,$$
is well-defined. It is clear that $\Lambda_+$ is additive and 
$$\Lambda_+(\lambda \phi) = \lambda\Lambda_+(\phi), \ \ \lambda \geq 0, \ \phi\in \left(\cl T_{X,A}^{\rm d}\right)^+.$$
Suppose that $\phi\in \cl T_{X,A}^{\rm d}$ is a hermitian functional. 
If $\phi = \phi_1 - \phi_2$, where $\phi_1$ and $\phi_2$ are positive functionals on $\cl T_{X,A}$, 
set 
$$\Lambda_h(\phi) = \Lambda_+(\phi_1) - \Lambda_+(\phi_2).$$
The map $\Lambda_h : \left(\cl T_{X,A}\right)_h^{\rm d} \to \cl L_{X,A}$ is well-defined: 
if $\phi_1 - \phi_2 = \psi_1 - \psi_2$, where $\phi_1$, $\phi_2$, $\psi_1$ and $\psi_2$ are positive functionals
then, by the additivity of $\Lambda_+$, we have that 
$\Lambda_+(\phi_1) + \Lambda_+(\psi_2) = \Lambda_+(\psi_1) + \Lambda_+(\phi_2)$, that is,
$\Lambda_+(\phi_1) -  \Lambda_+(\phi_2) = \Lambda_+(\psi_1) - \Lambda_+(\psi_2)$.
It is straightforward that the map $\Lambda_h$ is $\bb{R}$-linear, and thus it extends to a 
($\bb{C}$-)linear map $\Lambda : \cl T_{X,A}^{\rm d}\to \cl L_{X,A}$. 

Suppose that $(\phi_{i,j})_{i,j=1}^m\in M_m\left(\cl T_{X,A}^{\rm d}\right)^+$; thus, the map 
$\Phi : \cl T_{X,A}\to M_m$, given by 
$\Phi(u) = (\phi_{i,j}(u))_{i,j=1}^m$, is completely positive. 
By Proposition \ref{c_ucptau2},
$(\phi_{i,j}(e_{x,x',a,a'}))_{i,j} \in \left(M_{XA}\otimes M_m\right)^+$.
This shows that $\Lambda$ is completely positive. 

If $\Lambda(\phi) = 0$ then $\phi(e_{x,x',a,a'}) = 0$ for all $x,x'\in X$ and all $a,a'\in A$, 
implying 
$\phi = 0$; thus, $\Lambda$ is injective. 
Since $\cl L_{X,A}$ is an operator subsystem of $M_{XA}$, 
it is spanned by the positive matrices it contains. 
Using Theorem \ref{p_ucptau2}, we see that 
every positive element of $\cl L_{X,A}$
is in the range of $\Lambda$; it follows that $\Lambda$ is surjective.  

Finally, suppose that 
$\phi_{i,j}\in \cl T_{X,A}^{\rm d}$, $i,j=1,\dots,m$, are such that the matrix 
$\left(\Lambda(\phi_{i,j})\right)_{i,j=1}^m$ is a positive element of $M_m\left(\cl L_{X,A}\right)$.
Let $\Phi : \cl T_{X,A}\to M_m$ be given by $\Phi(u) = (\phi_{i,j}(u))_{i,j=1}^m$.
Then $\left(\Phi(e_{x,x',a,a'})\right)\in M_m\left(\cl L_{X,A}\right)^+$. 
By Proposition \ref{c_ucptau2}, $\Phi$ is completely positive, that is, 
$\left(\phi_{i,j}\right)_{i,j=1}^m\in M_m\left(\cl T_{X,A}^{\rm d}\right)^+$.
Thus, $\Lambda^{-1}$ is completely positive, and the proof is complete. 
\end{proof}

Let $\cl S$ be an operator system. A \emph{kernel} in $\cl S$ \cite{kptt_adv}
is a linear subspace $J\subseteq \cl S$, for which there exists an operator system $\cl T$ and a 
unital completely positive map $\psi : \cl S\to \cl T$ such that $J = \ker(\psi)$. 
If $J$ is a kernel in $\cl S$, the quotient space $\cl S/J$ can be equipped with a unique
operator system structure 
with the property that, 
whenever $\cl T$ is an operator system and 
$\phi : \cl S\to \cl T$ is a completely positive map annihilating $J$, the induced map 
$\tilde{\phi} : \cl S/J\to \cl T$ is completely positive. 
If $\cl T$ is an operator system, a surjective map $\phi : \cl S\to \cl T$ is called 
\emph{complete quotient}, if the map $\tilde{\phi}$ is a complete order isomorphism.
We refer the reader to \cite{kptt_adv} for further details.

Let 
$$J_{X,A} = 
\{(\mu_{x,x',a,a'}) \in M_{XA} : 
\mu_{x,x',a,a'} = 0  \mbox{ and } \mu_{x,x',a,a} = \mu_{x,x',a',a'}, a \neq a',$$
$$\mbox{ and } \sum_{x\in X}\mu_{x,x,a,a} = 0, a\in A\}.$$

\begin{corollary}\label{c_quotxa}
The space $J_{X,A}$ is a kernel in $M_{XA}$ and the operator system $\cl T_{X,A}$ 
is completely order isomorphic to the quotient $M_{XA}/J_{X,A}$.
\end{corollary}

\begin{proof}
By Proposition \ref{p_dualt}, the map
$\Lambda : \cl T_{X,A}^{\rm d} \to M_{XA}$ is a complete order embedding. 
By \cite[Proposition 1.8]{fp}, the dual $\Lambda^* :  M_{XA}^{\rm d} \to \cl T_{X,A}$ 
of $\Lambda$ is a complete quotient map. 
Identifying $M_{XA}^{\rm d}$ with $M_{XA}$ canonically, 
for an element $f\in M_{XA}^{\rm d}$, we have 
\begin{eqnarray*}
\Lambda^*(f) = 0 
& \Longleftrightarrow & 
\left\langle \Lambda^*(f), \phi \right\rangle = 0 \ \mbox{ for all } \phi\in \cl T_{X,A}\\
& \Longleftrightarrow & 
\left\langle f, \Lambda(\phi) \right\rangle = 0 \ \mbox{ for all } \phi\in \cl T_{X,A}\\
& \Longleftrightarrow & 
\left\langle f, T \right\rangle = 0 \ \mbox{ for all } T \in \cl L_{X,A}\\
& \Longleftrightarrow &
f\in J_{X,A}.
\end{eqnarray*}
Thus, $\ker(\Lambda^*) = J_{X,A}$.
\end{proof}


\section{Descriptions via tensor products}\label{s_dtp}

In this section, we provide a description of the classes of QNS correlations, introduced in Section \ref{s_cqnsc},
analogous to the description of the classes of NS correlations given in \cite{lmprsstw} 
(see also \cite{fkpt} and \cite{pt_QJM}). 
We will use the tensor theory of operator systems developed in \cite{kptt}.
If $\cl S$ and $\cl T$ are operator systems,
$\cl S\otimes_{\min}\cl T$ denotes
the minimal tensor product of $\cl S$ and $\cl T$: if 
$\cl A$ and $\cl B$ are unital C*-algebras, 
$\cl A\otimes_{\min} \cl B$ is the spatial tensor product of $\cl A$ and $\cl B$, and 
$\cl S\subseteq_{\rm c.o.i.} \cl A$ and $\cl T\subseteq_{\rm c.o.i.} \cl B$, then 
$\cl S\otimes_{\min}\cl T\subseteq_{\rm c.o.i.} \cl A\otimes_{\min}\cl B$. 
The commuting tensor product $\cl S\otimes_{\rm c} \cl T$ sits completely order isomorphically 
in the maximal tensor product $C^*_u(\cl S) \otimes_{\max} C^*_u(\cl T)$ of the universal C*-covers of 
$\cl S$ and $\cl T$, while 
the maximal tensor product $\cl S\otimes_{\max} \cl T$ is characterised by the property that it linearises
jointly completely positive maps $\theta : \cl S\times \cl T\to \cl B(H)$.
We refer the reader to \cite{kptt} for more details and further background.

Let $X$, $Y$, $A$ and $B$ be finite sets.
As in Section \ref{s_uopsys}, we write 
$e_{x,x',a,a'}$, $x,x'\in X$, $a,a'\in A$, for the canonical generators of $\cl T_{X,A}$. 
Similarly, we write $f_{y,y',b,b'}$, $y,y'\in Y$, $b,b'\in B$, for the canonical generators of $\cl T_{Y,B}$.
Given a linear functional $s : \cl T_{X,A} \otimes \cl T_{Y,B} \to \bb{C}$
(or a linear functional $s : \cl C_{X,A} \otimes \cl C_{Y,B} \to \bb{C}$), we let
$\Gamma_s : M_{XY}\to M_{AB}$ be the linear map given by 
\begin{equation}\label{eq_gammas}
\Gamma_s\left(e_{x} e_{x'}^*\otimes e_{y} e_{y'}^*\right)
= \sum_{a,a'\in A}\sum_{b,b'\in B} s\left(e_{x,x',a,a'} \otimes f_{y,y',b,b'}\right) e_ae_{a'}^*\otimes e_be_{b'}^*.
\end{equation}

\begin{remark}\label{r_ling}
{\rm The correspondence $s\to \Gamma_s$ is a linear map 
from the dual $(\cl T_{X,A} \otimes \cl T_{Y,B})^{\rm d}$ of 
$\cl T_{X,A} \otimes \cl T_{Y,B}$ into the space 
$\cl L(M_{XY},M_{AB})$ of all linear transformations 
from $M_{XY}$ into $M_{AB}$.}
\end{remark}

\begin{theorem}\label{th_qnsma}
Let $X,Y,A,B$ be finite sets and $\Gamma : M_{XY}\to M_{AB}$ be a linear map. The following are equivalent:
\begin{itemize}
\item[(i)] $\Gamma$ is a QNS correlation;
\item[(ii)] there exists a state $s : \cl T_{X,A} \otimes_{\max} \cl T_{Y,B} \to \bb{C}$ such that $\Gamma = \Gamma_s$.
\end{itemize}
\end{theorem}

\begin{proof}
(i)$\Rightarrow$(ii) 
Let $\Gamma : M_{XY}\to M_{AB}$ be a QNS correlation and write 
$$\Gamma\left(e_x e_{x'}^*\otimes e_y e_{y'}^*\right)
= \sum_{a,a'\in A}\sum_{b,b'\in B} C^{x,x',y,y'}_{a,a',b,b'} e_ae_{a'}^*\otimes e_be_{b'}^*,$$
for some $C^{x',x,y',y}_{a,a',b,b'}\in \bb{C}$, $x,x'\in X$, $y,y'\in Y$, $a,a'\in A$, $b,b'\in B$.
It follows from (\ref{eq_qns1}) and (\ref{eq_qns2}) that 
the Choi matrix $C:= \left(C^{x,x',y,y'}_{a,a',b,b'}\right)$ of $\Gamma$ satisfies the 
following conditions (see also \cite{dw}): 
\begin{itemize}
\item[(a)] $C\in M_{XYAB}^+$;
\item[(b)] there exists $c_{b,b'}^{y,y'} \in \bb{C}$ such that 
$\sum_{a\in A} C^{x,x',y,y'}_{a,a,b,b'} = \delta_{x,x'} c_{b,b'}^{y,y'}$, $y,y'\in Y$, $b,b'\in B$;
\item[(c)] there exists $d_{a,a'}^{x,x'}\in \bb{C}$ such that
$\sum_{b\in B} C^{x,x',y,y'}_{a,a',b,b} = \delta_{y,y'} d_{a,a'}^{x,x'}$, $x,x'\in X$, $a,a'\in A$.
\end{itemize}
By condition (b), $L_{\omega_{YB}}(C)\in \cl L_{X,A}$ for every $\omega_{YB}\in M_{YB}$, while 
by condition (c), $L_{\omega_{XA}}(C)\in \cl L_{Y,B}$ for every $\omega_{XA}\in M_{XA}$. 
Thus, 
$$C\in \left(\cl L_{X,A}\otimes\cl L_{Y,B}\right)\cap M_{XYAB}^+;$$
by the injectivity of the minimal operator system tensor product, 
$C\in \left(\cl L_{X,A}\otimes_{\min}\cl L_{Y,B}\right)^+$.

By \cite[Proposition 1.9]{fp} and Proposition \ref{p_dualt}, 
\begin{equation}\label{eq_PXA}
\left(\cl T_{X,A}\otimes_{\max}\cl T_{Y,B}\right)^{\rm d} \cong_{\rm c.o.i.} \cl L_{X,A}\otimes_{\min}\cl L_{Y,B},
\end{equation}
via the identification $\Lambda$ given by (\ref{p_Lam}). 
The state $s$ of $\cl T_{X,A}\otimes_{\max}\cl T_{Y,B}$ corresponding to $C$ 
via (\ref{eq_PXA}) satisfies 
\begin{equation}\label{eq_Cs}
C^{x,x',y,y'}_{a,a',b,b'} = s\left(e_{x,x',a,a'}\otimes f_{y,y',b,b'}\right), \ \ x,x'\in X, y,y'\in Y, a,a'\in A, b,b'\in B.
\end{equation}
Thus, $\Gamma = \Gamma_s$.

(ii)$\Rightarrow$(i)
Let $s$ be a state of $\cl T_{X,A}\otimes_{\max}\cl T_{Y,B}$, and define $C^{x,x',y,y'}_{a,a',b,b'}$ via 
(\ref{eq_Cs}); thus, $C$ is the Choi matrix of $\Gamma_s$. 
By (\ref{eq_TC}) and and the definition of the maximal tensor procuct,
$$\left(e_{x,x',a,a'} \otimes f_{y,y',b,b'}\right)\in M_{XYAB}\left(\cl T_{X,A} \otimes_{\max} \cl T_{Y,B}\right)^+,$$
and hence the matrix $C:= \left(C^{x,x',y,y'}_{a,a',b,b'}\right)$ is positive; 
by Choi's Theorem, $\Gamma_s$ is completely positive. 
Relations (\ref{eq_relt}) imply that $\Gamma_s$ is trace preserving and
that conditions (b) and (c) hold. 
Suppose that $\rho_X = (\rho_{x,x'})_{x,x'} \in M_X$ has zero trace and 
$\rho_Y = (\rho^{y,y'})_{y,y'} \in M_Y$. 
We have 
\begin{eqnarray*}
& & \sum_{x,x'\in X} \sum_{y,y'\in Y} \sum_{b,b'\in B} \sum_{a\in A} C^{x,x',y,y'}_{a,a,b,b'} 
\rho_{x,x'} \rho^{y,y'} e_b e_{b'}^*\\
& = & 
\left(\sum_{x\in X} \rho_{x,x}\right) \sum_{y,y'\in Y} \sum_{b,b'\in B}   \rho^{y,y'}  c_{b,b'}^{y,y'} e_b e_{b'}^* = 0,
\end{eqnarray*} 
that is, (\ref{eq_qns1}) holds; similarly, (c) implies (\ref{eq_qns2}). 
\end{proof}

\begin{theorem}\label{t_dqc}
Let $X,Y,A,B$ be finite sets and $\Gamma : M_{XY}\to M_{AB}$ be a linear map. The following are equivalent:
\begin{itemize}
\item[(i)] $\Gamma$ is a quantum commuting QNS correlation;
\item[(ii)] there exists a state $s : \cl T_{X,A} \otimes_{\rm c} \cl T_{Y,B} \to \bb{C}$ such that $\Gamma = \Gamma_s$;
\item[(iii)] there exists a state $s : \cl C_{X,A} \otimes_{\max} \cl C_{Y,B} \to \bb{C}$ such that $\Gamma = \Gamma_s$.
\end{itemize}
\end{theorem}

\begin{proof}
(i)$\Rightarrow$(iii)
Let $H$ be a Hilbert space, 
$E\in M_X\otimes M_A\otimes \cl B(H)$, $F\in M_Y\otimes M_B\otimes \cl B(H)$ 
form a commuting pair of stochastic operator matrices, 
and $\tau\in\cl T(H)^+$ be such that $\Gamma = \Gamma_{E\cdot F,\tau}$. 
By Theorem \ref{p_ucptau2}, there exist 
representations $\pi_X$ and $\pi_Y$ of 
$\cl C_{X,A}$ and $\cl C_{Y,B}$, respectively, such that 
$E_{x,x',a,a'} = \pi_X(e_{x,x',a,a'})$ and $F_{y,y',b,b'} = \pi_Y(e_{y,y',b,b'})$
for all $x,x'\in X$, $y,y'\in Y$, $a,a'\in A$, $b,b'\in B$.
Since $\cl C_{X,A}$ (resp. $\cl C_{Y,B}$) is generated by 
the elements $e_{x,x',a,a'}$, $x,x'\in X$, $a,a'\in A$ (resp. $f_{y,y',b,b'}$, $y,y'\in Y$, $b,b'\in B$), 
$\pi_X$ and $\pi_Y$ have commuting ranges.
Let $\pi_X\times\pi_Y$ be the (unique) *-representation $\cl C_{X,A}\otimes_{\max}\cl C_{Y,B}\to \cl B(H)$ 
such that $(\pi_X\times\pi_Y)(u\otimes v) = \pi_X(u)\pi_Y(v)$, $u\in \cl C_{X,A}$, $v\in \cl C_{Y,B}$. 
By (\ref{eq_EcdotF}), 
\begin{eqnarray*}
& & \left\langle\Gamma_{E\cdot F,\tau}(e_{x} e_{x'}^*\otimes e_{y} e_{y'}^*),e_{a} e_{a'}^*\otimes e_{b}e_{b'}^*\right\rangle\\
& & = \left\langle E_{x,x',a,a'}F_{y,y',b,b'},\tau\right\rangle
= \left\langle(\pi_X\times\pi_Y)(e_{x,x',a,a'}\otimes f_{y,y',b,b'}),\tau\right\rangle.
\end{eqnarray*}
Letting $s(w) = \langle(\pi_X\times\pi_Y)(w),\tau\rangle$, 
$w\in \cl C_{X,A}\otimes_{\max}\cl C_{Y,B}$, we have $\Gamma = \Gamma_s$.

(iii)$\Rightarrow$(i) Let $s$ be a state on $\cl C_{X,A}\otimes_{\max}\cl C_{Y,B}$ and write 
$\pi_s$ and $\xi_s$ for the corresponding  GNS representation of $\cl C_{X,A}\otimes_{\max}\cl C_{Y,B}$ and  
for its cyclic vector, respectively. Then $E:=(\pi_s(e_{x,x',a,a'}\otimes 1))_{x,x',a,a'}$ and 
$F:=(\pi_s(1\otimes f_{y,y',b,b'})_{y,y',b,b'}$ form a commuting pair of
stochastic operator matrices; moreover, 
for $x,x'\in X$, $y,y'\in Y$, $a,a'\in A$ and $b,b'\in B$, we have 
\begin{eqnarray*}
\left\langle\Gamma_s(e_{x}e_{x'}^*\otimes e_{y}e_{y'}^*), e_{a}e_{a'}^*\otimes e_{b}e_{b'}^*\right\rangle
& = & 
s(e_{x,x',a,a'}\otimes f_{y,y',b,b'})\\ 
& = & 
\langle \pi_s(e_{x,x',a,a'}\otimes f_{y,y',b,b'})\xi_s,\xi_s\rangle\\
& = & 
\langle E_{x,x',a,a'} F_{y,y',b,b'}\xi_s,\xi_s\rangle\\
& = & 
\left\langle\Gamma_{E,F,\xi_s}(e_{x} e_{x'}^*\otimes e_{y} e_{y'}^*), e_{a}e_{a'}^*\otimes e_{b} e_{b'}^*\right\rangle.
\end{eqnarray*}

(ii)$\Leftrightarrow$(iii) 
By Corollary \ref{c_ucst} and \cite[Theorem 6.4]{kptt}, 
$\cl T_{X,A}\otimes_{\rm c} \cl T_{Y,B}$ sits completely order isomorphically in $\cl C_{X,A}\otimes_{\max} \cl C_{Y,B}$;
thus the states of $\cl T_{X,A}\otimes_{\rm c} \cl T_{Y,B}$ are precisely the restrictions of the states
of $\cl C_{X,A}\otimes_{\max} \cl C_{Y,B}$.
\end{proof}

\begin{corollary}\label{c_closed}
The set $\cl Q_{\rm qc}$ is closed and convex.
\end{corollary}

\begin{proof}
By Theorem \ref{t_dqc} and Remark \ref{r_ling}, the map $s\to \Gamma_s$ is an affine bijection 
from the state space of $\cl T_{X,A} \otimes_{\rm c} \cl T_{Y,B}$ onto $\cl Q_{\rm qc}$. 
It is straightforward that it is also a homeomorphism, 
when its domain is equipped with the weak* topology.
Since the state space of $\cl T_{X,A} \otimes_{\rm c} \cl T_{Y,B}$ is 
weak* compact, its range is (convex and) closed. 
\end{proof}

\begin{theorem}\label{t_dqa}
Let $X,Y,A,B$ be finite sets and $\Gamma : M_{XY}\to M_{AB}$ be a linear map. The following are equivalent:
\begin{itemize}
\item[(i)] $\Gamma$ is an approximately quantum QNS correlation;

\item[(ii)] there exists a state $s : \cl T_{X,A} \otimes_{\min} \cl T_{Y,B} \to \bb{C}$ such that $\Gamma = \Gamma_s$;

\item[(iii)] there exists a state $s : \cl C_{X,A} \otimes_{\min} \cl C_{Y,B} \to \bb{C}$ such that $\Gamma = \Gamma_s$.
\end{itemize}
\end{theorem}
\begin{proof}
The proof is along the lines of the proof of \cite[Theorem 2.8]{pt_QJM};
we include the details for the convenience of the reader.

(iii)$\Rightarrow$(i) 
Let $\pi_X : \cl C_{X,A}\to\cl B(H_X)$ and $\pi_Y : \cl C_{Y,B}\to\cl B(H_Y)$ be faithful *-representations. 
Then $\pi_X\otimes\pi_Y : \cl C_{X,A} \otimes_{\min} \cl C_{Y,B}\to \cl B(H_X\otimes H_Y)$ 
is a faithful *-representation of $\cl C_{X,A} \otimes_{\min} \cl C_{Y,B}$. 
Let $s$ be a state satisfying (iii).  
By \cite[Corollary 4.3.10]{kadison-ringrose}, $s$ can be approximated in the 
weak* topology by elements of the convex hull of  vector states on  
$(\pi_X\otimes\pi_Y)(\cl C_{X,A} \otimes_{\min} \cl C_{Y,B})$;
thus, given $\varepsilon > 0$, 
there exist unit vectors 
$\xi_1,\ldots,\xi_n\in H_X\otimes H_Y$ and positive scalars $\lambda_1,\ldots,\lambda_n$ with $\sum_{i=1}^n\lambda_i=1$ such that
$$\left|s(e_{x,x',a,a'}\otimes f_{y,y',b,b'})
- \sum_{i=1}^n\lambda_i\left\langle \left(\pi_X(e_{x,x',a,a'})\otimes \pi_Y(f_{y,y',b,b'})\right)\xi_i,\xi_i\right\rangle\right|
< \varepsilon,$$
for all $x,x'\in X$, $y,y'\in Y$, $a,a'\in A$ and $b,b'\in B$.
Let $\xi = \oplus_{i=1}^n\sqrt{\lambda_i}\xi_i\in \mathbb C^{n}\otimes(H_X\otimes H_Y)$; then $\|\xi\| = 1$.  
Set $E_{x,x',a,a'} = I_n \otimes \pi_X(e_{x,x',a,a'})$ and $F_{y,y',b,b'} = \pi_Y(f_{y,y',b,b'})$. 
Then $(E_{x,x'a,a'})_{x,x',a,a'}$ (resp. $(F_{y,y',b,b'})_{y,y',b,b'}$)
is a stochastic operator matrix on $\bb{C}^n\otimes H_X$ (resp. $H_Y$), and 
$$\left|s\left(e_{x,x',a,a'}\otimes f_{y,y',b,b'}\right) 
- \left\langle E_{x,x'a,a'}\otimes F_{y,y',b,b'}\xi,\xi\right\rangle\right| < \varepsilon.$$

It follows that $\Gamma_s$ is in the closure of the set of correlations of the form 
$\Gamma_{E\odot F,\xi}$, where $E$ and $F$ act on, possibly infinite dimensional, Hilbert spaces 
$H$ and $K$. 
Given such a correlation $\Gamma_{E\odot F,\xi}$, let 
$(P_{\alpha})_{\alpha}$ (resp. $(Q_{\beta})_{\beta}$) be a net of finite rank projections on $H$ (resp. $K$)
such that $P_{\alpha} \to_{\alpha} I_H$ (resp. $Q_{\beta} \to_{\beta} I_K$) in the strong operator topology. 
Set $H_{\alpha} = P_{\alpha}H$ (resp. $K_{\beta} = Q_{\beta}K$), 
$E_{\alpha} = (I_X\otimes I_A\otimes P_{\alpha})E(I_X\otimes I_A\otimes P_{\alpha})$ 
(resp. $F_{\beta} = (I_Y\otimes I_B\otimes Q_{\beta})F(I_Y\otimes I_B\otimes Q_{\beta})$), 
and $\xi_{\alpha,\beta} = \frac{1}{\|(P_{\alpha}\otimes Q_{\beta})\xi\|}(P_{\alpha}\otimes Q_{\beta})\xi$
(note that $\xi_{\alpha,\beta}$ is eventually well-defined). 
Then $E_{\alpha}$ (resp. $F_{\beta}$) is a stochastic operator matrix acting on $H_{\alpha}$ (resp. $K_{\beta}$), 
and $\Gamma_{E_{\alpha}\odot F_{\beta},\xi_{\alpha,\beta}} \to_{(\alpha,\beta)} \Gamma_{E\odot F,\xi}$
along the product net. 
It follows that $\Gamma_s \in \cl Q_{\rm qa}$.

(i)$\Rightarrow$(iii) 
Given $\varepsilon > 0$, 
let $E$ and $F$ be stochastic operator matrices acting on finite dimensional Hilbert spaces 
$H_X$ and $H_Y$, respectively, 
and $\xi \in H_X\otimes H_Y$ be a unit vector, such that 
$$\left|\left\langle\Gamma(e_{x}e_{x'}^*\otimes e_{y}e_{y'}^*), e_{a}e_{a'}\otimes e_{b}e_{b'}^*\right\rangle
- \left\langle \left(E_{x,x',a,a'}\otimes F_{y,y',b,b'}\right)\xi,\xi\right\rangle\right| < \varepsilon,$$
for all $x,x'\in X$, $y,y'\in Y$, $a,a'\in A$ and $b,b'\in B$.
By Lemma \ref{l_rr}, there exists a *-representation $\pi_X$ (resp. $\pi_Y$) 
of $\cl C_{X,A}$ (resp. $\cl C_{Y,B}$) on $H_X$ (resp. $H_Y$) such that
$E_{x,x',a,a'} = \pi_X(e_{x,x',a,a'})$ (resp. $F_{y,y',b,b'} = \pi_Y(f_{y,y',b,b'})$), 
$x,x'\in X$, $a,a'\in A$ (resp. $y,y'\in Y$, $b,b'\in B$).
Let $s_\varepsilon$ be the state on $\cl C_{X,A}\otimes_{\min}\cl C_{Y,B}$ given by
$$s_\varepsilon\left(u\otimes v\right) = \left\langle \left(\pi_X(u)\otimes \pi_Y(v)\right)\xi,\xi\right\rangle,$$
and $s$ be a cluster point of the sequence $\{s_{1/n}\}_n$ in the weak* topology. Then
\begin{eqnarray*}
s\left(e_{x,x',a,a'}\otimes f_{y,y',b,b'}\right)
& =& 
\lim_{n\to\infty}s_{1/n}\left(e_{x,x',a,a'}\otimes f_{y,y',b,b'}\right)\\
& =& 
\left\langle\Gamma(e_{x} e_{x'}^*\otimes e_{y}e_{y'}^*), e_{a}e_{a'}\otimes e_{b}e_{b'}^* \right\rangle,
\end{eqnarray*}
giving $\Gamma = \Gamma_s$. 

(ii)$\Leftrightarrow$(iii) follows from the fact that 
$\cl T_{X,A} \otimes_{\min} \cl T_{Y,B} \subseteq_{\rm c.o.i} \cl C_{X,A} \otimes_{\min} \cl C_{Y,B}.$
\end{proof}

Recall \cite{ptt} that, given any Archimedean ordered unit (AOU) space $V$, 
there exists a (unique) operator system $\omin(V)$ (resp. $\omax(V)$) with 
underlying space $V$, 
called the \emph{minimal operator system} (resp. the \emph{maximal operator system}) of $V$
that has the property that 
every positive map from an operator system $\cl T$ into $V$ 
(resp. from $V$ into an operator system $\cl T$) 
is automatically completely positive as a map
from $\cl T$ into $\omin(V)$ (resp. from $\omax(V)$ into $\cl T$). 
If $V$ is in addition an operator system, we denote by 
$\omin(V)$ (resp. $\omax(V)$) the minimal (resp. maximal) operator system of the AOU space, underlying $V$.

\begin{lemma}\label{l_fdom}
Let $V$ and $W$ be finite dimensional AOU spaces with units $e$ and $f$, respectively. 
An element $u\in \omax(V)\otimes_{\max} \omax(W)$ is positive if and only if 
$u = \sum_{i=1}^k v_i \otimes w_i$,
for some $v_i\in V^+$, $w_i\in W^+$, $i = 1,\dots,k$.
\end{lemma}

\begin{proof}
Let $D$ be the set of all sums of elementary tensors $v\otimes w$ with 
$v \in V^+$ and $w\in W^+$. 
We claim that if, for every $\epsilon > 0$, there exists $u_{\epsilon}\in D$ such that
$\|u_{\epsilon}\| \to_{\epsilon\to 0} 0$  and 
$u + u_{\epsilon} \in D$ for every $\epsilon > 0$, then $u\in D$. 
Assume, without loss of generality, that $\|u_{\epsilon}\|\leq 1$ for all $\epsilon > 0$. 
Set $L = 2\dim(V)\dim(W) + 1$ and, using Carath\'{e}odory's Theorem, 
write 
$$
u + u_{\epsilon} =  \sum_{j=1}^L v_j^{(\epsilon)} \otimes w_j^{(\epsilon)},$$
where 
$v_j^{(\epsilon)} \in V^+$,
$w_j^{(\epsilon)} \in W^+$ and $\|v_j^{(\epsilon)}\| = \|w_j^{(\epsilon)}\|$
for all $j = 1,\dots,L$ and all $\epsilon > 0$. 
Since $v_j^{(\epsilon)} \otimes w_j ^{(\epsilon)} \leq u + u_{\epsilon}$ and $\|u + u_{\epsilon}\|\leq \|u\| + 1$ for all $\epsilon > 0$, 
we have $\|v_j^{(\epsilon)}\| \leq \sqrt{\|u\| + 1}$ and $\|w_j^{(\epsilon)}\| \leq \sqrt{\|u\| + 1}$, $j = 1,\dots,L$.
By compactness, we may assume that 
$v_j^{(\epsilon)} \to_{\epsilon \to 0} v_j$ and 
$w_j^{(\epsilon)} \to_{\epsilon \to 0} w_j$ for all $j = 1,\dots,L$.
It follows that $u = \sum_{j=1}^L v_j \otimes w_j \in D$.

Let
\begin{equation}\label{eq_T0}
S_0 = \sum_{p = 1}^l a_p \otimes v_p \ \mbox{ and } \ T_0 = \sum_{q = 1}^r b_q \otimes w_q,
\end{equation}
for some $a_p\in M_n$, $v_p\in V^+$, $p = 1,\dots, l$, and $b_q\in M_m^+$, $w_q\in W^+$, $q = 1,\dots, r$.
If $\alpha\in M_{1,nm}$ then 
$$\alpha(S_0\otimes T_0)\alpha^* = \sum_{p=1}^l \sum_{q=1}^r \left(\alpha(a_p\otimes b_q)\alpha^*\right) v_p\otimes w_q
\in D.$$
Suppose that $S\in M_n(\omax(V))^+$ and $\alpha\in M_{1,nm}$. 
By the definition of the maximal tensor product \cite{kptt},  if $\epsilon > 0$ then 
$S + \epsilon 1_n$ has the form of $S_0$ in (\ref{eq_T0}). 
Hence
$$\alpha \left(S\otimes T_0\right) \alpha^* + \epsilon  \alpha \left(1_n \otimes T_0\right) \alpha^*
= \alpha \left((S + \epsilon 1_n) \otimes T_0\right) \alpha^*\in D.$$
Since 
$\alpha \left(1_n \otimes T_0\right) \alpha^* \in D$, 
the previous paragraph shows that 
$$\alpha \left(S\otimes T_0\right) \alpha^*\in D.$$

Now let $T\in M_m(\omax(W))^+$, and write $T + \epsilon 1_m$ in the form of $T_0$ in (\ref{eq_T0}). 
Then 
$$\alpha \left(S\otimes T\right) \alpha^* + \epsilon  \alpha \left(S \otimes 1_m\right) \alpha^*
= \alpha \left(S \otimes (T + \epsilon 1_m)\right) \alpha^*\in D.$$
By the previous paragraph, $\alpha \left(S \otimes 1_m\right) \alpha^*\in D$; by the first 
paragraph, $\alpha \left(S\otimes T\right) \alpha^* \in D$. 

Let $u\in \left(\omax(V)\otimes_{\max} \omax(W)\right)^+$. 
By the definition of the maximal tensor product \cite{kptt}, 
for every $\epsilon > 0$, there exist $n,m\in \bb{N}$, 
$S\in M_n(\omax(V))^+$, $T\in M_m(\omax(W))^+$ and $\alpha \in M_{1,nm}$, such that 
$u + \epsilon 1 = \alpha \left(S\otimes T\right) \alpha^*$. 
By the previous and the first paragraph, $u\in D$. 
\end{proof}

\begin{theorem}\label{t_locd}
Let $X,Y,A,B$ be finite sets and $\Gamma : M_{XY}\to M_{AB}$ be a linear map. The following are equivalent:
\begin{itemize}
\item[(i)] $\Gamma$ is a local QNS correlation;
\item[(ii)] there exists a state $s : \omin(\cl T_{X,A}) \otimes_{\min} \omin(\cl T_{Y,B}) \to \bb{C}$ such that $\Gamma = \Gamma_s$.
\end{itemize}
\end{theorem}

\begin{proof} 
(ii)$\Rightarrow$(i) 
Let $s : \omin(\cl T_{X,A}) \otimes_{\min} \omin(\cl T_{Y,B}) \to \bb{C}$ be a state. 
Using \cite[Theorem 9.9]{kavruk_thesis} and \cite[Proposition 1.9]{fp}, 
we can identify $s$ with an element of 
$\left(\omax(\cl T_{X,A}^{\rm d}) \otimes_{\max} \omax(\cl T_{Y,B}^{\rm d})\right)^+$.
By Lemma \ref{l_fdom}, there exist states $\phi_i\in \left(\cl T_{X,A}^{\rm d}\right)^+$ and 
$\psi_i\in \left(\cl T_{Y,B}^{\rm d}\right)^+$, and non-negative scalars $\lambda_i$, $i = 1,\dots,m$, 
such that $s\equiv \sum_{i=1}^m \lambda_i \phi_i\otimes \psi_i$.  
Set $E_i = (\phi_i(e_{x,x',a,a'}))_{x,x',a,a'}$ (resp. $F_i = (\psi_i(f_{y,y',b,b'}))_{y,y',b,b'}$), 
and let $\Phi_i : M_X\to M_A$ (resp. $\Psi_i : M_Y\to M_B$) be the quantum channel with Choi matrix 
$E_i$ (resp. $F_i$), $i = 1,\dots,m$. 
Then $\Gamma_s = \sum_{i=1}^m \lambda_i \Phi_i\otimes \Psi_i$. 

(i)$\Rightarrow$(ii) 
Write $\Gamma = \sum_{i=1}^m \lambda_i \Phi_i\otimes \Psi_i$ as a convex combination of 
quantum channels
$\Phi_i : M_X\to M_A$ and $\Psi_i : M_Y\to M_B$, $i = 1,\dots,m$, 
and let $s$ be a functional on $\cl T_{X,A}\otimes \cl T_{Y,B}$ such that $\Gamma = \Gamma_s$. 
Let $E_i \in (M_X\otimes M_A)^+$ (resp. $F_i \in (M_Y\otimes M_B)^+$)
be the Choi matrix of $\Phi_i$ (resp. $\Psi_i$); thus, $E_i$ (resp. $F_i$) is a 
stochastic operator matrix acting on $\mathbb C$.
By Theorem \ref{p_ucptau2}, there exist positive functionals $\phi_i : \cl T_{X,A}\to \bb{C}$ and 
$\psi_i : \cl T_{Y,B}\to \bb{C}$ such that 
$(\phi_i(e_{x,x',a',a}))_{x,x',a,a'} = E_i$ and $(\psi_i(f_{y,y',b',b}))_{y,y',b,b'} = F_i$, $i = 1,\dots,m$. 
It is now straightforward to see that $s$ is the functional corresponding to $\sum_{i=1}^m \lambda_i\phi_i\otimes \psi_i$
and is hence, by Lemma \ref{l_fdom}, a state on $ \omin(\cl T_{X,A}) \otimes_{\min} \omin(\cl T_{Y,B})$. 
\end{proof}


\section{Classical-to-quantum no-signalling correlations}\label{s_ctqqnsc}

In this section, we consider the set of
classical-to-quantum no-signalling correlations, and provide descriptions of its 
various subclasses in terms of canonical operator systems.


\subsection{Definition and subclasses}\label{ss_dsc}

Let $X$, $Y$, $A$ and $B$ be finite sets and $H$ be a Hilbert space.

\begin{definition}\label{d_CQNS}
A family $\Theta = (\sigma_{x,y})_{x\in X, y\in Y}$ of states 
in $M_{AB}$ 
is called a \emph{classical-to-quantum no-signalling (CQNS) correlation} if
\begin{equation}\label{eq_wxy}
\Tr\mbox{}_A \sigma_{x',y} = \Tr\mbox{}_A \sigma_{x'',y} \mbox{ and } \Tr\mbox{}_B \sigma_{x,y'} = \Tr\mbox{}_B \sigma_{x,y''},
\end{equation}
for all $x,x',x''\in X$ and all $y,y',y''\in Y$.
\end{definition}

A stochastic operator matrix $E\in M_X\otimes M_A\otimes \cl B(H)$ will be called 
\emph{semi-classical} if $L_{e_x e_{x'}^*}(E) = 0$ whenever $x\neq x'$, that is, if
$$E = \sum_{x\in X} e_x e_x^* \otimes E_x,$$
for some $E_x\in \left(M_A\otimes\cl B(H)\right)^+$ with $\Tr_A E_x = I_H$, $x\in X$. 
We write $E = (E_x)_{x\in X}$; note that, in its own right, $E_x$ is a stochastic 
operator matrix in $\cl L(\bb{C})\otimes M_A\otimes \cl B(H)$, $x\in X$. 

Suppose that 
$E = (E_x)_{x\in X}$ and $F = (F_y)_{y\in Y}$ form a commuting pair of
semi-classical stochastic operator matrices, acting on a Hilbert space $H$ and
$\sigma$ is a vector state on $\cl B(H)$.
The family formed by the states 
\begin{equation}\label{eq_cqnscq}
\sigma_{x,y} = L_{\sigma}(E_x \cdot F_y), \ \ \ x\in X, y\in Y,
\end{equation}
is a CQNS correlation; 
indeed, by Proposition \ref{p_dotp}, 
$\Tr_A \sigma_{x,y} = L_{\sigma}(F_y)$ and $\Tr_B \sigma_{x,y} = L_{\sigma}(E_x)$ for all $x,y$.
We call the CQNS correlations of this form \emph{quantum commuting}. 
Similarly, if 
$(E_x)_{x\in X}$
(resp. $(F_y)_{y\in Y}$) 
is a semi-classical stochastic operator matrix on $H_A$ (resp. $H_B$) 
and $\sigma$ is a vector state on $\cl L(H_A\otimes H_B)$, the family formed by 
$$\sigma_{x,y} = L_{\sigma}(E_x\odot F_y), \ \ \ x\in X, y\in Y,$$
will be called a 
\emph{quantum} CQNS correlation. 
A CQNS correlation $\Theta = (\sigma_{x,y})_{x\in X, y\in Y}$ will be called 
\emph{approximately quantum} if there exist quantum CQNS correlations 
$\Theta_n = (\sigma_{x,y}^{(n)})_{x\in X, y\in Y}$, $n\in \bb{N}$, such that 
$$\sigma_{x,y}^{(n)} \to_{n\to\infty} \sigma_{x,y}, \ \ \ x\in X, y\in Y.$$
Finally, $\Theta$ will be called 
\emph{local} if there exist states $\sigma_{i,x}^{A}$ (resp. $\sigma_{i,y}^{B}$) in $M_A$
(resp. $M_B$) and scalars $\lambda_i > 0$, $i = 1,\dots,m$, such that 
$$\sigma_{x,y} = \sum_{i=1}^m  \lambda_i \sigma_{i,x}^{A} \otimes \sigma_{i,y}^{B}\ \ \ x\in X, y\in Y.$$

If $\cl E : \cl D_{XY}\to M_{AB}$ is a (classical-to-quantum) channel, we set 
$\Gamma_{\cl E} = \cl E\circ \Delta_{XY}$; thus, 
$\Gamma_{\cl E}$ is a (quantum) channel from $M_{XY}$ to $M_{AB}$.
Given a CQNS correlation $\Theta = \left(\sigma_{x,y}\right)_{x\in X, y\in Y}$, we let 
$\cl E_{\Theta} : \cl D_{XY}\to M_{AB}$ be the channel given by 
$$\cl E_{\Theta}\left(e_x e_x^*\otimes e_y e_y^*\right) = \sigma_{x,y}, \ \ \ x\in X, y\in Y,$$
and 
$\Gamma_{\Theta} = \Gamma_{\cl E_{\Theta}}$. 
In the sequel, we will often identify $\Theta$ with the channel $\cl E_{\Theta}$. 
For ${\rm x}\in \{{\rm loc}, {\rm q}, {\rm qa}, {\rm qc}, {\rm ns}\}$, 
we write 
$\cl {CQ}_{\rm x}$ for the set of all CQNS correlations of class ${\rm x}$; thus, the elements of 
$\cl {CQ}_{\rm x}$ will often be considered as channels from $\cl D_{XY}$ to $M_{AB}$. 
Similarly to the proof of Proposition \ref{p_genstok}, it can be shown that 
quantum and quantum commuting CQNS correlations can be defined
using normal (not necessarily vector) states.

In the next lemma, for (finite) sets $X$ and $A$ and a Hilbert space $H$, 
we let for brevity 
$$\tilde{\Delta}_X := \Delta_X\otimes \id\mbox{}_A\otimes \id\mbox{}_{\cl B(H)} : 
M_X\otimes M_A\otimes \cl B(H)\to \cl D_X\otimes M_A\otimes \cl B(H)$$
and
$$\tilde{\Delta}_{X,A} := \Delta_X\otimes \Delta_A \otimes \id\mbox{}_{\cl B(H)} : 
M_X\otimes M_A\otimes \cl B(H)\to \cl D_X\otimes \cl D_A\otimes \cl B(H).$$

\begin{lemma}\label{l_prs}
Let $H$ be a Hilbert space, $E \in M_X\otimes M_A\otimes \cl B(H)$ be a stochastic operator matrix
and $\sigma\in \cl T(H)$ be a state. 
Set $E' = \tilde{\Delta}_X(E)$ and $E'' = \tilde{\Delta}_{X,A}(E)$.
Then $E'$ (resp. $E''$) is a semi-classical (resp. classical) stochastic operator matrix,
\begin{equation}\label{eq_DeltaX}
\Gamma_{E,\sigma}\circ \Delta_X = \Gamma_{E',\sigma}  \ \mbox{ and } \ 
\Delta_A\circ \Gamma_{E,\sigma}\circ \Delta_X = \Gamma_{E'',\sigma}.
\end{equation}
Moreover, if $F \in M_Y\otimes M_B\otimes \cl B(H)$ is a stochastic operator matrix that 
forms a commuting pair with $E$ then 
\begin{equation}\label{eq_DeltaXY}
\tilde{\Delta}_{XY}(E\cdot F) = \tilde{\Delta}_X(E)\cdot \tilde{\Delta}_Y(F).
\end{equation}
\end{lemma}

\begin{proof}
Note that, if $E = \sum_{x,x'\in X} \sum_{a,a'\in A} e_x e_{x'}^* \otimes e_a e_{a'}^* \otimes E_{x,x',a,a'}$
then 
$$\tilde{\Delta}_X(E) = \sum_{x\in X} \sum_{a,a'\in A} e_x e_{x}^* \otimes e_a e_{a'}^* \otimes E_{x,x,a,a'}.$$
We now have 
$$\left \langle \Gamma_{E,\sigma}(\Delta_X(e_{x} e_{x'}^*)), e_{a} e_{a'}^* \right\rangle 
= \delta_{x,x'} \langle E_{x,x',a,a'},\sigma\rangle = 
\left \langle \Gamma_{E',\sigma}(e_{x} e_{x'}^*), e_{a} e_{a'}^* \right\rangle$$
for all $x,x'\in X$ and all $a,a'\in A$.
The second identity  in (\ref{eq_DeltaX}) is equally straightforward.
Finally, for (\ref{eq_DeltaXY}), notice that, if 
$E = \left(E_{x,x',a,a'}\right)$ and $F = \left(F_{y,y',b,b'}\right)$ then
both sides of the identity are equal to 
$$\sum_{x\in X} \sum_{y\in Y} \sum_{a,a'\in A} \sum_{b,b'\in B}
e_x e_{x}^* \otimes e_y e_{y}^*  \otimes e_a e_{a'}^* \otimes e_b e_{b'}^*  \otimes E_{x,x,a,a'}F_{y,y,b,b'}.$$
\end{proof}

\begin{theorem}\label{th_charcq}
Fix ${\rm x}\in \{{\rm loc}, {\rm q}, {\rm qa}, {\rm qc}, {\rm ns}\}$. 
If $\Gamma\in \cl Q_{\rm x}$ then $\Gamma|_{\cl D_{XY}}\in \cl{CQ}_{\rm x}$; 
conversely, 
if $\cl E\in \cl{CQ}_{\rm x}$ then $\Gamma_{\cl E} \in \cl Q_{\rm x}$.
Moreover, for a channel $\cl E : \cl D_{XY}\to M_{AB}$, we have that

\begin{itemize}
\item[(i)] 
$\cl E \in \cl{CQ}_{\rm qc}$ if and only if 
$\Gamma_{\cl E} = \Gamma_{E\cdot F,\sigma}$, where
$(E,F)$ is a commuting pair of semi-classical 
stochastic operator matrices, acting on a Hilbert spaces $H$,
and $\sigma$ is a normal state on $\cl B(H)$;

\item[(ii)] 
$\cl E \in \cl{CQ}_{\rm q}$ if and only if $\Gamma_{\cl E} = \Gamma_{E\odot F,\sigma}$, where
$E$ and $F$ are semi-classical stochastic operator matrices, acting on finite dimensional Hilbert spaces $H_A$ and 
$H_B$, respectively, and $\sigma$ is a state on $\cl L(H_A\otimes H_B)$.
\end{itemize}
\end{theorem}

\begin{proof}
It is trivial that if $\Gamma\in \cl Q_{\rm ns}$ then 
$\Gamma|_{\cl D_{XY}}\in \cl{CQ}_{\rm ns}$.
Conversely, suppose that $\cl E\in \cl{CQ}_{\rm ns}$, and let $\rho_X\in M_X$ and $\rho_Y\in M_Y$ be states, 
with $\Tr(\rho_X) = 0$. 
By (\ref{eq_wxy}), 
$$\Tr\mbox{}_A\Gamma_{\cl E}\left(\rho_X\otimes \rho_Y\right) = 
\sum_{x\in X} \sum_{y\in Y} \langle\rho_X e_x,e_x\rangle \langle\rho_Y e_y,e_y\rangle \Tr\mbox{}_A \sigma_{x,y} = 0$$
and, by symmetry, $\Gamma_{\cl E} \in \cl Q_{\rm ns}$.

Let $E\in M_X\otimes M_A\otimes \cl B(H)$ and $F\in M_Y\otimes M_B\otimes \cl B(H)$
form a commuting pair of stochastic operator matrices and $\sigma\in \cl T(H)$ be a state. 
It follows from Lemma \ref{l_prs} that 
$$\Gamma_{E\cdot F,\sigma}|_{\cl D_{XY}} 
= \Gamma_{\Delta_{XY}(E\cdot F),\sigma}|_{\cl D_{XY}}
= \Gamma_{\Delta_{X}(E)\cdot \Delta_{Y}(F),\sigma}|_{\cl D_{XY}} \in \cl{CQ}_{\rm qc}.$$
Conversely, suppose that $\cl E_{\Theta}\in \cl{CQ}_{\rm qc}$, where
$\Theta = (\sigma_{x,y})_{x\in X, y\in Y}$ is a CQNS correlation. 
Let $H$, $\sigma$, $E$ and $F$ be such that (\ref{eq_cqnscq}) holds;
then $\Gamma_{\Theta} = \Gamma_{E\cdot F,\sigma}$. 
A similar argument applies in the case ${\rm x} = {\rm q}$, and 
the case ${\rm x} = {\rm qa}$ follows from the fact that the map $\cl E \to \cl E\circ \Delta_{XY}$,
from $\cl L(\cl D_{XY},M_{AB})$ into $\cl L(M_{XY},M_{AB})$, is continuous. 
Finally, if $\sigma_{x,y} = \sigma_x\otimes \sigma^y$, where $\sigma_x\in M_A$ 
(resp. $\sigma^y\in M_B$) is a state, $x\in X$ (resp. $y\in Y$),
and $\Phi : M_X\to M_A$ (resp. $\Psi : M_Y\to M_B$) is the channel given by 
$\Phi(e_x e_{x'}^*) = \delta_{x,x'} \sigma_x$ (resp. $\Psi(e_y e_{y'}^*) = \delta_{y,y'} \sigma^y$),
then $\Gamma_{\cl E} = \Phi\otimes \Psi$, and the case ${\rm x} = {\rm loc}$ follows. 
\end{proof}


\subsection{Description in terms of states}\label{ss_dts}

We next introduce an operator system, universal for 
classical-to-quantum no-signalling correlations in a similar 
manner that $\cl T_{X,A}$ is universal for the (fully) quantum correlations, and
describe the subclasses of CQNS correlations
via states on tensor products of its copies.

Let 
$$\cl B_{X,A} = \underbrace{M_A\ast_1\cdots\ast_1 M_A}_{|X| \mbox{ times}},$$
a free product, amalgamated over the unit. 
For each $x\in X$, write $\{e_{x,a,a'} : a,a'\in A\}$ for the canonical matrix unit system of the $x$-th copy of $M_A$, and let
$$\cl R_{X,A} = {\rm span}\{e_{x,a,a'} : x\in X, a,a'\in A\},$$
considered as an operator subsystem of $\cl B_{X,A}$.

Given operator systems $\cl S_1,\dots,\cl S_n$, their \emph{coproduct} 
$\cl S = \cl S_1\oplus_1\cdots\oplus_1\cl S_n$ is an operator system, equipped with 
complete order embeddings $\iota_i : \cl S_i \to \cl S$, 
characterised by the universal property that, 
whenever $\cl R$ is an operator system and $\phi_i : \cl S_i \to \cl R$ is a unital completely positive map, $i = 1,\dots,n$, 
there exists a unique unital completely positive map $\phi : \cl S\to \cl R$ such that 
$\phi \circ \iota_i = \phi_i$, $i = 1,2,\dots,n$. 
We refer the reader to \cite[Section 8]{kavruk_JOT} for a detailed account of the coproduct of operator systems.

\begin{remark}\label{l_crxa}
{\rm 
Let $\cl A_i$, $i = 1,\dots,n$, be unital C*-algebras and 
$\cl S = {\rm span}\{a_i : a_i\in \cl A_i, i = 1,\dots,n\}$, considered as an operator subsystem of 
the free product $\cl A_1\ast_1\cdots\ast_1 \cl A_n$, amalgamated over the unit. 
It was shown in \cite[Theorem 5.2]{fkpt_NYJ} that $\cl S \cong_{\rm c.o.i.}\cl A_1 \oplus_1 \cdots \oplus_1 \cl A_n$.
In particular, we have
\begin{equation}\label{eq_intomx}
\cl R_{X,A} \cong \underbrace{M_A\oplus_1\cdots \oplus_1 M_A}_{|X| \mbox{ times}}.
\end{equation}
An application of \cite[Lemma 2.8]{pt_QJM} now shows that
\begin{equation}\label{eq_intomxy}
\cl R_{X,A} \otimes_{\rm c} \cl R_{Y,B} \subseteq_{\rm c.o.i.} \cl B_{X,A} \otimes_{\max} \cl B_{Y,B}.
\end{equation}
}
\end{remark}

\begin{theorem}\label{th_ucptau4}
Let $H$ be a Hilbert space and $\phi : \cl R_{X,A}\to \cl B(H)$ be a linear map.
The following are equivalent:
\begin{itemize}
\item[(i)] $\phi$ is a unital completely positive map;
\item[(ii)] $\left(\left(\phi(e_{x,a,a'})\right)_{a,a'\in A}\right)_{x\in X}$ is a semi-classical 
stochastic operator matrix.
\end{itemize}
\end{theorem}

\begin{proof}
(i)$\Rightarrow$(ii) 
The restriction $\phi_x$ of $\phi$ to the $x$-th copy of $M_A$ is a unital completely positive map. 
By Choi's Theorem, $\left(\phi_x(e_{x,a,a'})\right)_{a,a'}$ is a
stochastic operator matrix in $M_A\otimes \cl B(H)$ for every $x\in X$; 
thus, $\left(\left(\phi(e_{x,a,a'})\right)_{a,a'\in A}\right)_{x\in X}$ is a semi-classical stochastic operator matrix.

(ii)$\Rightarrow$(i) 
For each $x\in X$, let $\phi_x : M_A\to \cl B(H)$ be the linear map defined by letting 
$\phi_x(e_{a}e_{a'}^*) = \phi(e_{x,a,a'})$. 
By Choi's Theorem, $\phi_x$ is a (unital) completely positive map. By the universal property of the 
coproduct, there exists a (unique) unital completely positive map $\psi : \cl R_{X,A}\to \cl B(H)$
whose restriction to the $x$-th copy of $M_A$ coincides with $\phi_x$.
It follows that $\psi = \phi$, and hence $\phi$ is completely positive. 
\end{proof}

\begin{remark}\label{r_ctoq}
{\rm 
By \cite[Theorem 5.1]{fkpt_NYJ}, 
$\cl R_{X,A}$ is an operator system quotient of $M_{XA}$. 
Now \cite[Proposition 1.8]{fp} shows that, if
$$\cl Q_{X,A} = \left\{\oplus_{x\in X} T_x \in \oplus_{x\in X} M_A
: \exists \ c\in \bb{C} \mbox{ s.t. } \Tr T_x = c, x\in X\right\},$$
then the linear map $\Lambda_{\rm cq} : \cl R_{X,A}^{\rm d} \to \cl Q_{X,A}$, given by 
$$\Lambda_{\rm cq}(\phi) = \oplus_{x\in X} \left(\phi(e_{x,a,a'})\right)_{a,a'},$$
is a well-defined unital complete order isomorphism.
}
\end{remark}


We denote the canonical generators of $\cl R_{Y,B}$ by $f_{y,b,b'}$, $y\in Y$, $b,b'\in B$.
Given a functional $t : \cl R_{X,A} \otimes \cl R_{Y,B} \to \bb{C}$, we let
$\cl E_t : \cl D_{XY}\to M_{AB}$ be the linear map defined by
$$\cl E_t\left(e_xe_x^*\otimes e_y e_y^*\right)
= \sum_{a,a'\in A}\sum_{b,b'\in B} t\left(e_{x,a,a'} \otimes f_{y,b,b'}\right) e_ae_{a'}^*\otimes e_be_{b'}^*.$$
We note that $t\to \cl E_t$ is a linear map from $(\cl R_{X,A} \otimes \cl R_{Y,B})^*$ into $\cl L(\cl D_{XY},M_{AB})$.

Theorem \ref{th_ucptau4} and the universal property of the coproduct imply the existence of 
a unital completely positive map
$\beta_{X,A} : \cl R_{X,A}\to \cl T_{X,A}$ such that 
$$\beta_{X,A}(e_{x,a,a'}) = e_{x,x,a,a'}, \ \ \ x\in X, a,a'\in A.$$
Similarly, the matrix $(\delta_{x,x'} e_{x,a,a'})_{x,x',a,a'}$ is stochastic, and 
Theorem \ref{p_ucptau2} implies the existence of a unital completely positive map
$\beta_{X,A}' : \cl T_{X,A}\to \cl R_{X,A}$ such that
$$\beta_{X,A}'(e_{x,x',a,a'}) = \delta_{x,x'} e_{x,a,a'}, \ \ \ x,x'\in X, a,a'\in A.$$
It is clear that 
$$\beta_{X,A}'\circ \beta_{X,A} = \id\mbox{}_{\cl R_{X,A}}.$$

\begin{theorem}\label{th_cqdes}
The map $t\to \cl E_t$ is an affine isomorphism  
\begin{itemize}
\item[(i)] from the state space of $\cl R_{X,A} \otimes_{\max} \cl R_{Y,B}$ onto $\cl{CQ}_{\rm ns}$;

\item[(ii)] from the state space of $\cl R_{X,A} \otimes_{\rm c} \cl R_{Y,B}$ onto $\cl{CQ}_{\rm qc}$;

\item[(iii)] from the state space of $\cl R_{X,A} \otimes_{\min} \cl R_{Y,B}$ onto $\cl{CQ}_{\rm qa}$;

\item[(iv)] from the state space of $\omin\left(\cl R_{X,A}\right) \otimes_{\min} \omin\left(\cl R_{Y,B}\right)$ onto $\cl{CQ}_{\rm loc}$.
\end{itemize}
\end{theorem}

\begin{proof}
It is clear that the map $t\to \cl E_t$ is bijective.
It is also straightforward to see that, for a linear functional 
$s : \cl T_{X,A}\otimes \cl T_{Y,B}\to \bb{C}$, 
we have $\Gamma_s|_{\cl D_{XY}} = \cl E_t$, where 
$t = s\circ \left(\beta_{X,A}\otimes \beta_{Y,B}\right)$.
The claims now follow from Theorems \ref{th_qnsma}, \ref{t_dqc}, \ref{t_dqa}, \ref{t_locd}, \ref{th_charcq}
and the functoriality of the involved tensor products.
\end{proof}

As a consequence of Theorem \ref{th_cqdes}, we see that 
the sets $\cl{CQ}_{\rm qc}$ and $\cl{CQ}_{\rm loc}$ are closed
(as are $\cl{CQ}_{\rm ns}$ and $\cl{CQ}_{\rm qa}$).

\medskip

\noindent {\bf Remark. } 
As in Theorems \ref{t_dqc} and \ref{t_dqa}, the classes $\cl{CQ}_{\rm qc}$ and $\cl{CQ}_{\rm qa}$
can be equivalently described via states on the C*-algebraic tensor products 
$\cl B_{X,A} \otimes_{\max} \cl B_{Y,B}$ and $\cl B_{X,A} \otimes_{\min} \cl B_{Y,B}$,
respectively. For the class $\cl{CQ}_{\rm qa}$, this is a direct consequence of the injectivity of the minimal 
tensor product in the operator system category, while for the class $\cl{CQ}_{\rm qc}$, this is a 
consequence of Remark \ref{l_crxa}.


\section{Classical reduction and separation}\label{s_crs}

Let $X$ and $A$ be finite sets.
We let 
$$\cl A_{X,A} = \underbrace{\ell^{\infty}_A \ast_1\cdots \ast_1 \ell^{\infty}_A}_{|X| \mbox{ times}},$$
where the free product is amalgamated over the unit, 
and 
$$\cl S_{X,A} = \underbrace{\ell^{\infty}_A \oplus_1 \cdots \oplus_1 \ell^{\infty}_A}_{|X| \mbox{ times}},$$
the operator system coproduct of $|X|$ copies of $\ell^{\infty}_A$. 
Note that, by \cite[Theorem 5.2]{fkpt_NYJ} (see Remark \ref{l_crxa}), 
$\cl S_{X,A}$ is an operator subsystem of $\cl A_{X,A}$.
We let $(e_{x,a})_{a\in A}$ be the canonical basis of the $x$-th copy of $\ell^{\infty}_A$ inside $\cl S_{X,A}$;
thus, $\cl S_{X,A}$ is generated, as a vector space, by $\{e_{x,a} : x\in X, a\in A\}$, and the relations 
$$\sum_{a\in A} e_{x,a} = 1, \ \ \ x\in X,$$
are satisfied.
Note that, by the universal property of the operator system coproduct, 
$\cl S_{X,A}$ is characterised by the following property: 
whenever $H$ is a Hilbert space and $\{E_{x,a} : x\in X, a\in A\}$ is a family of positive operators on $H$ 
such that $(E_{x,a})_{a\in A}$ is a POVM for every $x\in X$, there exists a (unique) unital completely positive map
$\phi : \cl S_{X,A}\to \cl B(H)$ such that $\phi(e_{x,a}) = E_{x,a}$, $x\in X$, $a\in A$.

We denote by $\frak{E}$ the map sending a quantum channel $\Gamma : M_{XY}\to M_{AB}$ to 
$\Gamma|_{\cl D_{XY}}$
(and recall that $\frak{N}$ stands for the map sending $\Gamma$ to $\cl N_{\Gamma} = \Delta_{AB}\circ \Gamma|_{\cl D_{XY}}$); 
Remark \ref{th_onedq} below 
justifies calling $\frak{E}$ and $\frak{N}$ \emph{classical reduction maps}. 
The forward implications all follow similarly to the one in (ii), which was shown in Theorem \ref{th_charcq}, 
while the reverse ones can be seen after an application of Lemma \ref{l_prs}.
We recall that we identify $\cl C_{\rm ns}$ with the set $\{\cl N_p : p \mbox{ an NS correlation}\}$.

\begin{remark}\label{th_onedq}
Let $X$, $Y$, $A$ and $B$ be finite sets, ${\rm x}\in \{{\rm loc}, {\rm q}, {\rm qa}, {\rm qc}, {\rm ns}\}$,
$p\in \cl C_{\rm x}$ and $\cl E\in \cl{CQ}_{\rm x}$. 
The following hold:
\begin{itemize}
\item[(i)] $p \in \cl C_{\rm x}$ $\Leftrightarrow$ $\cl E_p \in \cl{CQ}_{\rm x}$ $\Leftrightarrow$ $\Gamma_p \in \cl Q_{\rm x}$;

\item[(ii)] $\cl E \in \cl{CQ}_{\rm x}$ $\Leftrightarrow$ $\Gamma_{\cl E} \in \cl Q_{\rm x}$.
\end{itemize}
\noindent
Moreover, the maps 
$\frak{E} : \cl{Q}_{\rm x} \to \cl{CQ}_{\rm x}$ and
$\frak{N} : \cl{CQ}_{\rm x} \to \cl C_{\rm x}$ are well-defined and surjective.
\end{remark}


We identify an element $\cl N$ of $\cl C_{\rm x}$ with the 
corresponding classical-to-quantum channel from $\cl D_{XY}$ into $M_{AB}$, and an element 
$\cl E$ of $\cl{CQ}_{\rm x}$ with the corresponding quantum channel from $M_{XY}$ into $M_{AB}$.
The subsequent table summarises the inclusions between the various classes of correlations:

$$
\begin{matrix}
\cl C_{\rm loc} & \subset & \cl C_{\rm q} & \subset & \cl C_{\rm qa} & \subset & \cl C_{\rm qc} & \subset & \cl C_{\rm ns}\\
\cap & & \cap & & \cap & & \cap & & \cap\\
\cl{CQ}_{\rm loc} & \subset & \cl{CQ}_{\rm q} & \subset & \cl{CQ}_{\rm qa} & \subset & \cl{CQ}_{\rm qc} & \subset & \cl{CQ}_{\rm ns}\\
\cap & & \cap & & \cap & & \cap & & \cap\\
\cl Q_{\rm loc} & \subset & \cl Q_{\rm q} & \subset & \cl Q_{\rm qa} & \subset & \cl Q_{\rm qc} & \subset & \cl Q_{\rm ns}.
\end{matrix}
$$

\medskip

\noindent 
By Bell's Theorem, $\cl C_{\rm loc} \neq \cl C_{\rm q}$ for all subsets $X,Y,A,B$ of cardinality at least $2$.
By Remark \ref{th_onedq}, we have that $\cl{CQ}_{\rm loc} \neq \cl{CQ}_{\rm q}$ and 
$\cl Q_{\rm loc} \neq \cl Q_{\rm q}$. 
By \cite{slofstra}, $\cl C_{\rm q} \neq \cl C_{\rm qa}$ for some finite sets $X$, $Y$, $A$ and $B$
(see also \cite{dpp})
and hence
$\cl{CQ}_{\rm q} \neq \cl{CQ}_{\rm qa}$ and 
$\cl Q_{\rm q} \neq \cl Q_{\rm qa}$ for a suitable choice of sets. 
The inequality $\cl C_{\rm qc} \neq \cl C_{\rm ns}$ is well-known
(it follows e.g. from \cite[Theorem 7.11]{fkpt}),
implying that $\cl{CQ}_{\rm qc} \neq \cl{CQ}_{\rm ns}$ and $\cl Q_{\rm qc} \neq \cl Q_{\rm ns}$.

It was recently shown \cite{jnvwy} that the inequality $\cl C_{\rm qa} \neq \cl C_{\rm qc}$ 
also holds true for suitable sets $X$, $Y$, $A$ and $B$, thus resolving the long-standing Tsirelson Problem 
and, by \cite{jnpp} and \cite{oz}, the Connes Embedding Problem, in the negative.
It thus follows from Remark \ref{th_onedq} that, for this choice of sets, 
$\cl{CQ}_{\rm qa} \neq \cl{CQ}_{\rm qc}$ and $\cl Q_{\rm qa} \neq \cl Q_{\rm qc}$. 
We next strengthen these inequalities.

\begin{lemma}\label{l_xaybe}
Let $X_i$ and $A_i$ be finite sets, $i = 1,2$, with $X_1\subseteq X_2$ and $A_1\subseteq A_2$. 
There exist unital completely positive maps
$\iota_1 : \cl S_{X_1,A_1}\to \cl S_{X_2,A_2}$
and $\iota_2 : \cl S_{X_2,A_2}\to \cl S_{X_1,A_1}$ such that $\iota_2\circ\iota_1 = \id$. 
\end{lemma}

\begin{proof}
Denote the canonical generators of $\cl S_{X_1,A_1}$ by $e_{x,a}$, and of $\cl S_{X_2,A_2}$ -- by $f_{x,a}$.
By induction, it suffices to prove the claim in two cases.

\smallskip

\noindent {\it Case 1. } $X_1 = X_2$ and $A_2 = A_1 \cup \{a_2\}$, where $a_2\not\in A_1$. 

\smallskip

Let $a_1\in A_1$. Define the maps $\iota_1$ and $\iota_2$ by setting
\begin{equation*}
\iota_1(e_{x,a}) 
= \left\{ 
\begin{array}{ll}
f_{x,a} & \text{if } a\in A_1\setminus \{a_1\} \text{,} \\ 
f_{x,a_1} + f_{x,a_2} & a = a_1\text{,}
\end{array}
\right.
\end{equation*}
and
\begin{equation*}
\iota_2(f_{x,a}) 
= \left\{ 
\begin{array}{ll}
e_{x,a} & \text{if } a\in A_1\setminus \{a_1\}\text{,} \\ 
\frac{1}{2} e_{x,a_1} & a \in \{a_1,a_2\}\text{.}
\end{array}
\right.
\end{equation*}

\smallskip

\noindent {\it Case 2. } $A_2 = A_1$ and $X_2 = X_1 \cup \{x_2\}$, where $x_2\not\in X_1$.

\smallskip

Let $x_1\in X_1$. Define 
$\iota_1(e_{x,a}) = f_{x,a}$, $x\in X_1$, $a\in A_1$, and 
\begin{equation*}
\iota_2(f_{x,a}) 
= \left\{ 
\begin{array}{ll}
e_{x,a} & \text{if } x\in X_1\text{,} \\ 
e_{x_1,a} & x = x_2\text{.}
\end{array}
\right.
\end{equation*}

\smallskip

By the universal property of the operator systems $\cl S_{X,A}$, $\iota_1$ and $\iota_2$ are unital completely positive maps, and 
the condition $\iota_2\circ\iota_1 = \id$ is readily verified. 
\end{proof}

\begin{theorem}\label{th_cep}
For all finite sets $X$, $Y$, $A$ and $B$ of sufficiently large cardinality, the following hold true:
\begin{itemize}
\item[(i)] $\cl Q_{\rm qa}(X,Y,A,B) \neq \cl Q_{\rm qc}(X,Y,A,B)$;

\item[(ii)] $\cl{CQ}_{\rm qa}(X,Y,A,B) \neq \cl{CQ}_{\rm qc}(X,Y,A,B)$;

\item[(iii)] $\cl T_{X,A}\otimes_{\min} \cl T_{Y,B} \neq \cl T_{X,A}\otimes_{\rm c} \cl T_{Y,B}$;

\item[(iv)] $\cl R_{X,A}\otimes_{\min} \cl R_{Y,B} \neq \cl R_{X,A}\otimes_{\rm c} \cl R_{Y,B}$;

\item[(v)] $\cl B_{X,A}\otimes_{\min} \cl B_{Y,B} \neq \cl B_{X,A}\otimes_{\max} \cl B_{Y,B}$;

\item[(vi)] $\cl C_{X,A}\otimes_{\min} \cl C_{Y,B} \neq \cl C_{X,A}\otimes_{\max} \cl C_{Y,B}$.
\end{itemize}
\end{theorem}

\begin{proof}
By \cite{jnvwy}, there exist (finite) sets $X_0$, $Y_0$, $A_0$ and $B_0$ and 
an NS correlation
$$p \in \cl C_{\rm qc}(X_0,Y_0,A_0,B_0) \setminus \cl C_{\rm qa}(X_0,Y_0,A_0,B_0).$$
Using \cite[Corollary 3.2]{lmprsstw}, let $s$ be a state on $\cl S_{X_0,A_0}\otimes_{\rm c} \cl S_{Y_0,B_0}$ 
such that 
\begin{equation}\label{eq_x0y0}
p(a,b|x,y) = s(e_{x,a}\otimes e_{y,b})
\end{equation}
for all $x\in X_0$, $y\in Y_0$, $a\in A_0$ and $b\in B_0$. 
Assume that $X_0\subseteq X$, $Y_0\subseteq Y$, $A_0\subseteq A$ and $B_0\subseteq B$. 
Write $\iota_i^A$ (resp. $\iota_i^B$), $i = 1,2$, for the maps arising from Lemma \ref{l_xaybe} for the 
operator systems $\cl S_{X_0,A_0}$ and $\cl S_{X,A}$ (resp. $\cl S_{Y_0,B_0}$ and $\cl S_{Y,B}$). 
By the functoriality of the commuting tensor product, the map 
$t := s\circ (\iota_2^A\otimes \iota_2^B)$ is a state on $\cl S_{X,A}\otimes_{\rm c} \cl S_{Y,B}$. 
The NS correlation $q \in \cl C_{\rm qc}(X,Y,A,B)$ arising from $t$ as in (\ref{eq_x0y0})
does not belong to the class $\cl C_{\rm qa}$. 
Indeed, if $q\in \cl C_{\rm qa}$ then, by \cite[Corollary 3.3]{lmprsstw}, $t$ is a state on 
$\cl S_{X,A}\otimes_{\min} \cl S_{Y,B}$, and hence
$s = t\circ (\iota_1^A\otimes \iota_1^B)$ (see Lemma \ref{l_xaybe}) 
is a state on $\cl S_{X_0,A_0}\otimes_{\min} \cl S_{Y_0,B_0}$
which, in view of \cite[Corollary 3.3]{lmprsstw}, contradicts the fact that $p$ is not approximately quantum.

It follows that $\cl C_{\rm qa}(X,Y,A,B) \neq \cl C_{\rm qc}(X,Y,A,B)$ for all 
sets $X$, $Y$, $A$ and $B$ of sufficiently large cardinality. 
Parts (i) and (ii) now follow from Remark \ref{th_onedq}. 
Claim (iii) follows from Theorems \ref{t_dqc} and \ref{t_dqa}, while (iv) -- from (ii) and Theorem \ref{th_cqdes}.
Finally, (v) follows from (iv) and Remark \ref{l_crxa}, and (vi) follows from 
(iii), Corollary \ref{c_ucst} and \cite[Theorem 6.4]{kptt}.
\end{proof}

Recall that an operator system $\cl S$ is said to possess
the \emph{operator system local lifting property (OSLLP)} \cite{kptt_adv}
if, whenever $\cl A$ is a unital C*-algebra, $\cl I\subseteq \cl A$ is a two-sided ideal,
$\cl T\subseteq \cl S$ is a finite dimensional operator subsystem and 
$\varphi : \cl T\to \cl A/\cl I$ is a unital completely positive map, there exists 
a unital completely positive map $\psi : \cl T\to \cl A$ such that $\varphi = q\circ \psi$ 
(here $q : \cl A\to \cl A/\cl I$ denotes the quotient map). 
We conclude this section with showing that the operator systems we introduced possess OSLPP.

\begin{proposition}\label{l_opsysq}
Let $\cl S$ be an operator system quotient of $M_k$, for some $k\in \bb{N}$, and $H$ be a Hilbert space.  
Then $\cl S \otimes_{\min}\cl B(H) \cong_{\rm c.o.i.} \cl S \otimes_{\max}\cl B(H)$, and hence 
$\cl S$ possesses OSLLP. 
\end{proposition}

\begin{proof}
Let $J\subseteq M_k$ be a kernel such that $\cl S = M_k/J$; write $q : M_k\to \cl S$ for the quotient map. 
By \cite[Proposition 1.8]{fp}, the dual $q^* : \cl S^{\rm d}\to M_k^{\rm d}$ is a complete order embedding. 

Fix $u\in M_n\left(\cl S\otimes_{\min}\cl B(H)\right)^+$; after a canonical identification, we consider 
$u$ as an element of $\left(\cl S\otimes_{\min} M_n\left(\cl B(H)\right)\right)^+$.
Let $\{S_1,\dots,S_m\}$ be a basis of $\cl S$, and write $u = \sum_{i=1}^m S_i \otimes T_i$, 
for some $T_i\in M_n\left(\cl B(H)\right)$, $i = 1,\dots,m$. 
By \cite[Proposition 6.1]{kavruk_JOT}, 
the map $\phi_u : \cl S^{\rm d} \to M_n\left(\cl B(H)\right)$, given by 
$\phi_u(f) = \sum_{i=1}^m f(S_i) T_i$, is completely positive. By Arveson's Extension Theorem, 
there exists a completely positive map $\psi : M_k^{\rm d} \to M_n\left(\cl B(H)\right)$
with $\psi \circ q^* = \phi_u$. 
Let $S'_i\in M_k$ be such that $q(S'_i) = S_i$, $i = 1,\dots,m$, and 
let $\{S_i' : i = m+1,\dots,k^2\}$ be a basis of $J$. Then 
$\{S'_1,\dots,S'_m,S'_{m+1},\dots,S'_{k^2}\}$ is a basis of $M_k$.
Let 
$$v = \sum_{i=1}^{k^2} S'_i \otimes T_i' \in M_k\otimes M_n\left(\cl B(H)\right)$$ 
be an element such that 
$$\psi(g) = \sum_{i=1}^{k^2} g(S'_i) T_i', \ \ \ g\in M_k^{\rm d};$$
by \cite[Proposition 6.1]{kavruk_JOT}, $v\in \left(M_k\otimes_{\min} M_n\left(\cl B(H)\right)\right)^+$. 
Since $M_k$ is nuclear, 
$v$ belongs to $\left(M_k \otimes_{\max} M_n\left(\cl B(H)\right)\right)^+$.
Let $w = (q\otimes \id)(v)$; by the functoriality of the maximal tensor product, 
$w\in (\cl S \otimes_{\max} M_n(\cl B(H)))^+$.
We have
$$w = \sum_{i=1}^{k^2} q(S_i')\otimes T_i' = \sum_{i=1}^{m} S_i\otimes T_i'.$$
For all $f\in \cl S^{\rm d}$, we have 
$$\sum_{i=1}^{m} f(S_i)T_i' 
= 
\sum_{i=1}^{k^2} q^*(f)(S_i')T_i' = \psi(q^*(f))
= 
\phi_u(f) = 
\sum_{i=1}^{m} f(S_i)T_i.$$
It follows that $T_i = T_i'$, $i = 1,\dots,m$, and hence $u = w$. 
Thus, $u\in M_n\left(\cl S\otimes_{\max} \cl B(H)\right)^+$, and 
it follows from \cite[Theorem 8.6]{kptt_adv} that $\cl S$ possesses OSLLP. 
\end{proof}

Proposition \ref{l_opsysq}, combined with Corollary \ref{c_quotxa} and Remark \ref{r_ctoq},
yield the following corollary. 

\begin{corollary}\label{c_OSLLP}
Let $X$ and $A$ be finite sets. Then $\cl T_{X,A}$ and $\cl R_{X,A}$ possess OSLPP. 
\end{corollary}

\noindent {\bf Remark. } 
It is worth noting the different nature of the C*-algebras $\cl A_{X,A}$ and $\cl B_{X,A}$ on one hand, 
and $\cl C_{X,A}$ on the other. This is best seen in the special case where $|X| = 1$, when
$\cl A_{X,A} \cong \cl D_A$, $\cl B_{X,A} \cong M_A$ and $\cl C_{X,A} \cong C^*_u(M_A)$.


\section{Quantum versions of synchronicity}\label{s_fc}

Let $X$ and $A$ be finite sets, $Y = X$ and $B = A$.
We will often distinguish the notation for 
$X$ vs. $Y$ (resp. $A$ vs. $B$) although they coincide, 
in order to make clear with respect to which term in a tensor product a partial trace is taken. 
An NS correlation $p = \left\{\left(p(a,b|x,y)\right)_{a,b\in A}  : x,y\in X\right\}$
is called \emph{synchronous} \cite{psstw} if 
$$p(a,b|x,x) = 0 \ \ \ x\in X, a,b\in A, a\neq b.$$
In this section, we examine possible quantum versions of the notion of synchronicity.
Our main motivation is the following result, which was proved in \cite{psstw}.

\begin{theorem}\label{th_psstw}
Let $p$ be an NS correlation. Then 
\begin{itemize}
\item[(i)] $p$ is synchronous and quantum commuting if and only if there exists a trace
$\tau : \cl A_{X,A}\to \bb{C}$ such that 
\begin{equation}\label{eq_sytr}
p(a,b|x,y) = \tau\left(e_{x,a}e_{y,b}\right), \ \ \ x,y\in X, a,b\in A;
\end{equation}
\item[(ii)] $p$ is synchronous and quantum if and only if there exist a 
finite dimensional C*-algebra $\cl A$, a trace $\tau_{\cl A}$ on $\cl A$ and 
a *-homomorphism $\pi : \cl A_{X,A} \to \cl A$ such that (\ref{eq_sytr}) holds for the trace $\tau = \tau_{\cl A}\circ \pi$;
\item[(iii)] $p$ is synchronous and local if and only if there exist an abelian 
C*-algebra $\cl A$, a trace $\tau_{\cl A}$ on $\cl A$ and 
a *-homomorphism $\pi : \cl A_{X,A} \to \cl A$ such that (\ref{eq_sytr}) holds for the trace $\tau = \tau_{\cl A}\circ \pi$.
\end{itemize}
\end{theorem}


\subsection{Fair correlations}\label{ss_fc}

If $\cl A$ is a unital C*-algebra, we write $\cl A^{\rm op}$ for the opposite C*-algebra of $\cl A$; 
recall that $\cl A^{\rm op}$ has the same underlying set (whose elements will be denoted by $u^{\rm op}$, for $u\in \cl A$),
the same involution, linear structure and norm, 
and multiplication given by $u^{\rm op} v^{\rm op} = (vu)^{\rm op}$, $u,v\in \cl A$.
For a subset $\cl S\subseteq \cl A$, we let  
$\cl S^{\rm op} = \left\{u^{\rm op} : u\in \cl S\right\}$. 

For a Hilbert space $H$, we denote by $H^{\rm d}$ its Banach space dual; if 
$K$ is a(nother) Hilbert space and $T \in \cl B(H,K)$, we denote by 
$T^{\rm d}$ its adjoint, acting from $K^{\rm d}$ into $H^{\rm d}$. 
We note the relation 
\begin{equation}\label{eq_Tds}
(T^*)^{\rm d} = (T^{\rm d})^*, \ \ \ T\in \cl B(H,K).
\end{equation}
It is straightforward to see that if $\cl A$ is a C*-algebra and $\pi : \cl A\to \cl B(H)$ is a (faithful)
*-representation then the map $\pi^{\rm op} : \cl A^{\rm op}\to \cl B(H^{\rm d})$, given by 
$\pi^{\rm op}(u^{\rm op}) = \pi(u)^{\rm d}$, is a (faithful) *-representation. 
Note that the transposition map $u\to (u^t)^{\rm op}$ is a *-isomorphism between $M_X$  and $M_X^{\rm op}$. 
It was shown in \cite{kps} that 
there exists a *-isomorphism $\partial _{\cl A}: \cl A_{X,A}\to \cl A_{X,A}^{\rm op}$ such that 
$\partial_{\cl A}(e_{x,a}) = e_{x,a}^{\rm op}$, $x\in X$, $a\in A$. 
The following analogous statements for $\cl C_{X,A}$ and $\cl B_{X,A}$ will be needed later. 

\begin{lemma}\label{l_op}
Let $X$ and $A$ be finite sets. 
\begin{itemize}
\item[(i)]
There exists a *-isomorphism $\partial : \cl C_{X,A}\to \cl C_{X,A}^{\rm op}$ such that 
$$\partial(e_{x,x',a,a'}) = e_{x',x,a',a}^{\rm op}, \ \ \ x,x'\in X, a,a'\in A.$$

\item[(ii)] 
There exists a *-isomorphism $\partial _{\cl B}: \cl B_{X,A}\to \cl B_{X,A}^{\rm op}$ such that 
$$\partial_{\cl B}(e_{x,a,a'}) = e_{x,a',a}^{\rm op}, \ \ \ x\in X, a,a'\in A.$$
\end{itemize}
\end{lemma}

\begin{proof}
(i) Let $\pi : \cl C_{X,A}\to \cl B(H)$ be a faithful *-representation. Write $E_{x,x',a,a'} = \pi(e_{x,x',a,a'})$, 
$x,x'\in X$, $a,a'\in A$. 
Using Theorem \ref{p_coor}, let $K$ be a Hilbert space and $(V_{a,x})_{a,x} : H^X\to K^A$ be an isometry such that 
$E_{x,x',a,a'} = V_{a,x}^* V_{a',x'}$, $x,x'\in X$, $a,a'\in A$. 
Let $W_{a,x} = \left(V_{a,x}^{\rm d}\right)^*$; thus, $W_{a,x}\in \cl B(H^{\rm d},K^{\rm d})$, $x\in X$, $a\in A$. 
Using (\ref{eq_Tds}), we have 
$$\sum_{a\in A} W_{a,x'}^*W_{a,x} 
=  
\sum_{a\in A} V_{a,x'}^{\rm d} \left(V_{a,x}^*\right)^{\rm d}
= 
\sum_{a\in A} \left(V_{a,x}^*V_{a,x'}\right)^{\rm d}
= 
\delta_{x,x'} I^{\rm d};$$
thus, $\left(W_{a,x}\right)_{a,x}$ is an isometry. 
By Theorem \ref{p_coor}, if $F_{x,x',a,a'} = W_{a,x}^*W_{a',x'}$, $x,x'\in X$, $a,a'\in A$, then 
$\left(F_{x,x',a,a'}\right)_{x,x',a,a'}$ is a stochastic operator matrix. 
Note that
$$F_{x,x',a,a'} = V_{a,x}^{\rm d}\left(V_{a',x'}^{\rm d}\right)^* = \left(V_{a',x'}^*V_{a,x}\right)^{\rm d}
= E_{x',x,a',a}^{\rm d}.$$
By the universal property of $\cl C_{X,A}$, there exists a *-homomorphism 
$\pi' : \pi\left(\cl C_{X,A}\right) \to \cl B\left(H^{\rm d}\right)$ such that 
$$\pi'\left(E_{x,x',a,a'}\right) = E_{x',x,a',a}^{\rm d}, \ \ \ x,x'\in X, a,a'\in A.$$
By the paragraph before Lemma \ref{l_op},
$\pi'\circ \pi$ can be regarded as a *-homo-morphism from $\cl C_{X,A}$ into $\cl C_{X,A}^{\rm op}$,
which maps $e_{x,x',a,a'}$ to $e_{x',x,a',a}^{\rm op}$. The claim follows by symmetry. 

(ii) The words of the form $e_{x_1,a_1,a_1'}\ldots e_{x_k,a_k,a_k'}$ span a dense $*$-subalgebra of $\cl B_{X,A}$ . 
As  $u\mapsto (u^t)^{\rm op}$ is a *-isomorphism from $M_A$ to $M_A^{\rm op}$ that 
maps the matrix unit $e_{a} e_{a'}^*$ to $\left(e_{a'} e_{a}^*\right)^{\rm op}$, 
the universal property of the free product implies that the map 
$\partial_{\cl B}$
given by 
$$\partial_{\cl B}( e_{x_1,a_1,a_1'}\ldots e_{x_k,a_k,a_k'})
= e_{x_1,a_1',a_1}^{\rm op}\ldots e_{x_k,a_k',a_k}^{\rm op}$$
extends to the desired *-isomorphism.
\end{proof}

If $\cl U$ is a subspace of a $C^*$-algebra $\cl A$, we call a linear functional $s : \cl U\otimes \cl U^{\rm op} \to\bb{C}$ \emph{fair} if
\begin{equation}\label{seq}
s(u\otimes 1) = s(1\otimes u^{\rm op})\text{ for all } u\in \cl U.
\end{equation}

It will be convenient to write $t_Y$ for the transpose map on $M_Y$.  
A state $\rho \in M_{XY}$ will be called \emph{fair} if 
$\Tr_X\left(({\rm id}\otimes t_Y)(\rho)\right) = \Tr_Y\left(({\rm id}\otimes t_Y)(\rho)\right)$. We write
$\Sigma_{X} = \{\rho\in M_{XY}^+: \rho \mbox{ a fair state}\}$, and
observe that an element $\rho = (\rho_{x,x',y,y'})\in M_{XY}^+$ belongs to $\Sigma_{X}$ precisely when
\begin{eqnarray*}
\sum_{x\in X} \sum_{y,y'\in Y} \rho_{x,x,y,y'}e_{y'}e_{y}^* = \sum_{x,x'\in X} \sum_{y\in Y}  \rho_{x,x',y,y}e_xe_{x'}^*,
\end{eqnarray*}
that is, when
\begin{equation}\label{rho}
\sum_{x\in X} \rho_{x,x,z,z'} = \sum_{y\in X} \rho_{z',z,y,y}, \ \ z,z'\in X.
\end{equation} 
We let $\Sigma_{X}^{\rm cl} = \Sigma_{X}\cap \cl D_{XY}$; thus, a state $\rho = (\rho_{x,y})_{x,y}\in \cl D_{XY}^+$
is in $\Sigma_{X}^{\rm cl}$ precisely when 
\begin{equation}\label{rho2}
\sum_{x\in X} \rho_{x,z} = \sum_{y\in X} \rho_{z,y}, \ \ z\in X.
\end{equation} 
It follows from (\ref{rho}) and (\ref{rho2}) that  
\begin{equation}\label{rho3}
\Delta_{XY}(\Sigma_{X})= \Sigma_{X}^{\rm cl}.
\end{equation}

\begin{definition}
A QNS correlation $\Gamma: M_{XY}\to M_{AB}$ 
(resp. a CQNS correlation $\cl E: \cl D_{XY}\to M_{AB}$, 
an NS correlation $\cl N : \cl D_{XY}\to \cl D_{AB}$) 
is called \emph{fair} if $\Gamma\left(\Sigma_{X}\right)\subseteq \Sigma_{A}$
(resp. $\cl E\left(\Sigma_{X}^{\rm cl}\right)\subseteq \Sigma_{A}$,
$\cl N\left(\Sigma_{X}^{\rm cl}\right)\subseteq \Sigma_{A}^{\rm cl}$).
\end{definition}

\begin{theorem}\label{sym}
Let $\Gamma$ be a QNS correlation. 
\begin{itemize}
\item[(i)] $\Gamma$ is fair if and only if there exists a state  
$s : \cl T_{X,A}\otimes_{\max} \cl T_{X,A}\to\mathbb C$ such that 
$\Gamma = \Gamma_s$ and the state $s\circ({\rm id}\otimes\partial)^{-1}$ is fair;

\item[(ii)] $\Gamma$ is fair and belongs to $\cl Q_{\rm qc}$ if and only if there exists a state  
$s : \cl T_{X,A}\otimes_{\rm c} \cl T_{X,A}\to\mathbb C$ such that $\Gamma = \Gamma_s$
and the state $s\circ({\rm id}\otimes\partial)^{-1}$ is fair;

\item[(iii)] $\Gamma$ is fair and belongs to $\cl Q_{\rm qa}$ if and only if there exists a  state  
$s : \cl T_{X,A}\otimes_{\min} \cl T_{X,A}\to\mathbb C$ such that $\Gamma = \Gamma_s$
and the state $s\circ({\rm id}\otimes\partial)^{-1}$ is fair;

\item[(iv)] 
$\Gamma$ is fair and belongs to $\cl Q_{\rm loc}$ if and only if there exists a  state  
$s : \omin(\cl T_{X,A}) \otimes_{\min} \omin(\cl T_{Y,B}) \to\mathbb C$ such that 
$\Gamma = \Gamma_s$
and the state $s\circ({\rm id}\otimes\partial)^{-1}$ is fair.
\end{itemize}
\end{theorem}

\begin{proof}
We only show (i); the proofs of (ii)-(iv) are similar. 
Let $\Gamma$ be a QNS correlation. 
By Theorem \ref{th_qnsma}, 
there exists a state $s\in \cl T_{X,A}\otimes_{\max} \cl T_{X,A}\to\mathbb C$ such that
$\Gamma = \Gamma_s$.
The condition 
$$\Tr\mbox{}_A\left(({\rm id}\otimes t_B)(\Gamma(\rho)\right) 
= \Tr\mbox{}_B\left(({\rm id}\otimes t_B)(\Gamma(\rho)\right)$$ 
is equivalent to
\begin{equation}\label{eq_sbb}
\sum_{a\in A} \left\langle({\rm id}\otimes t_B(\Gamma(\rho))e_{a}\otimes e_{b'},e_a\otimes e_b\right\rangle
= \sum_{a\in A} \left\langle({\rm id}\otimes t_B(\Gamma(\rho))e_{b'}\otimes e_a,e_b\otimes e_a\right\rangle, 
\end{equation}
 $b,b'\in B$.
Note that
$$\Gamma(\rho) 
= \sum_{a,a'\in A}\sum_{b,b'\in A}\sum_{x,x'\in X} \sum_{y,y'\in X} \rho_{x,x',y,y'}s(e_{x,x',a,a'}\otimes e_{y,y',b,b'})e_ae_{a'}^*\otimes e_{b}e_{b'}^*$$
and hence 
\begin{eqnarray*}
& & ({\rm id}\otimes t_B)(\Gamma(\rho))\\
&=&\sum_{a,a'\in A}\sum_{b,b'\in A}\sum_{x,x'\in X} \sum_{y,y'\in X} \rho_{x,x',y,y'}s(e_{x,x',a,a'}\otimes e_{y,y',b,b'})e_ae_{a'}^*\otimes e_{b'}e_{b}^*.
\end{eqnarray*}
Thus, letting $\mu_{y,y'}^{(1)} = \sum_{x\in X} \rho_{x,x,y,y'}$, we have that 
the left hand side of (\ref{eq_sbb}) coincides with 
\begin{eqnarray*}
&&\sum_{a\in A} \sum_{x,x'\in X} \sum_{y,y'\in X} \rho_{x,x',y,y'}s\left(e_{x,x',a,a}\otimes e_{y,y',b',b}\right)\\
&&= \sum_{x,x'\in X} \sum_{y,y'\in X} \rho_{x,x',y,y'}s\left(\left(\sum_{a\in A} e_{x,x',a,a}\right)\otimes e_{y,y',b',b}\right)\\
&&= \sum_{x,x'\in X} \sum_{y,y'\in X} \rho_{x,x',y,y'}\delta_{x,x'}s\left(1\otimes e_{y,y',b',b}\right)
= s\left(1\otimes \sum_{y,y'\in X}\mu_{y,y'}^{(1)}e_{y,y',b',b}\right)\\
&&= s\circ({\rm id}\otimes\partial)^{-1}\left(1\otimes \sum_{y,y'\in X}\mu_{y,y'}^{(1)}e_{y',y,b,b'}^{\rm op}\right).
\end{eqnarray*} 
Similarly, 
letting $\mu_{x,x'}^{(2)} = \sum_{y\in X}\rho_{x,x',y,y}$, we have that the right hand side of (\ref{eq_sbb}) 
coincides with
\begin{eqnarray*}
&&\sum_{a\in A} \sum_{x,x'\in X} \sum_{y,y'\in X} \rho_{x,x',y,y'}s\left(e_{x,x',b,b'}\otimes e_{y,y',a,a}\right)\\
&&= \sum_{x,x'\in X} \sum_{y,y'\in X} \rho_{x,x',y,y'}s\left(e_{x,x',b,b'}\otimes \left(\sum_{a\in A}e_{y,y',a,a}\right)\right)\\
&&= \sum_{x,x'\in X} \sum_{y,y'\in X} \rho_{x,x',y,y'}\delta_{y,y'}s\left(e_{x,x',b,b'}\otimes 1\right)\\
&&= s\left(\sum_{x,x'\in X}\mu_{x,x'}^{(2)}e_{x,x',b,b'}\otimes 1\right),
\end{eqnarray*}
that is, with
$$s\circ({\rm id}\otimes\partial)^{-1}\left(\sum_{y,y'\in X}\mu_{y',y}^{(2)}e_{y',y,b,b'}\otimes 1\right).$$
Let now $\rho\in\Sigma_{X}$. By (\ref{rho}), $\mu_{y,y'}^{(1)} = \mu_{y',y}^{(2)}$.  
Hence, if $s\circ({\rm id}\otimes\partial)^{-1}$
is fair then $\Gamma(\rho)\in \Sigma_{A}$, that is, $\Gamma$ is fair. 

Conversely, assuming that $\Gamma$ is fair, the previous paragraph shows that 
\begin{equation}\label{seqop}
s\circ({\rm id}\otimes\partial)^{-1}(u\otimes 1) = s\circ({\rm id}\otimes\partial)^{-1}(1\otimes u^{\rm op})
\end{equation}
for any $u$ of the form 
$u = \sum_{y,y'\in X}(\sum_x\rho_{x,x,y,y'})e_{y,y',b,b'}$ with $\rho\in \Sigma_{X}$. 
Letting
$\rho = e_xe_x^*\otimes e_xe_x^*\in\Sigma_{X}$ we conclude that 
(\ref{seqop}) holds for $u = e_{x,x,b,b'}$, $x\in X$, $b,b'\in A$. 
Letting $\rho = 1\otimes \omega^t + \omega\otimes 1$, where 
$\omega = \alpha (e_ze_{z}^*+e_{z'}e_{z'}^*)+\beta e_ze_{z'}^*+\bar\beta e_{z'}e_{z}^*$, $z\ne z'$, 
with $\alpha\geq |\beta|$, we obtain that
(\ref{seqop}) holds for 
$u = \alpha(2\sum_{y\in X}e_{y,y,b,b'}+|X|e_{z,z,b,b'}+|X|e_{z',z',b,b'})+\beta |X| e_{z',z,b,b'}+\bar\beta |X|e_{z,z',b,b'}$.
From this we deduce that (\ref{seqop}) holds for any 
$u = e_{y,y',b,b'}$, $y,y'\in X$, $b,b'\in A$. 
 \end{proof}

Let $\cl S\subseteq \cl B(K)$ be an operator system. We let 
$\cl S^{\rm op} = \{u^{\rm d} : u\in \cl S\}$, considered as an 
operator subsystem of $\cl B(K^{\rm d})$. Note that $\cl S^{\rm op}$ is well-defined: if 
$\phi : \cl S\to \cl B(\tilde{K})$ is a unital complete isometry, then the map $\tilde{\phi} : \cl S^{\rm op}\to \cl B(\tilde{K}^{\rm d})$,
given by $\tilde{\phi}(u^{\rm d}) = \phi(u)^{\rm d}$, is also unital and completely isometric. We thus write $u^{\rm op} = u^{\rm d}$ 
in the (abstract) operator system $\cl S^{\rm op}$. 

For a linear map $\Phi : M_X\to M_A$, let 
$\Phi^{\sharp} : M_X\to M_A$ be the linear map given by 
$\Phi^{\sharp}(\omega) = \Phi(\omega^{\rm t})^{\rm t}$.

\begin{lemma}\label{l_opopsys}
Let $\cl S$ be an operator system.
\begin{itemize}
\item[(i)]
If $\phi : \cl S\to \cl B(H)$ be a unital completely positive map 
then the map $\phi^{\rm op} : \cl S^{\rm op} \to \cl B(H^{\rm d})$, given by 
$\phi^{\rm op}(u^{\rm op}) = \phi(u)^{\rm d}$, is unital and completely positive.

\item[(ii)] Up to a canonical *-isomorphism, $C^*_u(\cl S^{\rm op}) = C^*_u(\cl S)^{\rm op}$.

\item[(iii)]
If $\Phi : M_X\to M_A$ is a completely positive map then so is $\Phi^{\sharp}$.
\end{itemize}
\end{lemma}

\begin{proof}
(i) 
Represent $\cl S\subseteq \cl B(K)$ as a concrete operator system. 
Then $\cl S^{\rm op}\subseteq \cl B(K^{\rm d})$. 
Suppose that $u_{i,j}\in \cl S$, $i,j = 1,\dots,n$, are such that 
$(u_{i,j}^{\rm d})_{i,j}\in M_n(\cl B(K^{\rm d}))^+$. 
Then 
$(u_{j,i})_{i,j} = (u_{i,j}^{\rm d})_{i,j}^{\rm d}\in M_n(\cl B(K))^+$
and hence $(\phi(u_{j,i}))_{i,j}\in M_n(\cl B(H))^+$. Thus, 
$$\left(\phi^{\rm op}(u_{i,j}^{\rm op})\right)_{i,j} = \left(\phi(u_{i,j})^{\rm d}\right)_{i,j} \in M_n\left(\cl B(H^{\rm d})\right)^+.$$

(ii) Suppose that $\psi : \cl S^{\rm op}\to \cl B(H)$ is a unital completely positive map. By
(i), $\psi^{\rm op} : \cl S\to \cl B(H^{\rm d})$ is (unital and) completely positive. 
By the universal property of the maximal C*-cover, there exists a *-homomorphism 
$\pi : C^*_u(\cl S) \to \cl B(H^{\rm d})$ extending $\psi^{\rm op}$. 
It follows that $\pi^{\rm op} : C^*_u(\cl S)^{\rm op} \to \cl B(H)$ is a *-homomorphism that extends $\psi$. 
Thus, $C^*_u(\cl S)^{\rm op}$ satisfies the universal property of the C*-cover of $\cl S^{\rm op}$.

(iii) The transposition is a (unital) complete order isomorphism from $M_X$ onto $M_X^{\rm op}$. 
The statement follows after observing that, under
the latter identification, $\Phi^{\sharp}$ coincides with $\Phi^{\rm op}$. 
\end{proof}

\begin{corollary}\label{c_lqnssy}
A local QNS correlation 
$\Gamma$ is fair if and only if 
$\Gamma = \sum_{i=1}^m$ $\lambda_i \Phi_i\otimes\Psi_i$ 
for some quantum channels $\Phi_i, \Psi_i : M_X\to M_A$ and scalars $\lambda_i\geq 0$, $i = 1,\dots,m$, 
$\sum_{i=1}^m \lambda_i = 1$, such that
\begin{equation}\label{eq_lsy}
\sum_{i=1}^m \lambda_i \Phi_i = \sum_{i=1}^m \lambda_i \Psi_i^\sharp.
\end{equation}
\end{corollary}

\begin{proof}
Suppose that $\Gamma$ is fair and, using Theorem \ref{sym}, 
write $\Gamma = \Gamma_s$, where 
$s$ is a state on $\omin(\cl T_{X,A}) \otimes_{\min} \omin(\cl T_{Y,B})$ such that $s\circ({\rm id}\otimes\partial)^{-1}$ is fair. 
As in the proof of Theorem \ref{t_locd}, 
identify $s$ with a convex combination $\sum_{i=1}^m \lambda_i \phi_i\otimes\psi_i$, where 
$\phi_i$ and $\psi_i$ are states on $\cl T_{X,A}$, $i = 1,\dots,m$; 
then the fairness condition is equivalent to 
\begin{equation}\label{eq_redu}
\sum_{i=1}^m \lambda_i \phi_i(u) = \sum_{i=1}^m \lambda_i \psi_i(\partial^{-1}(u^{\rm op})), \ \ u \in \cl T_{X,A}.
\end{equation}
Let $\Phi_i$ and $\Psi_i$ be the quantum channels from $M_X$ to $M_A$, 
corresponding to $\phi_i$ and $\psi_i$, respectively; then 
$\Gamma = \sum_{i=1}^m\lambda_i \Phi_i\otimes\Psi_i$.
Let
$\tilde \psi_i : u \mapsto \psi_i((\partial^{-1}(u^{\rm op}))$, $u\in \cl T_{X,A}$.  
By Lemma \ref{l_op}, $\tilde\psi_i$ is a state. 
Moreover, 
\begin{eqnarray*}
&&\langle\Psi_i^\sharp(e_xe_{x'}^*), e_ae_{a'}^*\rangle
= \langle\Psi_i(e_{x'}e_x^*)^t,e_ae_{a'}^*\rangle
= \langle\Psi_i(e_{x'}e_x^*),e_{a'}e_{a}^*\rangle\\
&& = \psi_i(e_{x',x,a',a})
= \psi_i(\partial^{-1}(e_{x,x',a,a'}^{\rm op}))
= \tilde\psi_i(e_{x,x',a,a'}),
\end{eqnarray*}
that is, the quantum channel $\Psi_i^\sharp$ corresponds to $\tilde \psi_i$.
Identity (\ref{eq_lsy}) now follows from (\ref{eq_redu}). 
The converse implication follows by reversing the previous steps.
\end{proof}

\begin{corollary}\label{c_fairo}
\begin{itemize}
\item[(i)] 
A CQNS correlation $\cl E$ is fair
if and only if there is a state  
$t : \cl R_{X,A}\otimes_{\max} \cl R_{X,A}\to\mathbb C$ such that  $t\circ({\rm id}\otimes\partial_{\cl B})^{-1}$ 
is fair and $\cl E = \cl E_t$.
Similar descriptions hold for fair correlations in the classes $\cl{CQ}_{\rm qc}$, 
$\cl{CQ}_{\rm qa}$ and $\cl{CQ}_{\rm loc}$.

\item[(ii)] An NS correlation $p$ is fair 
if and only if there is a state  
$t : \cl S_{X,A}\otimes_{\max} \cl S_{X,A}\to\mathbb C$ such that 
$t(u\otimes 1) = t(1\otimes u)$, $u\in \cl S_{X,A}$, and 
$$p(a,b|x,y) = t(e_{x,a}\otimes e_{y,b}), \ \ \ x,y\in X, a,b\in A.$$
Similar descriptions hold for fair correlations in the classes $\cl{C}_{\rm qc}$, 
$\cl{C}_{\rm qa}$ and $\cl{C}_{\rm loc}$.
\end{itemize}
\end{corollary}

\begin{proof}
We only give details for (i). 
Let $\cl E : \cl D_{XY}\to M_{AB}$ be a fair CQNS correlation. By (\ref{rho3}), 
$\cl E \circ \Delta_{XY} : M_{XY}\to M_{AB}$ is a fair QNS correlation. 
By Theorem \ref{sym} (i), $\cl E \circ \Delta_{XY} = \Gamma_s$, for some state $s$ on 
$\cl T_{X,A}\otimes_{\max} \cl T_{X,A}$ such that $s\circ({\rm id}\otimes\partial)^{-1}$ is fair.
It follows that $\cl E = \cl E_t$, where 
$t := s\circ (\beta_{X,A}\otimes \beta_{X,A})$ is a  state on 
$\cl R_{X,A}\otimes_{\max} \cl R_{X,A}$ and  $t\circ({\rm id}\otimes\partial_{\cl B})^{-1}$ is fair. 
Conversely, if $\cl E = \cl E_t$ for some  state $t$ on $\cl R_{X,A}\otimes_{\max} \cl R_{X,A}$  such that $t\circ({\rm id}\otimes\partial_{\cl B})^{-1}$ is fair then
$\Gamma_{\cl E} = \Gamma_s$, where $s := t\circ (\beta'_{X,A}\otimes \beta'_{X,A})$ and $s\circ({\rm id}\otimes\partial)^{-1}$ is fair. 
By Theorem \ref{sym} (i), $\Gamma_{\cl E}$ is fair, and hence so is $\cl E$.
The statements regarding $\cl{CQ}_{\rm qc}$, 
$\cl{CQ}_{\rm qa}$ and $\cl{CQ}_{\rm loc}$ follow after a straightforward modification of the argument. 
\end{proof}

\noindent {\bf Remark. } It follows from Theorem \ref{th_psstw}, Theorem \ref{sym} and Corollary \ref{c_fairo}
that fair correlations can be viewed as a non-commutative, and less restrictive, version of synchronous correlations.


\subsection{Tracial QNS correlations}\label{ss_tc}

Let $\cl A$ be a unital C*-algebra, $\tau : \cl A\to \bb{C}$ be a state 
and  $\cl A^{\rm op}$ be  the opposite C*-algebra of $\cl A$.
By the paragraph before Theorem 6.2.7 in \cite{bo}, the linear functional
$s_{\tau} : \cl A\otimes_{\max}\cl A^{\rm op}\to \bb{C}$, 
given by $s_{\tau}(u\otimes v^{\rm op}) = \tau(uv)$, is a state. 

A positive element $E\in M_X\otimes M_A\otimes \cl A$ will be called a \emph{stochastic $\cl A$-matrix} if 
$(\id\otimes\id\otimes\pi)(E)$ is a stochastic operator matrix for some faithful *-representation of $\cl A$. 
Such an $E$ will be called \emph{semi-classical}
if it belongs to $\cl D_X\otimes M_A\otimes \cl A$. 

Let $E = (g_{x,x',a,a'})_{x,x',a,a'}$ be a stochastic $\cl A$-matrix, and set 
$$E^{\rm op} = (g_{x',x,a',a}^{\rm op})_{x,x',a,a'} \in M_X\otimes M_A\otimes \cl A^{\rm op};$$
Lemma \ref{l_op} shows that 
$E^{\rm op}$ is a stochastic $\cl A^{\rm op}$-matrix. 
Thus, after a permutation of the tensor factors, we can consider 
$E\otimes E^{\rm op}$ as an element of $\left(M_{XA}\otimes M_{XA}\otimes(\cl A\otimes_{\max} \cl A^{\rm op})\right)^+$. 
By Theorem \ref{p_ucptau2}, there exists a *-homomorphism
$\pi_E : \cl C_{X,A}\to \cl A$, such that $\pi_E(e_{x,x',a,a'}) = g_{x,x',a,a'}$ for all $x,x',a,a'$.
By Corollary \ref{c_ucst} and Lemma \ref{l_opopsys}, $C^*_u(\cl T_{X,A}^{\rm op}) \equiv \cl C_{X,A}^{\rm op}$;
thus, 
$$\cl T_{X,A}\otimes_{\rm c} \cl T_{X,A}^{\rm op} \subseteq_{\rm c.o.i.} \cl C_{X,A} \otimes_{\max} \cl C_{X,A}^{\rm op}.$$
Write 
\begin{equation}\label{eq_fEtau}
f_{E,\tau} = s_{\tau}\circ (\pi_E\otimes \pi_E^{\rm op})\circ(\id\otimes\partial);
\end{equation}
we have that $f_{E,\tau}$ is a state on $\cl T_{X,A}\otimes_{\rm c}\cl T_{X,A}$, and 
$$f_{E,\tau}(e_{x,x',a,a'} \otimes e_{y,y',b,b'}) = \tau(g_{x,x',a,a'} g_{y',y,b',b}), 
x,x',y,y'\in X, a,a',b,b'\in A.$$
In the sequel, we write $\Gamma_{E,\tau} = \Gamma_{f_{E,\tau}}$; 
by Theorem \ref{t_dqc}, $\Gamma_{E,\tau} \in \cl Q_{\rm qc}$.
By Theorem \ref{p_ucptau2}, we may assume, without loss of generality, that 
$\cl A = \cl C_{X,A}$ and $E = (e_{x,x',a,a'})_{x,x',a,a'}$. In this case, we will abbreviate 
$\Gamma_{E,\tau}$ to $\Gamma_{\tau}$.

\begin{definition}\label{d_tracial}
A QNS correlation  $\Gamma$ is called 
\begin{itemize}
\item[(i)] \emph{tracial} if 
$\Gamma = \Gamma_{\tau}$, where $\tau : \cl C_{X,A}\to \bb{C}$ is a trace;

\item[(ii)] \emph{quantum tracial} if there exists a finite dimensional C*-algebra $\cl A$, 
a trace $\tau_{\cl A}$ on $\cl A$ and a *-homomorphism $\pi : \cl C_{X,A}\to \cl A$ such that 
$\Gamma = \Gamma_{\tau_{\cl A}\circ \pi}$;

\item[(iii)] \emph{locally tracial} if 
there exists an abelian C*-algebra $\cl A$, 
a state $\tau_{\cl A}$ on $\cl A$ and a *-homomorphism $\pi : \cl C_{X,A}\to \cl A$ such that 
$\Gamma = \Gamma_{\tau_{\cl A}\circ \pi}$.
\end{itemize}
\end{definition}

\begin{theorem}\label{th_belo}
Let $X$ and $A$ be finite sets.
\begin{itemize}
\item[(i)] 
If $\Gamma$ is a quantum tracial QNS correlation then $\Gamma\in \cl Q_{\rm q}$;
\item[(ii)]
A QNS correlation $\Gamma : M_{XX}\to M_{AA}$ is locally tracial if and only if 
there exists quantum channels $\Phi_j : M_X\to M_A$, $j = 1,\dots,k$, such that 
\begin{equation}\label{eq_Gsp}
\Gamma = \sum_{j=1}^k \lambda_j \Phi_j \otimes \Phi_j^{\sharp}
\end{equation}
as a convex combination. 
In particular, if $\Gamma$ is a locally tracial QNS correlation then $\Gamma\in \cl Q_{\rm loc}$.
\end{itemize}
\end{theorem}

\begin{proof}
(i) Suppose that $H$ is a finite dimensional 
Hilbert space on which $\cl A$ acts faithfully
and  let $\pi:\cl C_{X,A}\to\cl A$ be as in Definition \ref{d_tracial} (ii).
Let $E_{x,x',a,a'} = \pi(e_{x,x',a,a'})$ and 
$E = \left(E_{x,x',a,a'}\right)_{x,x',a,a'}$.
By the proof of Lemma \ref{l_op},
$E^{\rm op} := \left(E^{\rm d}_{x',x,a',a}\right)_{x,x',a,a'}$ is a stochastic operator matrix.
Let $\sigma$ be any positive functional on $\cl L(H\otimes H^{\rm d})$ that extends the state 
$s_{\tau_{\cl A}}$ which, by nuclearity, may be considered as a state on $\cl A\otimes_{\min}\cl A^{\rm op}$.
Then $\Gamma_{\tau} = \Gamma_{E\odot E^{\rm op},\sigma}$ and, by the paragraph 
before Remark \ref{r_qqnsns}, $\Gamma_{\tau}\in \cl Q_{\rm q}$. 

(ii) 
Suppose that $\Phi_j : M_X\to M_A$, $j = 1,\dots,k$, are quantum channels and 
$\Gamma$ is the convex combination (\ref{eq_Gsp}). 
Letting $\left(\lambda_{x,x',a,a'}^{(j)}\right)_{a,a'} = \Phi_j\left(e_x e_{x'}^*\right)$, $x,x'\in X$,
we have that the matrix $C_j = \left(\lambda_{x,x',a,a'}^{(j)}\right)_{x,x',a,a'}$ is a stochastic $\bb{C}$-matrix. 
By Theorem \ref{p_ucptau2}, there exists a (unique) *-representation $\pi_j : \cl C_{X,A}\to \bb{C}$ such that
$\pi_j(e_{x,x',a,a'}) = \lambda_{x,x',a,a'}^{(j)}$, $x,x'\in X$, $a,a'\in A$. 
Let $\pi : \cl C_{X,A} \to \cl D_k$ be the *-representation given by 
$$\pi\left(u\right) = \sum_{j=1}^k \pi_j\left(u\right) e_j e_j^*, \ \ \ u\in \cl C_{X,A},$$
and let $\tau_k : \cl D_k\to \bb{C}$ be the state defined by 
$\tau_k\left((\mu_j)_{j=1}^k\right) = \sum_{j=1}^k \lambda_j \mu_j$. 
We have 
\begin{eqnarray*}
& & \Gamma_{\tau_k\circ\pi}\left(e_{x}e_{x'}^* \otimes e_{y}e_{y'}^*\right)\\
& = & 
\sum_{a,a'\in A} \sum_{b,b'\in B} (\tau_k\circ\pi)(e_{x,x',a,a'}e_{y',y,b',b}) e_a e_{a'}^* \otimes e_b e_{b'}^* \\
& = & 
\sum_{a,a'\in A} \sum_{b,b'\in B} \tau_k\left(\sum_{j=1}^k \pi_j(e_{x,x',a,a'}e_{y',y,b',b}) e_j e_j^*\right)
e_a e_{a'}^* \otimes e_b e_{b'}^* \\
& = & 
\sum_{j=1}^k \lambda_j
\sum_{a,a'\in A} \sum_{b,b'\in B} \lambda_{x,x',a,a'}^{(j)} \lambda_{y',y,b',b}^{(j)}
e_a e_{a'}^* \otimes e_b e_{b'}^* \\
& = & 
\sum_{j=1}^k \lambda_j
\left(\sum_{a,a'\in A}  \lambda_{x,x',a,a'}^{(j)} e_a e_{a'}^*\right)
\otimes
\left(\sum_{b,b'\in B} \lambda_{y',y,b',b}^{(j)} e_b e_{b'}^*\right) \\
& = & 
\sum_{j=1}^k \lambda_j \Phi_j\left(e_{x}e_{x'}^*\right)\otimes  \Phi_j^{\sharp}\left(e_{y} e_{y'}^*\right).
\end{eqnarray*}

Conversely, 
let $\cl A$ be a unital abelian C*-algebra, $\tau_{\cl A} : \cl A\to \bb{C}$ a state, and  
$\pi : \cl C_{X,A}\to \cl A$ a *-homomorphism such that $\Gamma = \Gamma_{\tau_{\cl A}\circ\pi}$.
Without loss of generality,
assume that $\cl A = C(\Omega)$, where $\Omega$ is a compact Hausdorff topological space, 
and $\mu$ is a Borel probability measure on $\Omega$ such that 
$\tau_{\cl A}(f) = \int_{\Omega} fd\mu$, $f\in \cl A$.
Set $h_{x,x',a,a'} = \pi(e_{x,x',a,a'})$, $x,x'\in X$, $a,a'\in A$. 
For each $s\in \Omega$, let 
$\Phi(s) : M_X\to M_A$ be the quantum channel given by 
$\Phi(s)\left(e_{x}e_{x'}^*\right) = \left(h_{x,x',a,a'}(s)\right)_{a,a'}$. 
We have 
\begin{eqnarray*}
& & \Gamma\left(e_{x}e_{x'}^* \otimes e_{y}e_{y'}^*\right)\\
& = & 
\sum_{a,a'\in A} \sum_{b,b'\in A} 
\left(\int_{\Omega} h_{x,x',a,a'}(s) h_{y',y,b',b}(s) d\mu(s)\right) e_{a}e_{a'}^* \otimes e_{b}e_{b'}^*\\
& = & 
\int_{\Omega} \Phi(s)\left(e_{x}e_{x'}^*\right) \otimes \Phi(s)^{\sharp}\left(e_{y}e_{y'}^*\right) d\mu(s).
\end{eqnarray*}
It follows that $\Gamma$ is in the closed hull of the set of all 
correlations of the form (\ref{eq_Gsp}).
An argument using Carath\'{e}odory's Theorem,
similar to the one in the proof of Remark \ref{r_lqqnsi}, 
shows that $\Gamma$ has the form (\ref{eq_Gsp}).
\end{proof}

\begin{remark}\label{r_contra}
{\bf (i) } 
{\rm Every tracial QNS correlation $\Gamma = \Gamma_{E,\tau}$ is fair. Indeed, writing 
$E = (g_{x,x',a,a'})$, we have
\begin{eqnarray*}
f_{E,\tau} \circ (\id\otimes\partial)^{-1}(e_{x,x',a,a'}\otimes 1) 
& = & \tau(g_{x,x',a,a'})\\ 
& = & f_{E,\tau} \circ (\id\otimes\partial)^{-1}(1 \otimes e_{x,x',a,a'}^{\rm op}).
\end{eqnarray*}
It can be seen from Corollary \ref{c_lqnssy} and Theorem \ref{th_belo} 
(see the closing remarks of this section) that the converse does not hold true. }

\smallskip

{\bf (ii) } 
{\rm The set of all tracial (resp. quantum tracial, locally tracial) QNS correlations over $(X,A)$ is convex. 
Indeed, suppose that $\cl A$ (resp. $\cl B$) is a unital C*-algebra, $\tau_{\cl A}$ 
(resp. $\tau_{\cl B}$) a trace on $\cl A$ and $E$ (resp. $F$) a stochastic $\cl A$-matrix (resp. a stochastic $\cl B$-matrix). 
Let $\lambda \in (0,1)$, $\cl C = \cl A\oplus \cl B$, $\tau : \cl C\to \bb{C}$ be given by 
$\tau(u\oplus v) = \lambda \tau_{\cl A}(u) + (1-\lambda) \tau_{\cl B}(v)$, and $G = E\oplus F$, considered as an 
element of $M_X\otimes M_A\otimes \cl C$. Then $G$ is a stochastic $\cl C$-matrix and 
$$\lambda \Gamma_{E,\tau_{\cl A}} + (1-\lambda) \Gamma_{F,\tau_{\cl B}} = \Gamma_{G,\tau}.$$
}


{\bf (iii) } 
{\rm 
It is straightforward from Theorem \ref{th_psstw} that, if $p\in \cl C_{\rm qc}$ 
(resp. $p\in \cl C_{\rm q}$, $p\in \cl C_{\rm loc}$) is synchronous then 
$\Gamma_p$ is a tracial (resp. quantum tracial, locally tracial) QNS correlation. 
By \cite[Theorem 4.2]{dpp}, the set $\cl C_{\rm q}^{\rm s}$ of synchronous quantum NS correlations is not closed 
if $|X| = 5$ and $|A| = 2$.
Let $p\in \overline{\cl C_{\rm q}^{\rm s}}\setminus \cl C_{\rm q}^{\rm s}$. 
Then $p$ is a synchronous NS correlation and does not lie in $\cl C_{\rm q}$.
Assume that $\Gamma_p$ is quantum tracial. By Theorem \ref{th_belo}, 
$\Gamma_p\in \cl Q_{\rm q}$ and hence, by Remark \ref{th_onedq}, $p\in \cl C_{\rm q}$, a contradiction.  
It follows that the set of quantum tracial NS correlations is not closed. 
}

\smallskip

{\bf (iv) } 
{\rm 
The set of all tracial QNS correlations is closed; this can be seen via a standard argument (see e.g. \cite{mr2}):
Assuming that $(\Gamma_n)_{n\in \bb{N}}$ is a sequence of tracial QNS correlations 
converging to the QNS correlation $\Gamma$,
let $\cl A_n$ be a unital C*-algebra with a trace $\tau_n$, and 
$E_n = \left(g^{(n)}_{x,x',a,a'}\right)$ be a stochastic $\cl A_n$-matrix 
such that $\Gamma_n = \Gamma_{E_n,\tau_n}$. 
Let $\cl A$ be the tracial ultraproduct of the family $\{(\cl A_n,\tau_n)\}_{n\in \bb{N}}$ with respect to a non-trivial 
ultrafilter $\frak{u}$ \cite[Section 4]{had-li}.
Write $\tau$ for the trace on $\cl A$ and $E = (g_{x,x',a,a'})$ for the class of $\oplus_{n\in \bb{N}} E_n$ in $\cl A$. 
Then 
\begin{eqnarray*}
\left\langle \Gamma(e_x e_{x'}^* \otimes e_y e_{y'}^*),e_a e_{a'}^* \otimes e_b e_{b'}^* \right\rangle
& = & 
\lim_{n\to \infty} \tau_n\left(g^{(n)}_{x,x',a,a'} g^{(n)}_{y',y,b',b}\right)\\ 
& = & 
\tau\left(g_{x,x',a,a'} g_{y',y,b',b}\right).
\end{eqnarray*} 
}
\end{remark}

We next show that the class of all tracial QNS correlations, 
as well as each of the the subclasses of quantum tracial and locally tracial QNS correlations, 
have natural classes of invariant states. 
Given a unital C*-algebra $\cl A$, a trace $\tau : \cl A\to \bb{C}$
and a stochastic $\cl A$-matrix $E = (g_{z,z'})_{z,z'} \in \cl L(\bb{C})\otimes M_Z\otimes \cl A$, 
let $\omega^{E,\tau} = \left(\omega^{E,\tau}_{z,z',u,u'}\right)\in M_{ZZ}$ be 
defined by 
$$\omega^{E,\tau}_{z,z',u,u'} = \tau(g_{z,z'} g_{u',u}), \ \ \ z,z',u,u'\in Z.$$
Equivalently, let $E^{\rm op}$ be the stochastic $\cl A^{\rm op}$-matrix $\left(g_{u',u}^{\rm op}\right)$, 
and recall that
$s_{\tau} : \cl A\otimes_{\max}\cl A^{\rm op} \to \bb{C}$ is the state given by 
$s_{\tau}(u\otimes v^{\rm op}) = \tau(uv)$.
Then $\omega^{E,\tau} = L_{s_{\tau}}\left(E\otimes E^{\rm op}\right),$ 
where
$$L_{s_{\tau}} : M_{ZZ}\otimes \left(\cl A\otimes_{\max} \cl A^{\rm op}\right) \to M_{ZZ}$$
is the corresponding slice. It follows that $\omega^{E,\tau}$ is a state.

\begin{definition}\label{d_clexs}
Let $Z$ be a finite set. A state $\omega\in M_{ZZ}$ is called
\begin{itemize}
\item[(i)] \emph{C*-reciprocal} if there exists a unital C*-algebra algebra $\cl A$,
a trace $\tau$ on $\cl A$ and a stochastic $\cl A$-matrix $E\in M_Z\otimes \cl A$ such that 
$\omega = \omega^{E,\tau}$;
\item[(ii)] \emph{quantum reciprocal} if it is C*-reciprocal, and the C*-algebra $\cl A$ from 
(i) can be chosen to be finite dimensional;
\item[(iii)] \emph{locally reciprocal} if it is C*-reciprocal, and the 
C*-algebra $\cl A$ from (i) can be chosen to be abelian.
\end{itemize}
\end{definition}

We will denote by $\Upsilon(Z)$ 
(resp. $\Upsilon_{\rm q}(Z)$, $\Upsilon_{\rm loc}(Z)$)
the set of all C*-reciprocal (resp. quantum reciprocal, locally reciprocal) states in $M_{ZZ}$.

\begin{theorem}\label{p_ee}
Let $\Gamma$ be a QNS correlation.
\begin{itemize}
\item[(i)] If $\Gamma$ is tracial then $\Gamma\left(\Upsilon(X)\right)\subseteq \Upsilon(A)$;
\item[(ii)] If $\Gamma$ is quantum tracial then  $\Gamma\left(\Upsilon_{\rm q}(X)\right)\subseteq \Upsilon_{\rm q}(A)$;
\item[(iii)] If $\Gamma$ is locally tracial then $\Gamma\left(\Upsilon_{\rm loc}(X)\right)\subseteq \Upsilon_{\rm loc}(A)$.
\end{itemize}
\end{theorem}

\begin{proof}
(i) 
Let $\tau$ be a trace on $\cl C_{X,A}$, $\cl A$ be a C*-algebra, $\tau_{\cl A}$ be a trace on $\cl A$, and 
$E = (g_{x,x'})_{x,x'}\in M_X\otimes \cl A$ be a stochastic $\cl A$-matrix. 
Set $\omega = \Gamma_{\tau}\left(\omega^{E,\tau_{\cl A}}\right)$ and write $\omega = \left(\omega_{a,a',b,b'}\right)_{a,a',b,b'}$. 
Let $\cl B = \cl A\otimes_{\max} \cl C_{X,A}$ and 
$\tau_{\cl B} = \tau_{\cl A}\otimes \tau$ be the product trace on $\cl B$ \cite[Proposition 3.4.7]{bo}. 
Set 
$$h_{a,a'} = \sum_{x,x'\in X} g_{x,x'} \otimes e_{x,x',a,a'}, \ \ \ a,a'\in A;$$
thus, $F := (h_{a,a'})_{a,a'}\in M_A\otimes \cl B$. Moreover, 
\begin{eqnarray*}
\Tr\mbox{}_A F 
& = & \sum_{a\in A} h_{a,a} 
= \sum_{x,x'\in X} g_{x,x'} \otimes \left(\sum_{a\in A} e_{x,x',a,a}\right)\\
& = & 
\sum_{x,x'\in X} g_{x,x'} \otimes \delta_{x,x'}1
= 
\sum_{x \in X} g_{x,x} \otimes 1 = 1_{\cl B}.
\end{eqnarray*}
To see that $F$ is positive, we assume that $\cl A$ and $\cl C_{X,A}$ are faithfully 
represented and let $V_x$ and $V_{a,x}$ be operators such that 
$(V_x)_x$ is a row isometry, $(V_{a,x})_{a,x}$ is an isometry,
$g_{x,x'} = V_x^*V_{x'}$ and $e_{x,x',a,a'} = V_{a,x}^*V_{a',x'}$, $x,x' \in X$, $a,a'\in A$. 
Letting 
$W = \left(\sum_{x\in X} V_x\otimes V_{a,x}\right)_{a\in A}$, considered as row operator, we have that $F = W^*W$.
Hence $F$ is a stochastic $\cl B$-matrix. 
In addition, for $a,a',b,b'\in A$ we have 
$$\tau_{\cl B}\left(h_{a,a'}h_{b',b}\right) 
= \sum_{x,x'\in X}\sum_{y,y'\in X} \tau_{\cl A}\left(g_{x,x'} g_{y',y}\right) \tau\left(e_{x,x',a,a'}e_{y',y,b',b}\right)
= \omega_{a,a',b,b'},$$
implying that $\omega = \omega^{F,\tau_{\cl B}}$.

(ii) and (iii)
follow from the fact that if the C*-algebra $\cl A$ is finite dimensional (resp. abelian) 
and $\tau$ factors through a finite-dimensional 
(resp. abelian) C*-algebra then so does $\tau_{\cl B}$.
\end{proof}

\begin{remark}\label{r_rec}
{\rm 
{\bf (i) } 
The state $\frac{1}{|X|^2}I_{XX}$ is locally reciprocal, and hence
it follows from Theorem \ref{p_ee} that 
\begin{eqnarray}\label{eq_Ypsloc}
\Upsilon_{\rm loc} (A) 
& = & 
\left\{\frac{1}{|X|^2}\Gamma(I_{XX}) : X \mbox{ finite, } \Gamma \mbox{ loc. tracial QNS correlation}\right\}\nonumber\\
& = & 
\left\{\Gamma(1) : \Gamma : \bb{C}\to M_{AA} \mbox{ loc. tracial QNS correlation}\right\}.
\end{eqnarray}
Similar descriptions hold for $\Upsilon_{\rm q}(A)$ and $\Upsilon(A)$. 
Remark \ref{r_contra} thus implies that the sets $\Upsilon(A)$, 
$\Upsilon_{\rm q}(A)$ and $\Upsilon_{\rm loc}(A)$ are convex and the sets 
$\Upsilon(A)$ and $\Upsilon_{\rm loc}(A)$ are closed.

\smallskip

{\bf (ii) } 
Recall that a state $\rho \in M_{XX}$ is called de Finetti  \cite{ckr} if 
there exist states $\omega_i\in M_Z$, $i = 1,\dots, k$, such that
$\rho = \sum_{j=1}^k \lambda_j \omega_j \otimes \omega_j$
as a convex combination.
By (\ref{eq_Ypsloc}) and Theorem \ref{th_belo}, 
$$\Upsilon_{\rm loc} (X) = {\rm conv}\left\{\omega \otimes \omega^{\rm t} : \omega \mbox{ a state in } M_X\right\}.$$
Thus, the locally reciprocal states can be viewed as twisted de Finetti states. 
The presence of the transposition in our case is required in view of the necessity to employ opposite C*-algebras.
Thus, quantum reciprocal states can be viewed as an entanglement assisted version of (twisted) de Finetti states, while
C*-reciprocal states -- as their commuting model version.

\smallskip

{\bf (iii) } 
C*-reciprocal states are closely related to factorisable channels introduced in \cite{delaroche}
(see also \cite{haag-musat, mr}, to which we refer the reader for the definition used here).
Indeed, factorisable channels have Choi matrices of the form 
$\tau(g_{x,x'}h_{y',y})_{x,x',y,y'}$, where $\tau$ is a faithful normal trace on a von Neumann algebra $\cl A$, 
and $(g_{x,x'})_{x,x'}$ and $(h_{y,y'})_{y,y'}$ are matrix unit systems -- a special type of stochastic $\cl A$-matrices
(see \cite[Proposition 3.1]{mr}).
Equivalently, the Choi matrices of factorisable channels 
$\Phi : M_X\to M_X$ can be described \cite[Definition 3.1]{haag-musat} as 
the matrices of the form $\left(\tau(v_{a,x}^*v_{a',x'})\right)_{x,x',a,a'}$, where 
$V = (v_{a,x})_{a,x}\in M_X(\cl A)$ is a \emph{unitary} matrix. 
Note that, if $E$ is the stochastic operator matrix corresponding to $V$,
then the QNS correlation $\Gamma = \Gamma_{E,\tau}$ has marginal channels $\Gamma_A(\cdot) = \Gamma(\cdot\otimes I)$
and $\Gamma_B(\cdot) = \Gamma(I\otimes \cdot)$ that coincide with $\Phi$. 
We can thus view tracial QNS correlations as generalised couplings of factorisable channels.
Here, by a coupling of the pair $(\Phi,\Psi)$ of channels, we mean a channel $\Gamma$ with $\Gamma_A = \Phi$ and 
$\Gamma_B = \Psi$ -- a generalisation of classical coupling of probability distributions in the sense of 
optimal transport \cite{villani}. 
}
\end{remark}


\subsection{Tracial CQNS correlations}\label{ss_tccq}

In this subsection, we define a tracial version of CQNS correlations. 
Let $\cl A$ be a unital C*-algebra, $\tau : \cl A\to \bb{C}$ be a trace and 
$E\in \cl D_X\otimes M_A\otimes \cl A$ be a semi-classical stochastic $\cl A$-matrix. 
Write $E = (g_{x,a,a'})_{x,a,a'}$; 
thus, $(g_{x,a,a'})_{a,a'}\in \left(M_A\otimes \cl A\right)^+$ and $\sum_{a\in A} g_{x,a,a} = 1$, for each $x\in X$.
Set $E^{\rm op} = (g_{x,a',a}^{\rm op})_{x,a,a'}$; thus, $E^{\rm op} \in \cl D_X\otimes M_A\otimes \cl A^{\rm op}$
and Lemma \ref{l_op} shows that
$E^{\rm op}$ is a semi-classical stochastic $\cl A^{\rm op}$-matrix. 
Let $\phi_{E,x} : M_A\to \cl A$ be the unital completely positive map given by 
$\phi_{E,x}(e_ae_{a'}^*) = g_{x,a,a'}$. By Boca's Theorem \cite{boca}, there exists a unital completely positive map 
$\phi_E : \cl B_{X,A}\to \cl A$ such that $\phi_E(e_{x,a,a'}) = g_{x,a,a'}$, $x\in X$, $a,a'\in A$.
Let $\phi_E^{\rm op} : \cl B_{X,A}^{\rm op}\to \cl A^{\rm op}$ be the
map given by $\phi_E^{\rm op}(u^{\rm op}) = \phi_{E}(u)^{\rm op}$, which is completely positive by Lemma \ref{l_opopsys}. 
Write 
$$f_{E,\tau} = s_{\tau}\circ (\phi_E\otimes \phi_E^{\rm op})\circ(\id\otimes\partial_{\cl B});$$
thus, by (\ref{eq_intomxy}) $f_{E,\tau}$ is a state on $\cl R_{X,A}\otimes_{\rm c}\cl R_{X,A}$. 
Note that
$$f_{E,\tau}\left(e_{x,a,a'} \otimes e_{y,b,b'}\right) = \tau\left(g_{x,a,a'} g_{y,b',b}\right), \ 
x,y \in X, a,a',b,b'\in A.$$
In the sequel, we write $\cl E_{E,\tau} = \cl E_{f_{E,\tau}}$; 
by Theorem \ref{th_cqdes}, $\cl E_{E,\tau} \in \cl{CQ}_{\rm qc}$.

\begin{definition}\label{d_cqtr}
A CQNS correlation $\cl E$ is called 
\begin{itemize}
\item[(i)] \emph{tracial} if 
$\cl E = \cl E_{E,\tau}$, where $E\in \cl D_X\otimes M_A\otimes \cl A$ is a semi-classical 
stochastic $\cl A$-matrix for some unital C*-algebra $\cl A$ and $\tau : \cl A\to \bb{C}$ is a trace;

\item[(ii)] \emph{quantum tracial} if 
it is tracial and the C*-algebra as in (i) can be chosen to be finite dimensional;
\item[(iii)] \emph{locally tracial} if it 
it is tracial and the C*-algebra as in (i) can be chosen to be abelian.
\end{itemize}
\end{definition}

\begin{proposition}\label{p_exex34}
Let $\cl E : \cl D_{XX}\to M_{AA}$ be a CQNS correlation. 
\begin{itemize}
\item[(i)] If $\cl E$ is quantum tracial then $\cl E\in \cl{CQ}_{\rm q}$;
\item[(ii)] $\cl E$ is locally tracial if and only if 
there exist channels $\cl E_j : \cl D_X\to M_A$, $j = 1,\dots,k$, such that 
\begin{equation}\label{eq_Esh}
\cl E = \sum_{j=1}^k \lambda_j \cl E_j\otimes \cl E_j^{\sharp}.
\end{equation} 
In particular, 
if $\cl E$ is locally tracial then $\cl E\in \cl{CQ}_{\rm loc}$.
\end{itemize}
\end{proposition}

\begin{proof}
(i) 
Suppose that $\cl E$ is quantum tracial and write 
$\cl E = \cl E_{E,\tau}$, where $E = (g_{x,a,a'})_{x,a,a'} \in \cl D_X\otimes M_A\otimes \cl A$ is a semi-classical 
stochastic $\cl A$-matrix for some finite dimensional C*-algebra $\cl A$ and a trace $\tau : \cl A\to \bb{C}$.
The matrix $\tilde{E} = \left(\delta_{x,x'}g_{x,a,a'}\right)_{x,x',a,a'}$ is a stochastic matrix in  
$M_X\otimes M_A\otimes \cl A$ and hence gives rise, via Theorem \ref{p_ucptau2}, to a 
canonical *-homomorphism $\pi_{\tilde{E}} : \cl C_{X,A}\to \cl A$. Letting $\tilde{\tau} = \tau\circ \pi_{\tilde{E}}$, we have 
that $\tilde{\tau}$ is a trace on $\cl C_{X,A}$ and 
$\Gamma_{\cl E} = \Gamma_{\tilde{\tau}}$. 
Thus, $\Gamma_{\cl E} \in \cl Q_{\rm q}$.
By Remark \ref{th_onedq}, $\cl E\in \cl{CQ}_{\rm q}$.

(ii) We fix $\cl A$, $\tau$ and $E$ as in (i), with $\cl A$ abelian. The trace $\tilde{\tau}$, 
defined in the proof of (i), now factors through 
an abelian C*-algebra, and hence $\Gamma_{\cl E}$ is locally tracial.  
By Theorem \ref{th_belo}, there exists 
quantum channels $\Phi_j : M_X\to M_A$, $j = 1,\dots,k$, such that
$\Gamma_{\cl E} = \sum_{j=1}^k \Phi_j\otimes \Phi_j^{\sharp}$ as a convex combination. 
Letting $\cl E_j = \Phi_j|_{\cl D_X}$, $j = 1,\dots,k$, we see that $\cl E$ has the form (\ref{eq_Esh}). 

Conversely, suppose that $\cl E$ has the form (\ref{eq_Esh}).
By Theorem \ref{th_belo}, 
there exists an abelian C*-algebra $\cl A$, a *-representation $\pi : \cl C_{X,A}\to \cl A$ and a trace 
$\tau$ on $\cl A$ such that $\Gamma_{\cl E} = \Gamma_{\tau\circ\pi}$. 
The stochastic operator matrix $E = \left(\pi(e_{x,x,a,a'})\right)_{x,a,a'}$ is semi-classical and $\cl E = \cl E_{E,\tau}$.
\end{proof}

We now specialise Definition \ref{d_clexs} to states in $\cl D_{XX}$, that is, 
bipartite probability distributions. 
A probability distribution $q = (q(x,y))_{x,y\in X}$ on $X\times X$ will be called 
\emph{C*-reciprocal} if there exists a C*-algebra $\cl A$, a POVM $(g_x)_{x\in X}$ in $\cl A$ and a trace 
$\tau : \cl A\to \bb{C}$ such that $q(x,y) = \tau(g_x g_y)$, $x,y\in X$. 
If $\cl A$ can be chosen to be finite dimensional (resp. abelian), we call $q$ \emph{quantum reciprocal}
(resp. \emph{locally reciprocal}). 
We denote by $\Upsilon^{\rm cl}(X)$ 
(resp. $\Upsilon^{\rm cl}_{\rm q}(X)$, $\Upsilon^{\rm cl}_{\rm loc}(X)$)
the (convex) set of all C*-reciprocal (resp. quantum reciprocal, locally reciprocal) 
probability distributions on $X\times X$.

It can be seen as in Remark \ref{r_rec}
that the class of locally reciprocal probability distributions coincides with the 
well-known class of \emph{exchangeable} probability distributions, that is, the 
convex combinations of the form 
$$q(x,y) = \sum_{i=1}^n \lambda_i q_i(x) q_i(y), \ \ \ x,y\in X,$$
where $q_i$ is a probability distribution on $X$, $i = 1,\dots,n$.
Thus, C*-reciprocal and quantum reciprocal probability distributions can be viewed as
quantum versions of exchangeable distributions.

It is straightforward to see that, writing $\Delta = \Delta_{XX}$, we have 
$$\Delta(\Upsilon(X)) = \Upsilon^{\rm cl}(X), \ 
\Delta(\Upsilon_{\rm q}(X)) = \Upsilon^{\rm cl}_{\rm q}(X) \mbox{ and } 
\Delta(\Upsilon_{\rm loc}(X)) = \Upsilon^{\rm cl}_{\rm loc}(X).$$
These relations, combined with Theorem \ref{p_ee}, easily yield the following proposition, 
whose proof is omitted.

\begin{proposition}\label{p_cqnsexp}
Let $\cl E : \cl D_{XX}\to M_{AA}$ be a CQNS correlation. 
\begin{itemize}
\item[(i)] If $\cl E$ is tracial then $\cl E\left(\Upsilon^{\rm cl}(X)\right)\subseteq \Upsilon(A)$;
\item[(ii)] If $\cl E$ is quantum tracial then 
$\cl E\left(\Upsilon_{\rm q}^{\rm cl}(X)\right)\subseteq \Upsilon_{\rm q}(A)$;
\item[(iii)] If $\cl E$ is locally tracial then 
$\cl E\left(\Upsilon_{\rm loc}^{\rm cl}(X)\right)\subseteq \Upsilon_{\rm loc}(A)$.
\end{itemize}
\end{proposition}


\subsection{Tracial NS correlations}\label{ss_tns}

The correlation classes introduced in Sections \ref{ss_tc} and \ref{ss_tccq} have a natural 
NS counterpart.
For a C*-algebra $\cl A$, equipped with a trace $\tau$, and a classical stochastic $\cl A$-matrix
$E\in \cl D_X\otimes \cl D_A\otimes \cl A$, say, $E = (g_{x,a})_{x,a}$ (so that $g_{x,a}\in \cl A^+$ for all 
$x\in X$ and all $a\in A$ and $\sum_{a\in A} g_{x,a} = 1$, $x\in X$), 
write 
$$p_{E,\tau}(a,b|x,y) = \tau(g_{x,a} g_{y,b}), \ \ \ x,y\in X, a,b \in A.$$
Similar arguments to the ones in Sections \ref{ss_tc} and \ref{ss_tccq} show that 
$p_{E,\tau}\in \cl C_{\rm qc}$.

\begin{definition}\label{d_nstr}
An NS correlation $p$ is called 
\begin{itemize}
\item[(i)] \emph{tracial} if 
it is of the form $p_{E,\tau}$, where $E$ is a classical 
stochastic $\cl A$-matrix for some unital C*-algebra $\cl A$ and $\tau : \cl A\to \bb{C}$ is a trace;
\item[(ii)] \emph{quantum tracial} if 
it is tracial and the C*-algebra $\cl A$ in (i) can be chosen to be finite dimensional;
\item[(iii)] \emph{locally tracial} if it 
it is tracial and the C*-algebra $\cl A$ in (i) can be chosen to be abelian.
\end{itemize}
\end{definition}

The next two propositions are analogous to Theorem \ref{th_belo} and \ref{p_ee}, respectively, 
and their proofs are omitted.

\begin{proposition}\label{p_exexex}
Let $p$ be an NS correlation. 
\begin{itemize}
\item[(i)] If $p$ is quantum tracial then $p \in \cl{C}_{\rm q}$;
\item[(ii)] $p$ is locally tracial if and only if 
$p = \sum_{j=1}^k \lambda_j q_j\otimes q_j$,
where $q_j = \{q_j(\cdot|x) : x\in X\}$, is a family of probability distributions, $j = 1,\dots,k$.
In particular, 
if $p$ is locally tracial then $p\in \cl{C}_{\rm loc}$.
\end{itemize}
\end{proposition}

\begin{proposition}\label{p_nsexp}
Let $\cl N : \cl D_{XX}\to \cl D_{AA}$ be an NS correlation. 
\begin{itemize}
\item[(i)] If $\cl N$ is tracial then $\cl N \left(\Upsilon^{\rm cl}(X)\right)\subseteq \Upsilon^{\rm cl}(A)$.

\item[(ii)] If $\cl N$ is quantum tracial then 
$\cl N\left(\Upsilon_{\rm q}^{\rm cl}(X)\right)\subseteq \Upsilon_{\rm q}^{\rm cl}(A)$.

\item[(iii)] If $\cl N$ is locally tracial then 
$\cl N\left(\Upsilon_{\rm loc}^{\rm cl}(X)\right)\subseteq \Upsilon_{\rm loc}^{\rm cl}(A)$.
\end{itemize}
\end{proposition}


\subsection{Reduction for tracial correlations}\label{ss_rftrco}

We next specialise the statements contained in Remark \ref{th_onedq} to tracial correlations.

\begin{theorem}\label{th_rtc}
Let $X$ and $A$ be finite sets, 
$p$ be an NS correlation and $\cl E$ be a CQNS correlation. 
The following hold:
\begin{itemize}
\item[(i)] $p$ is tracial (resp. quantum tracial, locally tracial, fair) if and only if 
$\cl E_p$ is tracial (resp. quantum tracial, locally tracial, fair), if and only if
$\Gamma_p$ is tracial (resp. quantum tracial, locally tracial, fair);

\item[(ii)] 
$\cl E$ is tracial (resp. quantum tracial, locally tracial, fair) if and only if
$\Gamma_{\cl E}$ is tracial (resp. quantum tracial, locally tracial, fair).
\end{itemize}
\noindent
Moreover, 
\begin{itemize}
\item[(iii)] 
the map $\frak{N}$
is a surjection from the class of all tracial (resp.  quantum tracial, locally tracial, fair) CQNS correlations 
onto the class of all tracial (resp.  quantum tracial, locally tracial) NS correlations;
\item[(iv)] 
the map $\frak{C}$ 
is a surjection from the class of all tracial (resp.  quantum tracial, locally tracial, fair) QNS correlations 
onto the class of all tracial (resp.  quantum tracial, locally tracial, fair) CQNS correlations.
\end{itemize}
\end{theorem}

\begin{proof} 
We prove first the statements about tracial correlations.

(i) Suppose that the NS correlation $p$ is tracial, and write $p(a,b|x,y) = \tau(g_{x,a}g_{y,b})$, 
$x,y\in X$, $a,b\in A$, for some trace $\tau$ on a unital C*-algebra $\cl A$ and 
matrix $F = (g_{x,a})_{x,a} \in \left(\cl D_X\otimes \cl D_A\otimes \cl A\right)^+$ with $\sum_{a\in A} g_{x,a} = 1$, $x\in X$. 
The matrix $F' = (\delta_{a,a'}g_{x,a})_{x,a,a'} \in \cl D_X\otimes M_A\otimes \cl A$ is a semi-classical 
stochastic $\cl A$-matrix and, trivially, $\cl E_p = \cl E_{F',\tau}$. 
Similarly, the family $F'' = (\delta_{a,a'}\delta_{x,x'}g_{x,a})_{x,x',a,a'} \in M_X\otimes M_A\otimes \cl A$ is a
stochastic $\cl A$-matrix and $\Gamma_p = \Gamma_{F'',\tau}$.

Conversely, suppose that $\Gamma_p = \Gamma_{E,\tau}$,
where $E = (g_{x,x',a,a'})_{x,x',a,a'}$ is a
stochastic $\cl A$-matrix and $\tau$ is a trace on the unital C*-algebra $\cl A$.
Then $E' := (g_{x,x,a,a'})_{x,a,a'}$ (resp. $E'' := (g_{x,x,a,a})_{x,a}$) is a semi-classical (resp. classical)
stochastic $\cl A$-matrix such that $\cl E_p = \cl E_{E',\tau}$ (resp. $p = p_{E'',\tau}$).

(ii) is similar to (i). 

(iii) follows from the fact that, if $E$ is a stochastic $\cl A$-matrix and $\tau$ is a trace on $\cl A$
such that $\Gamma = \Gamma_{E,\tau}$ then 
$\frak{C}(\Gamma) = \cl E_{E',\tau}$, where $E'$ is given as in the second paragraph of the proof.

(iv) is similar to (iii). 
All remaining statements about quantum tracial and locally tracial correlations are analogous. 

Turning to the case of fair correlations, (ii) follows from the equivalence
$$\cl E(\Sigma_X^{\rm cl})\subseteq\Sigma_A\Longleftrightarrow \Gamma_{\cl E}(\Sigma_X)=\cl E(\Delta_{X,Y}(\Sigma_X)).$$
For (i), observe that $\cl E_p=\Delta_{A,B}\circ \cl E_p$ and hence
$$\cl E_p(\Sigma_X^{\rm cl})\subseteq\Sigma_A\Longleftrightarrow \cl E_p(\Sigma_X^{\rm cl})\subseteq\Sigma_A^{\rm cl}\Longleftrightarrow \cl N_p(\Sigma_X^{\rm cl})\subseteq\Sigma_A^{\rm cl},$$
showing that $p$ is fair if and only if so is $\cl E_p$. 
As $\Gamma_p = \Gamma_{\cl E_p}$, the equivalence with fairness of $\Gamma_p$ follows from (ii).
\end{proof}

We conclude this section with a comparison between the different classes of correlations of synchronous type.
Note first that, if $p$ is a synchronous quantum commuting NS correlations then, by Theorem \ref{th_psstw}, 
$\cl N_p$ is a tracial NS correlation. In fact, the synchronous 
quantum commuting NS correlations arise precisely from 
classical stochastic $\cl A$-matrices $(g_{x,a})_{x,a}$, where each $(g_{x,a})_{a\in A}$ is a PVM, as opposed to 
POVM. 
Theorem \ref{sym} implies that tracial QNS correlations are necessarily fair. 
We summarise these inclusions below:

$$
\begin{matrix}
\mbox{synch. } \cl C_{\rm loc} & \subset & \mbox{loc. tr. NS} & \subset & \mbox{loc. tr. CQNS} & \subset & \mbox{loc. tr. QNS}\\
\cap & & \cap & & \cap & & \cap\\
\mbox{synch. } \cl C_{\rm q} & \subset & \mbox{q. tr. NS} & \subset & \mbox{q. tr. CQNS} & \subset & \mbox{q. tr. QNS}\\
\cap & & \cap & & \cap & & \cap\\
\mbox{synch. } \cl C_{\rm qc} & \subset & \mbox{tracial NS} & \subset & \mbox{tracial CQNS} & \subset & \mbox{tracial QNS}\\
& & \cap & & \cap & & \cap\\
& & \mbox{fair NS} & \subset & \mbox{fair CQNS}  & \subset &  \mbox{fair QNS} 
\end{matrix}
$$

The inclusions in the table are all strict. 
Indeed, for the first column this follows from \cite{dpp}.
It can be shown, using results 
on the completely positive semidefinite cone of matrices 
\cite{lp, blp} that $\Upsilon^{\rm cl}_{\rm loc} \neq \Upsilon^{\rm cl}_{\rm q}$ \cite{akt_new}. 
The properness of the first inclusion in the second column now follows from Remark \ref{r_rec}.
The properness of the second inclusion in the second column was pointed out in Remark \ref{r_contra} (iii), 
and Theorem \ref{th_rtc} implies that the first and the second inclusions in the third and the fourth column are proper.

Let $p = \{p(\cdot | x): x\in X\}$ and $q = \{q(\cdot| x): x\in X\}$ be 
families of distributions so that, for some $x\in X$, we have that 
$\supp p(\cdot |x)\cap\supp q(\cdot |x) = \emptyset$. 
Then $\tilde{p} = 1/2(p\otimes q + q\otimes p)$ is a fair NS correlation. 
However, $\tilde{p}$ is not tracial; indeed, assuming the contrary, 
we have that $\tilde{p} = \sum_{j=1}^m \lambda_j p_j\otimes p_j$ 
as a convex combination, 
where $\{p_j\}_{j=1}^m$ consists of families of probability distributions indexed by $X$.
Since 
$$\tilde{p}(a,a|x,x) = \frac{1}{2}(p(a|x) q(a|x) + q(a|x) p(a|x)) = 0, \ \ \ a\in A,$$  
we have $\sum_{j=1}^m \lambda_j p_j(a|x)^2 = 0$, 
and hence $p_j(a|x)=0$, for all $a\in A$ and all $j$, a contradiction. 
Thus, the last inclusion in the second column is strict, and by Theorem \ref{th_rtc} so are the 
last inclusions in the third and the fourth column.

Using Theorem \ref{th_belo} and Proposition \ref{p_exex34}, one can easily see that
the second and third inclusion on the first row are strict, and hence these inclusions are strict on 
all other rows as well.
Any NS correlation of the form 
$q\otimes q$, where $q = \{q(\cdot|x) : x\in X\}$ is a family of probability distributions with at least one $x$ 
having $|\supp q(\cdot|x)| > 1$, is not synchronous, but is locally tracial; thus, the first inclusion 
in the first, second and third rows are strict.


\section{Correlations as strategies for non-local games}\label{s_tqnlg}

In this section, we discuss how QNS correlations can be 
viewed as perfect strategies for quantum non-local games,
extending the analogous viewpoint on NS correlations to the quantum case. 
Let $X$, $Y$, $A$ and $B$ be finite sets. 
A \emph{non-local game} on $(X,Y,A,B)$ is a cooperative game, played by 
two players against a verifier, determined by a \emph{rule function} (which will often be identified with the game)
$\lambda : X\times Y\times A \times B \to \{0,1\}$. 
The set $X$ (resp. $Y$) is interpreted as a set of questions to, while the set $A$ (resp. $B$) 
as a set of answers of, player Alice (resp. Bob). 
In a single round of the game, the verifier feeds in a pair $(x,y)\in X\times Y$ and 
the players produce a pair $(a,b)\in A\times B$; they win the round if and only if
$\lambda(x,y,a,b) = 1$. 
An NS correlation $p$ on $X\times Y \times A\times B$ is called a 
\emph{perfect strategy} for the game $\lambda$ if 
$$\lambda(x,y,a,b) = 0 \ \Longrightarrow \ p(a,b|x,y) = 0.$$
The terminology is motivated by the fact that if, given a pair $(x,y)$ of questions, 
the players choose their answers according to the probability distribution $p(\cdot,\cdot|x,y)$, 
they will win every round of the game.


\subsection{Quantum graph colourings}\label{ss_gcg}

Let $G$ be a simple graph on a finite set $X$. 
For $x,y\in X$, we write $x\sim y$ if $\{x,y\}$ is an edge of $G$. 
By assumption, $x\sim y$ implies $x\neq y$; we write $x\simeq y$ if $x\sim y$ or $x = y$. 
A classical colouring of $G$ is a map $f : X\to A$, where $A$ is a finite set, 
such that 
$$x\sim y \ \ \ \Longrightarrow \ \ \ f(x)\neq f(y).$$
The chromatic number $\chi(G)$ of $G$ is the minimal cardinality $|A|$ of a set $A$ for which 
a classical colouring $f : X\to A$ of $G$ exists. 

The graph colouring game for $G$ (called henceforth the $G$-colouring game) \cite{cmnsw}
is the non-local game with $Y = X$, $B = A$, and rules
\begin{itemize}
\item[(i)] $x = y \ \ \Longrightarrow \ \ a = b$;
\item[(ii)] $x \sim y \ \  \Longrightarrow \ \ a \neq b$.
\end{itemize}
Thus, an NS correlation $p = \left\{(p(a,b|x,y))_{a,b\in A} : x,y\in X\right\}$ 
is a perfect strategy of the $G$-colouring game if 
\begin{itemize}
\item[(S)] $p$ is synchronous;
\item[(P)] $x \sim y \ \Rightarrow \ p(a,a|x,y) = 0$ for all $a$.
\end{itemize}

It is easy to see that if $p$ is a perfect strategy of the $G$-colouring game from the class $\cl C_{\rm loc}$
then $G$ possesses a classical colouring from the set $A$. 
Thus, the perfect strategies for the $G$-colouring game from $\cl C_{\rm x}$, where 
${\rm x} \in \{{\rm loc}, {\rm q}, {\rm qc}\}$
can be thought of as \emph{classical ${\rm x}$-colourings} of $G$.
The ${\rm x}$-chromatic number of $G$ is the parameter 
$$\chi\mbox{}_{\rm x}(G) = \min\left\{|A| : G \mbox{ has a classical } {\rm x}\mbox{-colouring by } A\right\};$$
in particular, $\chi_{\rm loc}(G = \chi(G)$ (see \cite{cmnsw, lmr, pt_QJM} and the references therein).

We call $p$ a \emph{$G$-proper} correlation if condition (P) is satisfied. 
For a finite set $A$, we let 
$\Omega_A$ be the non-normalised maximally entangled matrix in $M_{AA}$, namely,
$$\Omega_A = \sum_{a,b\in A} e_ae_b^*\otimes e_ae_b^*.$$

\begin{remark}\label{r_Omega}
Let $G$ be a graph with vertex set $X$.
An NS correlation $p$ over $(X,X,A,A)$ is $G$-proper if and only 
$$x \sim y \ \Longrightarrow \  \left\langle \cl E_p\left(e_xe_x^*\otimes e_y e_y^*\right),\Omega_A\right\rangle = 0.$$
\end{remark}

\begin{proof}
The claim is immediate from the fact that
\begin{eqnarray*}
& & \left\langle \cl E_p\left(e_xe_x^*\otimes e_y e_y^*\right),\Omega_A\right\rangle\\
& = & 
\sum_{a,b\in A} \sum_{a',b'\in A} p\left(a,b|x,y\right) 
\left\langle e_{a}e_{a}^*\otimes e_{b} e_{b}^*, e_{a'} e_{b'}^*\otimes e_{a'} e_{b'}^*\right\rangle\\
& = &
\sum_{a,b\in A} \sum_{a',b'\in A} p\left(a,b|x,y\right) 
\left\langle e_{a}e_{a}^*, e_{a'} e_{b'}^*\right\rangle\left\langle e_{b} e_{b}^*,e_{a'} e_{b'}^*\right\rangle\\
& = & 
\sum_{a\in A} p\left(a,a|x,y\right).
\end{eqnarray*}
\end{proof}

Remark \ref{r_Omega} allows to generalise the classical ${\rm x}$-colourings of a graph $G$ 
to the quantum setting as follows.

\begin{definition}\label{d_CQNScol}
Let $G$ be a graph with vertex set $X$. 
A CQNS correlation $\cl E : \cl D_{XX}\to M_{AA}$ is called \emph{$G$-proper} if 
$$x \sim y \ \Longrightarrow \  \left\langle \cl E(e_xe_x^*\otimes e_ye_y^*),\Omega_A\right\rangle = 0.$$
A $G$-proper CQNS correlation $\cl E$ is called
\begin{itemize}
\item[(i)] 
a \emph{quantum ${\rm loc}$-colouring of $G$ by $A$} if $\cl E$ is locally tracial;
\item[(ii)] 
a \emph{quantum ${\rm q}$-colouring of $G$ by $A$} if 
$\cl E$ is quantum tracial;
\item[(iii)] 
a \emph{quantum ${\rm qc}$-colouring of $G$ by $A$} if 
$\cl E$ is tracial.
\end{itemize}
\end{definition}

\noindent For ${\rm x}\in \{{\rm loc}, {\rm q}, {\rm qc}\}$, let 
$$\xi\mbox{}_{\rm x}(G) = \min\left\{|A| : \ \exists \mbox{ a quantum ${\rm x}$-colouring of $G$ by $A$}\right\}$$
be the \emph{quantum ${\rm x}$-chromatic number} of $G$.

Recall \cite{ss} that an \emph{orthogonal representation} of a graph $G$ with vertex set $X$ is a 
family $(\xi_x)_{x\in X}$ of unit vectors in $\bb{C}^k$ such that
$$x\sim y \ \ \ \Longrightarrow \ \ \ \left\langle \xi_x,\xi_y\right\rangle = 0.$$
The \emph{orthogonal rank} $\xi(G)$ of $G$ is given by 
$$\xi(G) = \min\left\{k : \ \exists \mbox{ an orthogonal representation of $G$ in  } \bb{C}^k\right\}.$$

\begin{proposition}\label{p_locor}
Let $G$ be a graph with vertex set $X$. The following are equivalent: 
\begin{itemize}
\item[(i)] the graph $G$ has an orthogonal representation in $\bb{C}^k$;

\item[(ii)] there exists a quantum ${\rm loc}$-colouring of $G$ by a set $A$ with $|A| = k$. 
\end{itemize}
\end{proposition}

\begin{proof}
(i)$\Rightarrow$(ii) 
Suppose that $(\xi_x)_{x\in X}\subseteq \bb{C}^k$ is an orthogonal representation of $G$. 
Let 
$\cl E_0 : \cl D_{X}\to M_{A}$ be the quantum channel given by 
$$\cl E_0(e_x e_x^*) = \xi_x\xi_x^*, \ \ \ x\in X,$$
and set $\cl E = \cl E_0\otimes \cl E_0^{\sharp}$;
by Proposition \ref{p_exex34}, $\cl E$ is locally tracial.
If $x\sim y$ then 
\begin{eqnarray*}
\left\langle \cl E(e_x e_x^* \otimes e_y e_y^*),\Omega_A\right\rangle
& = & 
\sum_{a,b\in A} \left\langle \xi_x \xi_x^* \otimes \left(\xi_y \xi_y^*\right)^{\rm t}, e_a e_b^*\otimes e_a e_b^*\right\rangle\\
& = & 
\sum_{a,b\in A} \Tr\left(\left(\xi_x \xi_x^*\right) \left(e_a e_b^*\right)^{\rm t}\right) 
\Tr\left( \left( \xi_y \xi_y^* \right)^{\rm t} \left(e_a e_b^*\right)^{\rm t}\right)\\
& = & 
\sum_{a,b\in A} \Tr\left(\left(\xi_x \xi_x^*\right) \left(e_b e_a^*\right)\right) 
\Tr\left( \left( \xi_y \xi_y^* \right) \left(e_a e_b^*\right)\right)\\
& = & 
\sum_{a,b\in A} \left\langle \xi_x,e_a \right\rangle \left\langle e_b,\xi_x\right\rangle \left\langle \xi_y,e_b \right\rangle
\left\langle e_a,\xi_y\right\rangle\\
& = & 
\left(\sum_{a\in A} \left\langle \xi_x,e_a \right\rangle \left\langle e_a,\xi_y \right\rangle \right)
\left(\sum_{b\in A} \left\langle \xi_y,e_b\right\rangle \left\langle e_b,\xi_x \right\rangle \right)\\
& = &  
\left|\left\langle \xi_x,\xi_y\right\rangle\right|^2 = 0;
\end{eqnarray*}
thus, $\cl E$ is a quantum ${\rm loc}$-colouring of $G$. 

(ii)$\Rightarrow$(i) 
Suppose that $\cl E : \cl D_{XX}\to M_{AA}$ is a quantum ${\rm loc}$-colouring of $G$,
and write $\cl E = \sum_{j=1}^k \lambda_j \cl E_j\otimes \cl E_j^{\sharp}$ as a convex combination
with positive coefficients, 
where $\cl E_j : \cl D_X\to M_A$ is a quantum channel, $j = 1,\dots,k$.
Suppose that $x\sim y$. 
Then 
$$\sum_{j=1}^k \lambda_j \left\langle  \left(\cl E_j\otimes \cl E_j^{\sharp}\right)\left(e_x e_x^* \otimes e_y e_y^*\right),
\Omega_A\right\rangle = 0$$
and hence, by the non-negativity of each of the terms of the sum,
\begin{equation}\label{eq_E1}
\left\langle  \cl E_1(e_x e_x^*) \otimes \cl E_1^{\sharp}(e_y e_y^*),\Omega_A\right\rangle = 0.
\end{equation}
Let $\xi_x$ be a unit eigenvector of $\cl E_1(e_x e_x^*)$, corresponding to a positive eigenvalue, $x\in X$. 
Condition (\ref{eq_E1}) implies that 
$\left\langle  \xi_x \xi_x^* \otimes \left(\xi_y \xi_y^*\right)^{\rm t},\Omega_A\right\rangle = 0$,
which in turn means, by the arguments in the previous paragraph, that $\langle\xi_x,\xi_y\rangle = 0$.
\end{proof}

By Proposition \ref{p_locor}, $\xi_{\rm loc}(G) = \xi(G)$. Thus, 
the parameters $\xi_{\rm q}$ and $\xi_{\rm qc}$ 
can be viewed as quantum versions of the orthogonal rank.

\begin{proposition}\label{r_chxxix}
Let $G$ be a graph. Then 
\begin{itemize}
\item[(i)] $\xi_{\rm qc}(G) \leq \xi_{\rm q}(G) \leq \xi_{\rm loc}(G)$, and 
\item[(ii)] $\xi_{\rm x}(G)\leq \chi_{\rm x}(G)$ for ${\rm x}\in \{{\rm loc}, {\rm q},{\rm qc}\}$.
\end{itemize}
\end{proposition}

\begin{proof}
(i) The inequalities follow from the fact that $\cl{CQ}_{\rm loc}\subseteq \cl{CQ}_{\rm q} \subseteq \cl{CQ}_{\rm qc}$.

(ii) 
Let $p$ be a synchronous NS correlation that is an ${\rm x}$-colouring of $G$ by a set $A$. 
By Theorem \ref{th_rtc}, $\cl E_p\in \cl{CQ}_{\rm x}$. By Remark \ref{r_Omega}, $\cl E_p$ is $G$-proper. 
Thus, $\xi_{\rm x}(G)\leq \chi_{\rm x}(G)$. 
\end{proof}

\noindent {\bf Remarks. (i) } 
There exist graphs $G$ for which $\xi(G) < \chi(G)$ (see e.g. \cite{ss}). 
By Proposition \ref{p_locor}, for such $G$ we have a 
strict inequality in Proposition \ref{r_chxxix} (ii) in the case 
${\rm x} = {\rm loc}$.
In \cite{mrob2}, an example of a graph $G$ on 13 vertices was exhibited with the property that 
$\xi(G) < \chi_{\rm q}(G)$. 
By Proposition \ref{r_chxxix} (i), for this graph $G$, we have a 
strict inequality in Proposition \ref{r_chxxix} (ii) in the case 
${\rm x} = {\rm q}$.
We do not know if a strict inequality can occur in the case ${\rm x} = {\rm qc}$.

\smallskip

{\bf (ii) } 
It was shown in \cite{mrob2} that there exists a graph $G$ such that 
$\chi_{\rm q}(G) < \xi(G)$. By Proposition \ref{r_chxxix} (ii), this implies 
$\xi_{\rm q}(G) < \xi(G)$. 
We do not whether $\xi_{\rm qc}(G)$ can be strictly smaller than $\xi_{\rm q}(G)$.

\smallskip

We next exhibit a lower bound on $\xi_{\rm qc}(G)$ in terms of the Lov\'{a}sz number $\theta(G)$ of $G$. 
We refer the reader to \cite{lo} for the definition and properties of the latter parameter.
We denote by $K_d$ the complete graph on $d$ vertices.
We will need some notation, which will also be essential in Subsection \ref{ss_qhncg}. 
If $\kappa\subseteq X\times X$, let 
$$\cl S_{\kappa} = {\rm span}\left\{e_x e_y^* : (x,y)\in \kappa\right\};$$
thus, $\cl S_{\kappa}$ is a linear subspace of $M_X$ which is a bimodule over the diagonal algebra $\cl D_X$. 
We write
$$E(G) = \{(x,y) \in X\times X : x\simeq y\} \mbox{ and }
E_0(G) = \{(x,y) \in X\times X : x\sim y\},$$
and let $\cl S_G := \cl S_{E(G)}$ be the \emph{graph operator system} of $G$ \cite{dsw}, 
and $\cl S_G^0 := \cl S_{E_0(G)}$ be the \emph{graph operator anti-system} of $G$ \cite{stahlke}
(here we use the terminology of \cite{btw}).

\begin{proposition}\label{p_lovasz}
Let $G$ be a graph with vertex set $X$. Then $\xi_{\rm qc} (G) \geq \sqrt{\frac{|X|}{\theta(G)}}$. 
Moreover, $\xi_{\rm q}(K_{d^2}) = \xi_{\rm qc}(K_{d^2}) = d$. 
\end{proposition}

\begin{proof}
Let $\cl A$ be a C*-algebra, $\tau : \cl A\to \bb{C}$ be a trace, $(E_{x,a,a'})\in \cl D_X\otimes M_A\otimes \cl A$ be a
semi-classical stochastic $\cl A$-matrix, and $\Theta = (\omega_{x,y})_{x,y\in X}$ be a quantum ${\rm qc}$-colouring of $G$, such that 
$$\omega_{x,y} = \left(\tau\left(E_{x,a,a'}E_{y,b',b}\right)\right)_{a,a',b,b'}, \ \ \ x,y\in X.$$
Assume, without loss of generality, that $\cl A\subseteq \cl B(H)$ as a unital C*-subalgebra and that $\xi\in H$ is a unit vector
with $\tau(u) = \langle u \xi,\xi\rangle$, $u \in \cl A$. 
Set $\xi_{x,a,a'} = E_{x,a,a'}\xi$, $x\in X$, $a,a'\in A$; then 
$$\omega_{x,y} = \left(\langle \xi_{x,a,a'}, \xi_{y,b,b'}\rangle\right)_{a,a',b,b'}, \ \ \ x,y\in X.$$
We note that 
\begin{equation}\label{eq_addtox}
\sum_{a\in A} \xi_{x,a,a} = \xi, \ \ \ x\in X.
\end{equation}
In addition, if $x\sim y$ then 
\begin{equation}\label{eq_Gproa}
\sum_{a,b\in A}\left\langle \xi_{x,a,b}, \xi_{y,a,b} \right\rangle 
= 
\sum_{a,b\in A}\left\langle \omega_{x,y}, e_a e_b^* \otimes e_a e_b^*\right\rangle
=
\left\langle \omega_{x,y},\Omega_A \right\rangle
= 0.
\end{equation}

Let 
$$Q_{a,a',b,b'} = \left(\langle \xi_{x,a,a'}, \xi_{y,b,b'}\rangle\right)_{x,y\in X}, \ \ \ a,a',b,b'\in A.$$
Note that, up to an application of the canonical shuffle,
$$\left(Q_{a,a',b,b'}\right)_{a,a',b,b'} = \left(\omega_{x,y}\right)_{x,y} 
= \left(\cl E_{\Theta}(e_x e_x^*\otimes e_y e_y^*)\right)_{x,y},$$
and hence, after another application of the canonical shuffle, Choi's Theorem implies that the linear map 
$\Psi : M_{AA}\to M_{X}$, given by 
$$\Psi\left(e_a e_b^* \otimes e_{a'} e_{b'}^*\right) = Q_{a,a',b,b'}, \ \ \ a,a',b,b'\in A,$$
is completely positive.
We have
$$\Psi(I_{AA}) = \sum_{a,b\in A} \Psi\left(e_a e_a^* \otimes e_b e_b^*\right) = \sum_{a,b\in A} Q_{a,b,a,b}.$$
By (\ref{eq_Gproa}), 
$$x\sim y \ \Longrightarrow \ \left\langle \Psi(I_{AA})e_x,e_y\right\rangle = 0.$$

By Theorem \ref{p_coor}, there exist operators $V_{a,x}$ such that $(V_{a,x})_{a,x}$ is an isometry and 
$E_{x,a,a'} = V_{a,x}^* V_{a',x}$, $x\in X$, $a,a'\in A$. 
Thus, if $x\in X$ then 
\begin{eqnarray*}
\left\langle \Psi(I_{AA})e_x,e_x\right\rangle 
& = & 
\sum_{a,b\in A}\left\langle \xi_{x,a,b}, \xi_{x,a,b} \right\rangle
= 
\sum_{a,b\in A} \left\|E_{x,a,b}\xi\right\|^2
\\
& = & 
\sum_{a,b\in A} \left\|V_{a,x}^* V_{b,x}\xi\right\|^2
\leq 
\sum_{a\in A} \sum_{b\in A} \left\|V_{b,x}\xi\right\|^2\\
& = & 
|A| \sum_{b\in A} \left\langle V_{b,x}\xi,V_{b,x}\xi\right\rangle
=  
|A| \left\langle \sum_{b\in A}  V_{b,x}^* V_{b,x}\xi,\xi\right\rangle = |A|.
\end{eqnarray*}
Write 
$\Psi(I_{AA}) = D + T$, where $D$ is diagonal and $T \perp \cl S_{G}$.  
We have shown that $D\leq |A|I_X$; thus $|A| I_X + T\in M_X^+$.
It follows that 
\begin{eqnarray}\label{eq_IXT}
\|\Psi(I_{AA})\|
& \leq & 
\left\||A| I_X + T\right\|\nonumber\\
& \leq & 
\max\left\{\left\||A| I_X + S\right\| : S\in \cl S_G^{\perp}, |A| I_X + S \in M_X^+\right\} 
=
|A|\theta(G).
\end{eqnarray}

Let $J_X$ be the matrix in $M_X$ all of whose entries are equal to one. 
By (\ref{eq_addtox}),
\begin{eqnarray}\label{eq_OAJ}
\Psi(\Omega_{A}) 
& = & 
\sum_{a,b\in A} \Psi\left(e_a e_b^*\otimes e_a e_b^*\right) 
= \sum_{a,b\in A} \left(\left\langle \xi_{x,a,a}, \xi_{y,b,b} \right\rangle\right)_{x,y}\nonumber\\
& = &  
\left(\left\langle \sum_{a\in A} \xi_{x,a,a}, \sum_{b\in A} \xi_{y,b,b} \right\rangle\right)_{x,y} = J_X.
\end{eqnarray}
By (\ref{eq_IXT}) and (\ref{eq_OAJ}), 
$$|X| = \left\|J_X \right\| \leq \left\|\Omega_A\right\| \left\|\Psi\right\| = 
|A| \left\|\Psi(I_{AA}) \right\| \leq |A|^2\theta(G).$$ 
Taking the minimum over all $|A|$ completes the proof of the inequality.

Realise 
$A = \mathbb Z_d = \{0,1,\ldots, d-1\}$ and let $X = A\times A$. 
Let $\zeta$ be a primitive $|A|$-th root of unity.
For $x = (a',b')$ and $y = (a'',b'')\in X$, let 
$$\xi_{x,y} = \frac{1}{\sqrt{d}}\zeta^{b''(a''-a')}\sum_{l=0}^{d-1}\zeta^{(b''-b')l}e_l\otimes e_{l-a'+a''}$$
and write $\sigma_{x,y} = \xi_{x,y}\xi_{x,y}^*$. 
We have 
\begin{eqnarray*}
\sigma_{x,y}
& = & 
\frac{1}{d} \sum_{l,n=0}^{d-1}\zeta^{(b''-b')(l-n)}e_le_n^*\otimes e_{l-a'+a''}e_{n-a'+a''}^*\\
& = &
\frac{1}{d} \sum_{l,n=0}^{d-1}\zeta^{(b''-b')(l-n)}e_{l+a'}e_{n+a'}^*\otimes e_{l+a''}e_{n+a''}^*.
\end{eqnarray*}
Note that $\Theta = (\sigma_{x,y})_{x,y}$ is a CQNS correlation; indeed,
$$\Tr{}_A\sigma_{x,y} = \frac{1}{d} \sum_{l=0}^{d-1} e_{l+a''}e_{l+a''}^*
= \frac{1}{d} I_A = \frac{1}{d} \sum_{l=0}^{d-1} e_{l+a'}e_{l+a'}^* = \Tr{}_B\sigma_{x,y}$$
for all $x,y\in X$. 
Since 
$$\sum_{a,b\in A}\langle \sigma_{x,y},e_ae_b^*\otimes e_ae_b^*\rangle
= \frac{1}{d} \sum_{a,b\in A}\delta_{a',a''}\zeta^{(b''-b')(a-b)} = d \delta_{a',a''}\delta_{b',b''},$$
we have that $\Theta$  is $K_{d^2}$-proper.

We claim that $\Theta$ is tracial. 
To see this, let $E_{x,z,z'} = \zeta^{(z'-z)b'}e_{z-a'}e_{z'-a'}^*\in \cl L(\bb{C}^A)$, $x = (a',b')\in X$, $z,z'\in A$, 
and set $E_x = (E_{x,z,z'})_{z,z'\in A}$, $x\in X$. 
Fix $x\in X$; then $\sum_{z\in A} E_{x,z,z} = I_A$. Furthermore, 
if $\xi = (\xi_z)_{z\in A}$, $\xi_z\in \bb{C}^A$, then
$$\left\langle E_x\xi,\xi\right\rangle
= \sum_{z,z'\in A}\zeta^{z'b'}\langle\xi_{z'},e_{z'-a'}\rangle\langle e_{z-a'},\xi_z\rangle\zeta^{-zb'}
= \left|\sum_{z\in A}\zeta^{zb'}\langle\xi_{z},e_{z-a'}\rangle\right|^2\geq 0.$$
Thus, $E = (E_x)_{x\in X} \in \cl D_{X}\otimes M_A\otimes \cl L(\bb{C}^A)$ is a semi-classical stochastic matrix. 
Moreover,
for $x = (a',b'), y = (a'',b'')\in X$ and $z,z',w,w'\in A$ we have
\begin{eqnarray*}
& & 
\Tr\left(E_{x,z,z'}E_{y,w',w}\right)\\
& = & 
\Tr\left(\zeta^{(z'-z)b'}\zeta^{(w-w')b''} (e_{z-a'}e_{z'-a'}^*)(e_{w'-a''}e_{w-a''}^*)\right)\\
& = &
\delta_{z'-a', w'-a''}\delta_{z-a',w-a''}\zeta^{(z'-z)(b'-b'')}\\
& = &
\sum_{l,n=0}^{d-1}\zeta^{(b''-b')(l-n)}\left\langle e_{z'},e_{n+a'}\right\rangle
\left\langle e_{w'},e_{n+a''}\right\rangle  \left\langle e_{l+a'}, e_{z}\right\rangle
\left\langle e_{l+a''},e_w\right\rangle\\
& = &
\sum_{l,n=0}^{d-1}\zeta^{(b''-b')(l-n)} \left\langle (e_{l+a'}e_{n+a'}^*)e_{z'}\otimes (e_{l+a''}e_{n+a''}^*)e_{w'}, e_z\otimes e_w\right\rangle\\
& = & 
d \left\langle\sigma_{x,y}(e_{z'}\otimes e_{w'}),e_z\otimes e_w\right\rangle.
\end{eqnarray*}
Therefore $\Theta$ is quantum tracial.  It follows that $\xi_{\rm q}(K_{d^2}) \leq d$; 
On the other hand, $\theta(K_{d^2}) = 1$ and hence $\xi_{\rm qc}(K_{d^2}) \geq d$. 
Proposition \ref{r_chxxix} now implies that $\xi_{\rm qc}(K_{d^2}) = \xi_{\rm q}(K_{d^2}) = d$. 
\end{proof}


\subsection{Graph homomorphisms}\label{ss_qhncg}

In this subsection, we consider a quantum version 
of the graph homomorphism game first studied in \cite{mrob}.
Let $G$ and $H$ be graphs with vertex sets $X$ and $A$, respectively. 
Recall that the homomorphism game $G\to H$ has $Y = X$, $B = A$, and 
$\lambda(x,y,a,b) = 0$ if and only if, either $x = y$ and $a\neq b$, or $x\sim y$ and $a\not\sim b$. 
A synchronous NS correlation
$p = \left\{\left(p(a,b|x,y)\right)_{a,b\in A} : x,y\in X\right\}$ is thus called a perfect ${\rm x}$-strategy 
for the game $G\to H$ if $p\in \cl C_{\rm x}$ and 
$$x \sim y, \ a\not\sim b \ \Longrightarrow \ p(a,b|x,y) = 0.$$

For a subset $\kappa\subseteq X\times X$, let 
$P_{\kappa} : M_X\to M_X$ be the map given by
$$P_{\kappa}(T) = \sum_{(x,y)\in \kappa} (e_x e_x^*) T (e_y e_y^*), \ \ \ T\in M_X.$$
Thus, $P_{\kappa}$ is the Schur projection onto $\cl S_{\kappa}$; it 
can be canonically identified with the (positive) element $\sum_{(x,y)\in \kappa} (e_x e_x^*) \otimes (e_y e_y^*)$
of $\cl D_{XX}$. We set $(P_{\kappa})_{\perp} = P_{\kappa^c}$. 
For a graph $G$, we write for brevity $P_G = P_{E_0(G)}$.

\begin{proposition}\label{p_clgh}
Let $G$ (resp. $H$) be a graph with vertex set $X$ (resp. $A$), and 
$p = \left\{\left(p(a,b|x,y)\right)_{a,b\in A} : x,y\in X\right\}$ be a synchronous NS correlation.
The following are equivalent:
\begin{itemize}
\item[(i)] $p$ is a perfect strategy for the homomorphism game $G\to H$;

\item[(ii)] $\left\langle\cl N_p(P_G), (P_{H})_{\perp}\right\rangle = 0$.
\end{itemize}
\end{proposition}

\begin{proof}
(i)$\Rightarrow$(ii)
We have
\begin{eqnarray*}
& & \left\langle\cl N_p\left(P_G\right), (P_{H})_{\perp}\right\rangle\\
& = & 
\left\langle \cl N_p\left(\sum_{x\sim y} e_x e_x^* \otimes e_y e_y^*\right), 
\sum_{a'\not\sim b'} e_{a'} e_{a'}^* \otimes e_{b'} e_{b'}^*\right\rangle\\
& = & 
\sum_{x\sim y} \sum_{a,b\in A} \sum_{a'\not\sim b'} p(a,b|x,y) \left\langle e_a e_a^* \otimes e_b e_b^*, 
e_{a'} e_{a'}^* \otimes e_{b'} e_{b'}^* \right\rangle\\
& = & 
\sum_{x\sim y} \sum_{a\sim b} \sum_{a'\not\sim b'} p(a,b|x,y) 
\left\langle e_a e_a^* \otimes e_b e_b^*, e_{a'} e_{a'}^* \otimes e_{b'} e_{b'}^*\right\rangle
= 0.
\end{eqnarray*}

(ii)$\Rightarrow$(i)
If $x\sim y$ and $a\not\sim b$ then
$e_x e_x^*\otimes e_y e_y^*\leq P_G$ and $e_a e_a^*\otimes e_b e_b^*\leq (P_H)_{\perp}$.
By the monotonicity of the pairing, 
$$p\left(a,b|x,y\right) 
=  
\left\langle\cl N_p\left(e_x e_x^*\otimes e_y e_y^*\right), e_a e_a^*\otimes e_b e_b^*\right\rangle 
\leq \left\langle\cl N_p(P_G), (P_{H})_{\perp}\right\rangle = 0.$$
\end{proof}

General operator systems in $M_X$ were considered in \cite{dsw} as a quantum versions of graphs 
(noting that $\cl S_G$ is an operator system), while operator anti-systems
(that is, selfadjoint subspaces of $M_X$ each of whose elements has trace zero \cite{btw}) 
were proposed as such a quantum version in \cite{stahlke} (noting that $\cl S_G^0$ is an operator anti-system). 
Note that one can pass from any of the two notions to the other by taking orthogonal complements. 
Due to the specific definition of QNS correlations in \cite{dw}, employed also here,
it will be convenient to use a slightly different (but equivalent) perspective on non-commutative graphs, 
which we now describe. 
Let $Z$ be a finite set, $H = \bb{C}^Z$,
$H^{\rm d}$ be its dual space and 
${\rm d} : H\to H^{\rm d}$ be the map given by ${\rm d}(\xi) (\eta) = \langle \eta,\xi\rangle$; 
we write $\xi^{\dd} = \dd(\xi)$. 
Note that, if $T\in \cl L(H)$ then 
\begin{equation}\label{eq_Mxi}
T^{\rm d}\xi^{\rm d} = (T^*\xi)^{\rm d}, \ \ \ T\in \cl L(H).
\end{equation}
Let $\theta : H\otimes H\to \cl L(H^{\dd},H)$ be the linear map given by 
$$\theta(\xi\otimes\eta)(\zeta^{\dd}) = \langle \xi,\zeta\rangle\eta, \ \ \ \zeta\in H.$$
By (\ref{eq_Mxi}),
\begin{equation}\label{eq_MNtheta}
\theta((S\otimes T)\zeta) = T\theta(\zeta)S^{\rm d}, \ \ \ \zeta\in H\otimes H, \ S,T\in \cl L(H).
\end{equation}
We denote by $\mm : H\otimes H\to \bb{C}$ the map given by 
$$\mm(\zeta) = \left\langle \zeta, \sum_{z\in Z} e_z\otimes e_z\right\rangle, \ \ \ \zeta\in H\otimes H.$$
Let also 
$\frak{f} : H\otimes H \to H\otimes H$ be the flip
operator given by $\frak{f}(\xi\otimes\eta) = \eta\otimes\xi$. 
Note that, if $\xi,\eta, \zeta_1,\zeta_2 \in  H$ then 
\begin{eqnarray*}
\langle \theta(\xi\otimes\eta)^*(\zeta_1),\zeta_2^{\rm d}\rangle
& = & 
\langle \zeta_1,\theta(\xi\otimes\eta)\zeta_2^{\rm d}\rangle
= 
\langle \zeta_1,\langle \xi,\zeta_2\rangle \eta\rangle
= 
\langle \zeta_2,\xi \rangle \langle \zeta_1,\eta\rangle\\
& = & 
\langle \xi^{\rm d},\zeta_2^{\rm d} \rangle \langle \zeta_1,\eta\rangle
= 
\langle \langle \zeta_1,\eta\rangle \xi^{\rm d},\zeta_2^{\rm d} \rangle 
\end{eqnarray*}
and hence
$${\rm d}^{-1}(\theta(\xi\otimes\eta)^* ({\rm d}^{-1}(\zeta_1^{\rm d})) 
= 
{\rm d}^{-1}( \langle \zeta_1,\eta\rangle \xi^{\rm d})
= 
\langle \eta,\zeta_1\rangle \xi 
= 
\theta(\eta\otimes\xi)(\zeta_1^{\rm d});$$
thus, 
\begin{equation}\label{eq_cotd}
{\rm d}^{-1} \circ \theta(\zeta)^* \circ  {\rm d}^{-1} = (\theta\circ \frak{f})(\zeta), \ \ \ \zeta\in H\otimes H.
\end{equation}
In addition, 
$$\sum_{z\in Z} \langle \theta(\xi\otimes\eta) ({\rm d}(e_z)), e_z\rangle 
= \sum_{z\in Z} \langle \xi,e_z\rangle \langle\eta,e_z\rangle
= {\rm m}(\xi\otimes\eta),$$
and hence 
\begin{equation}\label{eq_cotd2}
{\rm m}(\zeta) = \sum_{z\in Z} \langle (\theta(\zeta) \circ {\rm d}) (e_z), e_z\rangle, \ \ \ \zeta\in H\otimes H.
\end{equation}

\begin{definition}\label{d_ss}
A linear subspace $\cl U\subseteq H\otimes H$ is called \emph{skew} if $\mm(\cl U) = \{0\}$ and 
\emph{symmetric} if $\frak{f}(\cl U) = \cl U$. 
\end{definition}

Suppose that $\cl U$ is a symmetric skew subspace of $H\otimes H$. 
Let 
$\cl S_{\cl U} = \theta(\cl U)$;
by (\ref{eq_cotd}) and (\ref{eq_cotd2}),
the subspace $\cl S_{\cl U}$ of $\cl L(H^{\dd},H)$ satisfies
\smallskip
\begin{itemize}
\item $T\in \cl S_{\cl U}$ $\Longrightarrow$ $\dd^{-1}\circ T^* \circ \dd^{-1} \in \cl S_{\cl U}$, and 
\item $T\in \cl S_{\cl U}$ $\Longrightarrow$  $\sum_{z\in Z} \langle (T\circ \dd) (e_z),e_z\rangle = 0$.
\end{itemize}
\smallskip
We call a subspace of $\cl L(H^{\dd},H)$ satisfying these properties a \emph{twisted operator anti-system}.
Conversely, given a twisted operator anti-system $\cl S\subseteq \cl L(H^{\dd},H)$, (\ref{eq_cotd}) and (\ref{eq_cotd2}) imply that 
the subspace $\cl U_{\cl S} = \theta^{-1}(\cl S)$ of $H\otimes H$ is symmetric and skew. 
Given a graph $G$, let 
$$\cl U_G = {\rm span}\{e_x\otimes e_y : x\sim y\};$$
it is clear that $\cl U_G$ is a symmetric skew subspace of $\bb{C}^X\otimes \bb{C}^X$. 
We thus consider symmetric skew subspaces of $\bb{C}^X\otimes \bb{C}^X$ as a 
non-commutative version of graphs. 

We write $P_{\cl U}$ for the orthogonal projection from $\bb{C}^X\otimes \bb{C}^X$ onto $\cl U$. 
Let $\cl U_{\perp}\subset \left(\bb{C}^X\otimes \bb{C}^X\right)^{\dd}$ be the annihilator of $\cl U$ and write
$P_{\cl U_\perp}\in \cl L\left((\bb{C}^X\otimes \bb{C}^X)^{\dd}\right) $ for the orthogonal projection onto $\cl U_{\perp}$. 
Observe that $\zeta^{\dd}\in \cl U_{\perp}$ if and only if $\zeta$ 
belongs to the orthogonal complement $\cl U^{\perp}$ of $\cl U$ in $\bb{C}^X\otimes \bb{C}^X$.
Thus, for $\zeta \in H\otimes H$ we have 
\begin{eqnarray*}
P_{\cl U_\perp}\zeta^{\rm d} = \zeta^{\rm d} 
& \Leftrightarrow & 
P_{\cl U}^{\perp}(\zeta) = \zeta
\Leftrightarrow 
\langle P_{\cl U}^{\perp}(\zeta),\zeta^{\rm d} \rangle = 1\\
& \Leftrightarrow & 
\langle \zeta, (P_{\cl U}^{\perp})^{\rm d}(\zeta^{\rm d}) \rangle = 1
\Leftrightarrow
(P_{\cl U}^{\perp})^{\rm d}(\zeta^{\rm d}) = \zeta^{\rm d},
\end{eqnarray*}
and hence
\begin{equation}\label{eq_pperpan}
P_{\cl U_\perp} = (P_{\cl U}^\perp)^{\dd}. 
\end{equation}

Let $A$ be a finite set and $\omega\in M_A$. 
Writing $f_\omega$ for the functional on $M_A$ given by
$f_\omega(\rho) = \Tr(\rho\omega^{\rm t})$,
we have that the map $\omega\to f_{\omega}$ is a complete order isomorphism from $M_A$ onto 
$M_A^{\rm d}$
(see e.g. \cite[Theorem 6.2]{ptt}). 
On the other hand, the map $\omega^{\dd}\mapsto \omega^{\rm t}$ is 
 a *-isomorphism  from 
$\cl L\left((\bb{C}^A)^{\dd}\right)$ onto $M_A$. 
The composition of these maps, $\omega^{\dd}\mapsto f_{w^t}$, is thus 
a complete order isomorphism from 
$\cl L\left((\bb{C}^A)^{\dd}\right)$ onto $M_A^{\rm d}$. In the sequel, we identify these two spaces; 
note that, via this identification,
\begin{equation}\label{eq_dtom}
\langle \rho,\omega^{\dd}\rangle = \langle\rho,\omega^{\rm t}\rangle = \Tr(\rho\omega), \ \ \ \rho,\omega\in M_A. 
\end{equation}

\begin{definition}\label{d_hss}
Let $X$ and $A$ be finite sets and $\cl U\subseteq \bb{C}^X\otimes \bb{C}^X$,  
$\cl V\subseteq \bb{C}^A\otimes \bb{C}^A$
be symmetric skew subspaces. A QNS correlation $\Gamma : M_{XX}\to M_{AA}$ is called
\begin{itemize}
\item[(i)] a \emph{quantum commuting homomorphism} from $\cl U$ to $\cl V$ 
(denoted $\cl U\stackrel{\rm qc}{\to}\cl V$)
if $\Gamma$ is tracial and 
\begin{equation}\label{eq_projde}
\left\langle \Gamma\left(P_{\cl U}\right),P_{\cl V_{\perp}}\right\rangle = 0.
\end{equation}

\item[(ii)]
a \emph{quantum homomorphism} from $\cl U$ to $\cl V$ 
(denoted $\cl U\stackrel{\rm q}{\to}\cl V$)
if $\Gamma$ is quantum tracial and (\ref{eq_projde}) holds;

\item[(iii)]a \emph{local homomorphism} from $\cl U$ to $\cl V$ 
(denoted $\cl U\stackrel{\rm loc}{\to}\cl V$)
if $\Gamma$ is locally tracial and (\ref{eq_projde}) holds.
\end{itemize}
\end{definition}

Given operator anti-systems $\cl S\subseteq M_X$ and $\cl T\subseteq M_A$,
Stahlke \cite{stahlke} defines a 
\emph{non-commutative graph homomorphism} from $\cl S$ to $\cl T$ to be a 
quantum channel $\Phi : M_X\to M_A$ with family $\{M_i\}_{i=1}^m$ of Kraus operators, such that 
$M_i \cl S M_j^*\subseteq \cl T$, $i,j = 1,\dots,m$; if such $\Phi$ exists, he writes 
$\cl S\to\cl T$. 
The appropriate version of this notion for twisted operator anti-systems 
-- directly modelled on Stahlke's definition -- is as follows.
For $T\in M_Z$, we write $\overline{T} = T^{* {\rm t}}$ for the conjugated matrix of $T$. 

\begin{definition}\label{d_ssss}
Let $X$ and $A$ be finite sets,
and $\cl S\subseteq \cl L\left((\bb{C}^{X})^{\dd},\bb{C}^X\right)$ and $\cl T\subseteq \cl L\left((\bb{C}^{A})^{\dd},\bb{C}^A\right)$ be 
twisted operator anti-systems. 
A \emph{homomorphism} from $\cl S$ into $\cl T$ is a quantum channel 
$$\Phi : M_X\to M_A, \ \ \Phi(T) = \sum_{i=1}^m M_i T M_i^*,$$
such that 
$$\overline{M}_j \cl S M_i^{\rm d}\subseteq \cl T, \ \ \ i,j = 1,\dots,m.$$
\end{definition}

If $\cl S$ and $\cl T$ are twisted operator anti-systems, we write $\cl S\to\cl T$ as in \cite{stahlke} 
to denote the existence of a homomorphism from $\cl S$ to $\cl T$.

\begin{proposition}\label{p_losame}
Let $X$ and $A$ be finite sets and $\cl U\subseteq \bb{C}^X\otimes \bb{C}^X$, 
$\cl V\subseteq \bb{C}^A\otimes \bb{C}^A$ be symmetric skew spaces. 
Then $\cl U\stackrel{\rm loc}{\to}\cl V$ if and only if $\cl S_{\cl U}\to \cl S_{\cl V}$.
\end{proposition}

\begin{proof}
Suppose that $\cl U\stackrel{\rm loc}{\to}\cl V$ and let $\Gamma$ be a locally tracial QNS correlation for which 
(\ref{eq_projde}) holds. 
By Theorem \ref{th_belo}, there exist quantum channels $\Phi_j : M_X\to M_A$, $j = 1,\dots,k$, 
such that $\Gamma = \sum_{j=1}^k \lambda_j \Phi_j\otimes \Phi_j^{\sharp}$ as a convex combination.
We have 
$$\sum_{j=1}^k \lambda_j \left\langle \left(\Phi_j\otimes \Phi_j^{\sharp}\right) \left(P_{\cl U}\right),P_{\cl V_{\perp}}\right\rangle 
= \left\langle \Gamma\left(P_{\cl U}\right),P_{\cl V_{\perp}}\right\rangle = 0;$$
since each of the terms in the sum on the left hand side is non-negative, selecting $j$ with $\lambda_j > 0$ and 
setting $\Phi = \Phi_j$, we have 
\begin{equation}\label{eq_sin}
\left\langle \left(\Phi\otimes \Phi^{\sharp}\right) \left(P_{\cl U}\right),P_{\cl V_{\perp}}\right\rangle = 0.
\end{equation}

Let $\Phi(\omega) = \sum_{i=1}^m M_i \omega M_i^*$, $\omega\in M_X$, be a Kraus representation of $\Phi$. 
For $\omega\in M_X$, we have
$$\Phi^{\sharp}(\omega) = \sum_{i=1}^m \left(M_i \omega^{\rm t} M_i^*\right)^{\rm t} 
= 
\sum_{i=1}^m (M_i^*)^{\rm t} \omega M_i^{\rm t}
= 
\sum_{i=1}^m \overline{M}_i \omega \overline{M}_i^*.$$
It follows that
\begin{equation}\label{eq_Mi}
\left(\Phi\otimes \Phi^{\sharp}\right)(\rho) 
= 
\sum_{i,j = 1}^m (M_i \otimes \overline{M}_j) \rho (M_i \otimes \overline{M}_j)^*,
\ \ \ \rho \in M_{XX}.
\end{equation}

Let $\xi\in \cl U$ and $\eta\in \cl V^{\perp}$ be unit vectors; then $\xi\xi^*\leq P_{\cl U}$. 
In addition, $\eta^{\rm d} = P_{\cl V_{\perp}}\eta^{\rm d}$ and hence 
$(\eta\eta^*)^{\rm d}=\eta^{\rm d}(\eta^{\rm d})^* \leq P_{\cl V_{\perp}}$; 
thus, (\ref{eq_sin}) implies
$$\left\langle \left(\Phi\otimes \Phi^{\sharp}\right) \left(\xi\xi^*\right),(\eta\eta^*)^{\dd}\right\rangle = 0.$$
By (\ref{eq_Mi}) and positivity, 
$$\left\langle  (M_i \otimes \overline{M}_j)(\xi\xi^*)(M_i \otimes \overline{M}_j)^*,(\eta\eta^*)^{\dd}\right\rangle = 0, \ \ \ i,j = 1,\dots,m,$$
which, by (\ref{eq_dtom}), means that 
$$\left\langle (M_i \otimes \overline{M}_j)\xi,\eta\right\rangle = 0, \ \ \ i,j = 1,\dots,m.$$
Thus, $(M_i \otimes \overline{M}_j)\xi \in \cl V$ for every $\xi\in \cl U$ and, by (\ref{eq_MNtheta}),
$$\overline{M}_j \theta(\xi) M_i^{\rm d} = \theta((M_i \otimes \overline{M}_j)\xi) \in \cl S_{\cl V}, \ \ \ \xi\in \cl U,$$
that is, $\cl S_{\cl U}\to \cl S_{\cl V}$. 

Conversely, suppose that $\Phi : M_X\to M_A$ is a quantum channel with a family of Kraus operators
$(M_i)_{i=1}^m\subseteq \cl L(\bb{C}^X,\bb{C}^A)$ such that 
$\overline{M}_j \cl S_{\cl U} M_i^{\rm d}\subseteq \cl S_{\cl V}$, $i,j = 1,\dots,m$. 
The previous paragraphs show that 
$$\left\langle(\Phi\otimes\Phi^{\sharp})(\xi\xi^*),(\eta\eta^*)^{\dd}\right\rangle =\Tr((\Phi\otimes\Phi^{\sharp})(\xi\xi^*)(\eta\eta^*))= 0$$
for all unit vectors $\xi\in \cl U$,  $\eta\in \cl V^{\perp}.$
It follows that
$$(\Phi\otimes\Phi^{\sharp})(\xi\xi^*) = (\eta\eta^*)^{\perp} (\Phi\otimes\Phi^{\sharp})(\xi\xi^*) (\eta\eta^*)^{\perp},$$
for all unit vectors $\eta\in \cl V^{\perp}$. Taking infimum over all such $\eta$, we obtain 
$$(\Phi\otimes\Phi^{\sharp})(\xi\xi^*) = P_{\cl V} (\Phi\otimes\Phi^{\sharp})(\xi\xi^*)P_{\cl V},$$ 
for all unit vectors $\xi\in \cl U$.
Thus, by (\ref{eq_dtom}) and (\ref{eq_pperpan}), 
$$\Tr\left(\left(\Phi\otimes\Phi^{\sharp}\right)(\xi\xi^*) P_{\cl V}^{\perp}\right)
= \left\langle (\Phi\otimes\Phi^{\sharp})(\xi\xi^*),P_{\cl V_\perp}\right\rangle = 0,$$
for all unit vectors $\xi\in \cl U$.
Writing $P_{\cl U} = \sum_{i=1}^l \xi_i\xi_i^*$, where $(\xi_i)_{i=1}^l$ is an orthonormal basis of 
$\cl U$, we obtain
$\left\langle (\Phi\otimes\Phi^{\sharp})(P_{\cl U}),P_{\cl V_{\perp}}\right\rangle = 0$.
\end{proof}

For graphs $G$ and $H$, write $G\to H$ if there exists a homomorphism from $G$ to $H$. 
The next corollary justifies viewing the symmetric skew spaces as non-commutative graphs.

\begin{corollary}\label{c_clg}
Let $G$ and $H$ be graphs. 
We have that $G\to H$ if and only if $\cl U_G \stackrel{{\rm loc}}{\to} \cl U_H$. 
\end{corollary}

\begin{proof}
Write $X$ and $A$ for the vertex sets of $G$ and $H$, respectively. 
Assume that $G\to H$. 
By \cite{stahlke},  $\cl S_G^0\to \cl S_H^0$. 
Write $\{M_i\}_{i=1}^m$ for the set of Kraus operators such that $M_i \cl S_G^0 M_j^*\subseteq \cl S_H^0$,
$i,j = 1,\dots,m$.  
Let $J_X : \bb{C}^X\to \bb{C}^X$ be the map given by $J_X(\eta) = \bar\eta$. 
Then $\theta(e_x\otimes e_y) = J_X\circ e_ye_x^*\circ\dd^{-1}$, $x,y\in X$. 
Therefore,
$$(J_A\circ M_i\circ J_X)(J_X\circ \cl S_G^0\circ{\dd}^{-1})(\dd\circ M_j^*\circ{\dd}^{-1})\subseteq J_A \circ {\cl S}_H^0\circ {\dd}^{-1},$$
implying
$\overline{M}_i\cl S_{\cl U_G}M_j^{\dd}\subseteq \cl S_{\cl U_H}$;
by Proposition \ref{p_losame}, $\cl U_G \stackrel{{\rm loc}}{\to} \cl U_H$. 
The converse follows after reversing the arguments. 
\end{proof}


\subsection{General quantum non-local games}\label{ss_gnlg}

We write $\cl P_{\cl M}$ for the projection lattice of a von Neumann algebra $\cl M$, 
and denote as usual by $\vee$ (resp. $\wedge$) the join 
(resp. the wedge) operation in $\cl P_{\cl M}$; thus, for $P_1,P_2\in \cl P_{\cl M}$, the projection $P_1\vee P_1$ 
(resp. $P_1\wedge P_2$) has range the closed span (resp. the intersection) of the ranges of $P_1$ and $P_2$. 
If $\cl M$ and $\cl N$ are von Neumann algebras, a map $\nph : \cl P_{\cl M}\to \cl P_{\cl N}$ is called 
\emph{join continuous} if 
$\nph\left(\vee_{i\in\bb{I}} P_i\right) = \vee_{i\in\bb{I}} \nph(P_i)$ 
for any family $\{P_i\}_{i\in \bb{I}}\subseteq \cl P_{\cl M}$.
Note that if $\cl M$ is finite dimensional, then join continuity is equivalent to the preservation of finite joins. 

Let $H$ be a Hilbert space and $P$ be an orthogonal projection on $H$ with range $\cl U$. 
As in Subsection \ref{ss_qhncg}, we denote by $\cl U_{\perp}$ the annihilator of $\cl U$ in the space $H^{\rm d}$, and 
by $P_{\perp}$ -- the orthogonal projection on $H^{\rm d}$ with range $\cl U_{\perp}$.

\begin{definition}\label{d_qnlg}
Let $X$, $Y$, $A$ and $B$ be finite sets. 
\begin{itemize}
\item[(i)] 
A map $\nph : \cl P_{M_{XY}}\to \cl P_{M_{AB}}$
(resp. $\nph : \cl P_{\cl D_{XY}}\to \cl P_{M_{AB}}$, 
$\nph : \cl P_{\cl D_{XY}}\to \cl P_{\cl D_{AB}}$) is called 
a \emph{quantum non-local game} (resp. a \emph{classical-to-quantum non-local game}, 
a \emph{classical non-local game}) if $\nph$ is join continuous and $\nph(0) = 0$.
We say that such $\nph$ is a game \emph{from $XY$ to $AB$}.

\item[(ii)] 
A QNS (resp. CQNS, NS) correlation $\Lambda$
is called a perfect strategy for the quantum (resp. classical-to-quantum, classical) 
non-local game $\nph$ if 
\begin{equation}\label{eq_erdos}
\left\langle \Lambda(P),\nph(P)_{\perp}\right\rangle = 0, 
\ \ \ P\in \cl P_{M_{XY}} \ (\mbox{resp. } P\in \cl P_{\cl D_{XY}}).
\end{equation}
\end{itemize}
\end{definition}

\begin{remark}\label{r_erdos}
{\rm 
{\bf (i)}
Join continuous zero-preserving maps $\nph : \cl P_{\cl B(H)} \to \cl P_{\cl B(K)}$, 
where $H$ and $K$ are Hilbert spaces, 
were first considered by J. A. Erdos in \cite{erdos}. 
They are equivalent to \emph{bilattices} introduced in \cite{st1} -- that is,
subsets $\frak{B}\subseteq \cl P_{\cl B(H)} \times \cl P_{\cl B(K)}$ such that 
$(P,0), (0,Q)\in \frak{B}$ for all $P\in \cl P_{\cl B(H)}$, $Q\in \cl P_{\cl B(K)}$, 
and $(P_1,Q_1), (P_2,Q_2)\in \frak{B}$ $\Rightarrow$
$(P_1\vee P_2,Q_1\wedge Q_2)\in \frak{B}$ and $(P_1\wedge P_2,Q_1\vee Q_2)\in \frak{B}$. 
Thus, quantum non-local games (resp. classical-to-quantum non-local games, classical non-local games) 
can be alternatively defined as bilattices; we have chosen to use maps instead because they 
are more convenient to work with when compositions are considered (see Definition \ref{d_coqnlg}).

Conditions (\ref{eq_erdos}) are reminiscent of J. A. Erdos' characterisation \cite{erdos} of 
reflexive spaces of operators, introduced by L. N. Loginov and V. S. Shulman in \cite{ls}. 
As shown in \cite{erdos}, a subspace $\cl S\subseteq \cl B(H,K)$ ($H$ and $K$ being Hilbert spaces)
is reflexive in the sense of \cite{ls} if and only if there exists a join continuous zero-preserving map 
$\nph : \cl P_{\cl B(H)} \to \cl P_{\cl B(K)}$ such that $\cl S$ coincides with the the space 
$${\rm Op}(\nph) = \left\{T\in \cl B(H,K) : \nph(P)^{\perp} T P = 0, \ \mbox{ for all } P\in \cl P_{\cl B(H)}\right\}.$$

\smallskip

{\bf (ii) } The quantum (resp. classical-to-quantum, classical) non-local game 
$\nph$ with $\nph(P) = I_{AB}$ for every non-zero $P \in \cl P_{M_{XY}}$ (resp. $P \in \cl P_{\cl D_{XY}}$) 
will be referred to as the \emph{empty game}. 
It is clear that the set of perfect strategies for the empty game coincides with the class of all 
no-signalling correlations. 

\smallskip

{\bf (iii) } Let $G$ be a graph with vertex set $X$ and $A$ be a finite set. 
The quantum graph colouring game considered in Subsection \ref{ss_gcg} is the 
classical-to-quantum non-local game $\nph : \cl P_{\cl D_{XX}}\to \cl P_{M_{AA}}$, given by 
$$
\nph(e_xe_x^* \otimes e_y e_y^*) = 
\begin{cases}
\frac{1}{|A|} \Omega_A^{\perp} & \text{if } x\sim y\\
I & \text{otherwise.}
\end{cases}
$$
Similarly, letting $\cl U\subseteq \bb{C}^X\otimes \bb{C}^X$ and 
$\cl V\subseteq \bb{C}^A\otimes \bb{C}^A$ be symmetric skew spaces, 
we define the homomorphism game $\cl U\to \cl V$ to be the 
quantum non-local game $\psi$, given by 
$$
\psi(P) = 
\begin{cases}
P_{\cl V} & \text{if } 0 \neq P \leq P_{\cl U}\\
0 & \text{if } P = 0\\
I & \text{otherwise.}
\end{cases}
$$
For ${\rm x}\in \{{\rm loc}, {\rm q}, {\rm qc}\}$, we have that $\cl U\stackrel{{\rm x}}{\to}\cl V$ if and only if 
the game $\cl U\to \cl V$ has a perfect strategy of class $\cl Q_{\rm x}$. 
}
\end{remark}

Let $(X,Y,A,B,\lambda)$ be a non-local game. 
For a subset $\alpha\subseteq X\times Y$, let $P_{\alpha} \in \cl P_{\cl D_{XY}}$ be the projection 
with range ${\rm span}\{e_x\otimes e_y : (x,y)\in \alpha\}$.
For $(x,y)\in X\times Y$, let 
$$\beta_{x,y}(\lambda) = \{(a,b) \in A\times B : \lambda(x,y,a,b) = 1\}.$$
We associate with $\lambda$ the (unique) classical non-local game 
$\nph_{\lambda} : \cl P_{\cl D_{XY}} \to \cl P_{\cl D_{AB}}$ determined by the requirement 
$$\nph\mbox{}_{\lambda}\left(P_{\{(x,y)\}}\right) = P_{\beta_{x,y}(\lambda)}, \ \ \ (x,y)\in X\times Y.$$ 

\begin{proposition}\label{p_pseqcn}
An NS correlation $p$ is a perfect strategy for the non-local game (with rule function) $\lambda$
if and only if $\cl N_p$ is a perfect strategy for $\nph_{\lambda}$. 
\end{proposition}

\begin{proof}
Note that, if $(x,y)\in X\times Y$ then 
$\left(P_{\beta_{x,y}(\lambda)}\right)_{\perp}$ has range 
${\rm span}\{e_a e_a^*\otimes e_b e_b^* : \lambda(x,y,a,b) = 0\}$. 
As in Proposition \ref{p_clgh}, it is thus easily seen that $p$ is a perfect strategy for $\lambda$ 
if and only if 
$$\left\langle \cl N_p\left(P_{\{(x,y)\}}\right),\left(P_{\beta_{x,y}(\lambda)}\right)_{\perp}\right\rangle = 0, \ \ \ (x,y)\in X\times Y.$$

Assume that $p$ is a perfect strategy for $\lambda$. 
For a projection $P\in \cl D_{XY}$, write $P = \vee\{P_{\{(x,y)\}} : P(e_x\otimes e_y) = e_x\otimes e_y\}$; 
then
$$\nph\mbox{}_{\lambda}(P) = \vee\{P_{\beta_{x,y}(\lambda)} : P(e_x\otimes e_y) = e_x\otimes e_y\}.$$
Thus,
$\left\langle \cl N_p(P_{\{(x,y)\}}),\nph_\lambda(P)_{\perp}\right\rangle = 0$ for all pairs 
$(x,y)$ with $P(e_x\otimes e_y) = e_x\otimes e_y$.
Taking the join over all those $(x,y)$, we conclude that 
$\left\langle \cl N_p(P),\nph_\lambda(P)_{\perp}\right\rangle = 0$.
The converse is direct from the first paragraph. 
\end{proof}

\begin{definition}\label{d_coqnlg}
Let $X$, $Y$, $A$, $B$, $Z$ and $W$ be finite sets and 
$\nph_1$ (resp. $\nph_2$) be a game from $XY$ to $AB$ (resp. from $AB$ to $ZW$). 
The \emph{composition} of $\nph_1$ and $\nph_2$ is the game $\nph_2\circ \nph_1$ from $XY$ to $ZW$.
\end{definition}

It is clear that $\nph_2\circ \nph_1$ is well-defined in all cases except when $\nph_1$ is a quantum game, while 
$\nph_2$ is a classical-to-quantum game.

\begin{lemma}\label{l_cGPOVM}
Let $X$, $A$ and $Z$ be finite sets, $H$ and $K$ be Hilbert spaces and $E\in M_X\otimes M_A\otimes \cl B(H)$ and 
$F\in M_A\otimes M_Z\otimes \cl B(K)$ be stochastic operator matrices.
Set 
$$G_{x,x',z,z'} = \sum_{a,a'\in A} F_{a,a',z,z'}  \otimes E_{x,x',a,a'} , \ \ \ x,x'\in X, z,z'\in Z.$$
Then $G = (G_{x,x',z,z'})_{x,x',z,z'}$ is a stochastic operator matrix in $M_X\otimes M_Z\otimes \cl B(K\otimes H)$.
\end{lemma}

\begin{proof}
Let $V = (V_{a,x})_{a,x}$ (resp. $W = (W_{z,a})_{z,a}$)
be an isometry from $H^X$ (resp. $K^A$) to $\tilde{H}^A$ (resp. $\tilde{K}^Z$) 
for some Hilbert space $\tilde{H}$ (resp. $\tilde{K}$),
such that 
$$E_{x,x',a,a'} = V_{a,x}^* V_{a',x'} \ \mbox{ and } \ F_{a,a',z,z'} = W_{z,a}^* W_{z',a'}$$
for all $x,x'\in X$, $a,a'\in A$ and $z,z'\in Z$. 
Set 
$$U_{z,x} = \sum_{a\in A} W_{z,a}\otimes V_{a,x}, \ \ \ x\in X, z\in Z.$$
For $x,x'\in X$, we have 
\begin{eqnarray*}
\sum_{z\in Z} U_{z,x}^* U_{z,x'} 
& = & 
\sum_{z\in Z} \left(\sum_{a\in A} W_{z,a}^*\otimes V_{a,x}^*\right) \left(\sum_{a'\in A} W_{z,a'}\otimes V_{a',x'}\right)\\
& = & 
\sum_{z\in Z} \sum_{a,a'\in A} W_{z,a}^* W_{z,a'} \otimes V_{a,x}^*V_{a',x'}\\
& = & 
\sum_{a,a'\in A} \left(\sum_{z\in Z} W_{z,a}^* W_{z,a'}\right) \otimes V_{a,x}^*V_{a',x'}\\
& = & 
\sum_{a,a'\in A} \delta_{a,a'} I_K \otimes V_{a,x}^*V_{a',x'} 
= 
\sum_{a\in A} I_K \otimes V_{a,x}^*V_{a,x'}\\ 
& = & 
\delta_{x,x'} I_K \otimes I_H;
\end{eqnarray*}
thus, $(U_{z,x})_{z,x}$ is an isometry from $(K\otimes H)^X$ into $(\tilde{K}\otimes \tilde{H})^Z$. 
In addition, for $x,x'\in X$ and $z,z'\in Z$, we have 
\begin{eqnarray*}
U_{z,x}^* U_{z',x'}
& = & 
\left(\sum_{a\in A} W_{z,a}^*\otimes V_{a,x}^*\right) \left(\sum_{a'\in A} W_{z',a'}\otimes V_{a',x'}\right)\\
& = & 
\sum_{a,a'\in A} F_{a,a',z,z'} \otimes E_{x,x',a,a'} 
= 
G_{x,x',z,z'}.
\end{eqnarray*}
By Theorem \ref{p_coor}, $G$ is a stochastic operator matrcx acting on $K\otimes H$.
\end{proof}

We call the stochastic operator matrix $G$ from Lemma \ref{l_cGPOVM} the \emph{composition} of $F$ and $E$ 
and denote it by $F\circ E$.

\begin{theorem}\label{th_comg}
Let $\nph_1$ (resp. $\nph_2$) be a quantum game from $XY$ to $AB$ (resp. from $AB$ to $ZW$) and 
${\rm x}\in \{{\rm loc},{\rm q},{\rm qa},{\rm qc},{\rm ns}\}$. 
\begin{itemize}
\item[(i)]
If $\Gamma_i$ is a perfect strategy for $\nph_i$ from the class $\cl Q_{\rm x}$, $i = 1,2$,
then $\Gamma_2\circ \Gamma_1$ is a perfect strategy for $\nph_2\circ \nph_1$
from the class $\cl Q_{\rm x}$. 
\item[(ii)]
Asume that $X = Y$, $A = B$ and $Z = W$.
If $\Gamma_i$ is a perfect tracial (resp. quantum tracial, locally tracial) strategy for $\nph_i$, $i = 1,2$, 
then $\Gamma_2\circ \Gamma_1$ is a perfect tracial (resp. quantum tracial, locally tracial) strategy for $\nph_2\circ \nph_1$.
\end{itemize}
\end{theorem}

\begin{proof}
First note that if $\Gamma_i$ is a QNS correlation then so is $\Gamma_2\circ \Gamma_1$. 
Indeed, suppose that $\rho \in M_{XY}$ is such that $\Tr_X\rho = 0$. 
By Remark \ref{c_strpt}, $\Tr_A\Gamma_1(\rho) = 0$, and hence, again by Remark \ref{c_strpt}, 
$\Tr_Z(\Gamma_2(\Gamma_1(\rho)) = 0$.

Suppose that $\Gamma_i\in \cl Q_{\rm qc}$, $i = 1,2$. Let 
$(E^{(i)},F^{(i)})$ be a commuting pair of 
stochastic operator matrices acting on a Hilbert space $H_i$, and $\sigma_i$ be a normal state on $\cl B(H_i)$, 
such that $\Gamma_i = \Gamma_{E^{(i)}\cdot F^{(i)},\sigma_i}$, $i = 1,2$. 
Write $E^{(1)} = \left(E^{(1)}_{x,x',a,a'}\right)$, 
$F^{(1)} = \left(F^{(1)}_{y,y',b,b'}\right)$, 
$E^{(2)} = \left(E^{(2)}_{a,a',z,z'}\right)$ and 
$F^{(2)} = \left(F^{(2)}_{b,b',w,w'}\right)$. 
Set $H = H_2 \otimes H_1$, $\sigma = \sigma_2\otimes \sigma_1$, 
$E = E^{(2)}\circ E^{(1)}$ and $F = F^{(2)}\circ F^{(1)}$; 
note that, by Lemma \ref{l_cGPOVM}, $E$ and $F$ are 
stochastic operator matrices. 
It is straightforward that $(E,F)$ is a commuting pair. 
Write $E = (E_{x,x',z,z'})$ and $F = (F_{y,y',w,w'})$. 
Note that 
\begin{eqnarray*}
& & \sum_{a,a',b,b'} 
\left\langle E^{(1)}_{x,x',a,a'} F^{(1)}_{y,y',b,b'}, \sigma_1\right\rangle \left\langle E^{(2)}_{a,a',z,z'} F^{(2)}_{b,b',w,w'}, 
\sigma_2\right\rangle\\
& = & 
\sum_{a,a',b,b'} 
\left\langle \left(E^{(2)}_{a,a',z,z'} \otimes E^{(1)}_{x,x',a,a'}\right) 
\left(F^{(2)}_{b,b',w,w'} \otimes F^{(1)}_{y,y',b,b'}\right), \sigma_2\otimes \sigma_1\right\rangle\\
& = & 
\left\langle E_{x,x',z,z'} F_{y,y',w,w'}, \sigma\right\rangle,
\end{eqnarray*} 
and hence
\begin{eqnarray*}
& & \left(\Gamma_2\circ \Gamma_1\right)\left(e_{x}e_{x'}^*\otimes e_{y}e_{y'}^*\right)\\
& = & 
\sum_{a,a',b,b'} \left\langle E^{(1)}_{x,x',a,a'} F^{(1)}_{y,y',b,b'}, \sigma_1\right\rangle 
\Gamma_2\left(e_a e_{a'}^*\otimes e_b e_{b'}^*\right)\\
& = & 
\sum_{z,z',w,w'} 
\sum_{a,a',b,b'} 
\langle E^{(1)}_{x,x',a,a'} F^{(1)}_{y,y',b,b'}, \sigma_1\rangle \langle E^{(2)}_{a,a',z,z'} F^{(2)}_{b,b',w,w'}, \sigma_2\rangle
e_z e_{z'}^*\otimes e_w e_{w'}^*\\
& = & 
\sum_{z,z',w,w'}
\left\langle E_{x,x',z,z'} F_{y,y',w,w'}, \sigma\right\rangle e_z e_{z'}^*\otimes e_w e_{w'}^*;
\end{eqnarray*}
thus, $\Gamma_2\circ \Gamma_1 = \Gamma_{E\cdot F,\sigma}$. 

If $\Gamma_i\in \cl Q_{\rm q}$, $i = 1,2$, then the arguments in the previous paragraph 
-- replacing operator products by tensor products as necessary -- show that $\Gamma_2\circ \Gamma_1\in \cl Q_{\rm q}$. 
By the continuity of the composition, the assumptions $\Gamma_i\in \cl Q_{\rm qa}$, $i = 1,2$, imply that 
$\Gamma_2\circ \Gamma_1\in \cl Q_{\rm qa}$.
Finally, assume that $\Gamma_i\in \cl Q_{\rm loc}$, $i = 1,2$, and write 
$\Gamma_i = \sum_{k=1}^{m_i} \lambda_k^{(i)} \Phi_k^{(i)} \otimes \Psi_k^{(i)}$ as a convex combination, 
where $\Phi_k^{(i)} : M_X\to M_A$ and $\Psi_k^{(i)} : M_Y\to M_B$ are quantum channels, $i = 1,2$.
Then
$$\Gamma_2\circ \Gamma_1 = \sum_{k=1}^{m_1} \sum_{l=1}^{m_2}
\lambda_k^{(1)} \lambda_l^{(2)} \left(\Phi_l^{(2)}\circ \Phi_k^{(1)}\right) \otimes  \left(\Psi_l^{(2)}\circ \Psi_k^{(1)}\right)$$
as a convex combination, and hence $\Gamma_2\circ \Gamma_1\in \cl Q_{\rm loc}$.

Suppose that $\Gamma_i$ is a tracial QNS correlation; thus, there exist unital C*-algebras
$\cl A_1$ and $\cl A_2$, traces $\tau_1$ and $\tau_2$ on $\cl A_1$ and $\cl A_2$, respectively, 
and stochastic matrices 
$E^{(1)}\in M_X\otimes M_A\otimes \cl A_1$ and $E^{(2)}\in M_A\otimes M_Z\otimes \cl A_2$, 
such that
$\Gamma_i = \Gamma_{E^{(i)}, \tau_i}$, $i = 1,2$.
 The arguments given for (i) show that 
$$\Gamma_2\circ \Gamma_1 = \Gamma_{E^{(2)}\circ E^{(1)}, \tau_2\otimes \tau_1},$$ 
where $\tau_2\otimes \tau_1$ is the product trace on $\cl A_2\otimes_{\min}\cl A_1$;
$E^{(2)}\circ E^{(1)}$ is considered as a stochastic $\cl A_2\otimes_{\min}\cl A_1$-matrix
(note that we identify
$\left(E^{(2)}\circ E^{(1)}\right)^{\rm op}$ with $E^{(2){\rm op}}\circ E^{(1){\rm op}}$ in the natural way). 

It remains to show that if $\Gamma_i$ is a perfect strategy for $\nph_i$, $i = 1,2$, then 
$\Gamma_2\circ \Gamma_1$ is a perfect strategy for $\nph_2\circ \nph_1$. 
Let $P\in \cl P_{M_{XY}}$ and $\omega$ be a pure state with $\omega\leq P$. 
Then $\Gamma_1(\omega) = \nph_1(P)\Gamma_1(\omega)\nph_1(P)$ and hence
$\Gamma_1(\omega)\leq \nph_1(P)$. 
Similarly, for any pure state $\sigma$ with $\sigma\leq\nph_1(P)$ 
we have $\Gamma_2(\sigma)=\nph_2(\nph_1(P))\Gamma_2(\sigma)\nph_2(\nph_1(P))$, 
giving $\left\langle(\Gamma_2(\sigma),(\nph\mbox{}_2\circ \nph\mbox{}_1)(P)_{\perp} \right\rangle = 0$.
In particular, 
$$\left\langle(\Gamma_2\circ \Gamma_1)(\omega), (\nph\mbox{}_2\circ \nph\mbox{}_1)(P)_{\perp} \right\rangle = 0.$$
As in 
the proof of Proposition \ref{p_losame}, this yields 
$$\left\langle(\Gamma_2\circ \Gamma_1)(P), (\nph\mbox{}_2\circ \nph\mbox{}_1)(P)_{\perp} \right\rangle = 0,$$
establishing the claim. 
\end{proof}

Suppose that $p_1$ (resp. $p_2$) is an NS correlation from $XY$ to $AB$ (resp. from $AB$ to $ZW$). 
It is straightforward to verify that the correlation $p$ with $\cl N_p = \cl N_{p_2}\circ \cl N_{p_1}$ is given by 
$$p(z,w|x,y) = \sum_{a\in A}\sum_{b\in B} p_2(z,w|a,b) p_1(a,b|x,y);$$
we write $p = p_2\circ p_1$. Such compositions were first studied in \cite{po}. 
For a non-local game from $XY$ to $AB$ (resp. from $AB$ to $ZW$) with rule function $\lambda_1$ 
(resp. $\lambda_2$), let 
$\lambda_2 \circ \lambda_1 : X\times Y \times Z\times W\to \{0,1\}$
be given by 
$$(\lambda_2 \circ \lambda_1)(x,y,z,w) = 1 \Leftrightarrow \exists \ (a,b) \mbox{ s.t. }
\lambda_1(x,y,a,b) = \lambda_2(a,b,z,w) = 1.$$
Combining Theorem \ref{th_comg} with classical reduction and Proposition \ref{p_pseqcn}, we obtain the following
perfect strategy version of \cite[Proposition 3.5]{po}, which simultaneously extends 
the graph homomorphism transitivity results contained in \cite[Theorem 3.7]{po}.

\begin{corollary}\label{c_pog}
Let $\lambda_1$ (resp. $\lambda_2$) be the rule functions of non-local games 
from $XY$ to $AB$ (resp. from $AB$ to $ZW$) and 
${\rm x}\in \{{\rm loc},{\rm q},{\rm qa},{\rm qc},{\rm ns}\}$. 
If $p_i$ is a perfect strategy for $\lambda_i$ from the class $\cl C_{\rm x}$, $i = 1,2$,
then $p_2\circ p_1$ is a perfect strategy for $\lambda_2\circ \lambda_1$
from the class $\cl C_{\rm x}$. 
\end{corollary}

Combining Theorem \ref{th_comg} with Remark \ref{r_erdos} (iii) yields the following transitivity result;
in view of Proposition \ref{p_losame}, it extends \cite[Proposition 9]{stahlke}.

\begin{corollary}\label{c_sss}
Let $X$, $A$ and $Z$ be finite sets, 
$\cl U\subseteq \bb{C}^X\otimes \bb{C}^X$, $\cl V\subseteq \bb{C}^A\otimes \bb{C}^A$
and $\cl W\subseteq \bb{C}^Z\otimes \bb{C}^Z$ be symmetric skew spaces, and 
${\rm x}\in \{{\rm loc},{\rm q},{\rm qc}\}$. 
If $\cl U\stackrel{{\rm x}}{\to}\cl V$ and $\cl V \stackrel{{\rm x}}{\to}\cl W$
then $\cl U\stackrel{{\rm x}}{\to}\cl W$.
\end{corollary}

\smallskip

\noindent 
{\bf Acknowledgement. } 
It is our pleasure to thank Michael Brannan, Li Gao, Marius Junge, 
Dan Stahlke and Andreas Winter 
for fruitful discussions on the topic of this paper.

\smallskip

\noindent 
{\bf Note. } 
After the paper was completed, we became aware of the work 
\cite{bgh}, in which the authors
define quantum-to-classical no-signalling correlations and study a version of the homomorphism game 
from a non-commutative to a classical graph. Although there are similarities between our approaches, 
there is no duplication of results in the current paper with those in \cite{bgh}.


\end{document}